\documentclass[letterpaper]{article} 
\usepackage{aaai23}  
\usepackage{times}  
\usepackage{helvet}  
\usepackage{courier}  
\usepackage[hyphens]{url}  
\usepackage{graphicx} 
\usepackage{natbib}  
\usepackage{caption} 
\usepackage{algorithm}
\usepackage{algorithmic}

\usepackage{newfloat}
\usepackage{listings}
\DeclareCaptionStyle{ruled}{labelfont=normalfont,labelsep=colon,strut=off} 
\floatstyle{ruled}
\newfloat{listing}{tb}{lst}{}
\floatname{listing}{Listing}
\nocopyright 

\usepackage{amsmath}
\usepackage{mathtools}
\usepackage{mathrsfs}
\usepackage{amssymb}
\usepackage{booktabs}
\usepackage{tablefootnote}
\usepackage{pifont}
\usepackage[shortlabels]{enumitem}
\setlist[enumerate]{noitemsep, topsep=0pt}
\usepackage{adjustbox}
\usepackage{xspace}

\newcommand{\R}{\mathbb{R}}

\newcommand{\Exp}{\mathbb{E}}
\newcommand{\basis}{\mathtt{\mathbf{e}}}

\newcommand{\blx}{\mathtt{\mathbf{x}}}
\newcommand{\bly}{\mathtt{\mathbf{y}}}
\newcommand{\blz}{\mathtt{\mathbf{z}}}

\newcommand{\x}{\mathtt{\mathbf{X}}}
\renewcommand{\d}{\mathtt{\mathbf{D}}}
\newcommand{\y}{\mathtt{\mathbf{Y}}}
\newcommand{\z}{\mathtt{\mathbf{Z}}}
\newcommand{\w}{\mathtt{\boldsymbol{\mathcal{W}}}}
\newcommand{\B}{\mathtt{\mathbf{B}}}

\newcommand{\identity}{\mathtt{\mathbf{I}}}
\newcommand{\vecx}{\mathtt{\mathbf{x}}}
\newcommand{\avg}{\frac{1}{N}\basis\basis^\top}

\newcommand{\W}{\mathtt{\mathbf{W}}}
\newcommand{\cmark}{\ding{51}}
\newcommand{\xmark}{\ding{55}}
\newcommand{\graph}{\mathcal{G}}
\newcommand{\edges}{\mathcal{E}}
\newcommand{\vertices}{\mathcal{V}}

\newcommand{\xik}{\mathtt{\mathbf{x}}_i^{(k)}}
\newcommand{\xiknew}{\mathtt{\mathbf{x}}_i^{(k+1)}}
\newcommand{\barxknew}{\bar{\mathtt{\mathbf{x}}}^{(k+1)}}
\newcommand{\barxk}{\bar{\mathtt{\mathbf{x}}}^{(k)}}

\newcommand{\mymatrix}[1]{\begin{bmatrix}#1\end{bmatrix}}
\newcommand{\norm}[1]{\left\Vert #1 \right\Vert}
\newcommand{\abs}[1]{\left\lvert #1\right\rvert}
\newcommand{\ip}[1]{\left\langle #1\right\rangle}

\DeclareMathOperator*{\argmin}{argmin}

\newcommand{\bigO}[1]{\mathcal{O}\left(#1\right)}
\newcommand{\fronorm}[1]{\left\Vert#1\right\Vert_F^2}
\renewcommand{\vec}[1]{\mathtt{\mathbf{#1}}}
\newcommand{\stack}[2]{\stackrel{\mathclap{#1}}{#2}}
\newcommand{\prob}[1]{\hbox{Prob}\big(#1\big)}
\newcommand{\defined}{\triangleq}

\newcommand{\ourmethod}{DEEPSTORM\xspace} 

\newtheorem{corollary}{Corollary}
\newtheorem{assumption}{Assumption}
\newtheorem{remark}{Remark}
\newtheorem{definition}{Definition}
\newtheorem{lemma}{Lemma}
\newtheorem{theorem}{Theorem}

\title{Proximal Stochastic Recursive Momentum Methods for Nonconvex Composite Decentralized Optimization}
\author {
    Gabriel Mancino-Ball\textsuperscript{\rm 1},
    Shengnan Miao\textsuperscript{\rm 1},
    Yangyang Xu\textsuperscript{\rm 1},
    Jie Chen\textsuperscript{\rm 2}
}
\affiliations {
    \textsuperscript{\rm 1}Department of Mathematical Sciences, Rensselaer Polytechnic Institute, Troy, NY 12180 \\
    \textsuperscript{\rm 2}MIT IBM-Watson AI Lab, IBM Research, Cambridge, MA 02142\\
    mancig@rpi.edu, snmiao236@gmail.com, xuy21@rpi.edu, chenjie@us.ibm.com
}
\begin{document}

\maketitle

\begin{abstract}
  Consider a network of $N$ decentralized computing agents collaboratively solving a nonconvex stochastic composite problem. In this work, we propose a single-loop algorithm, called DEEPSTORM, that achieves optimal sample complexity for this setting. Unlike double-loop algorithms that require a large batch size to compute the (stochastic) gradient once in a while, DEEPSTORM uses a small batch size, creating advantages in occasions such as streaming data and online learning. This is the first method achieving optimal sample complexity for decentralized nonconvex stochastic composite problems, requiring $\mathcal{O}(1)$ batch size. We conduct convergence analysis for DEEPSTORM with both constant and diminishing step sizes. Additionally, under proper initialization and a small enough desired solution error, we show that DEEPSTORM with a constant step size achieves a network-independent sample complexity, with an additional linear speed-up with respect to $N$ over centralized methods. All codes are made available at~\url{https://github.com/gmancino/DEEPSTORM}.
\end{abstract}

\section{Introduction}\label{sec:intro}

Recent years have seen an increase in designing efficient algorithms for solving large-scale machine learning problems, over a network of $N$ computing agents connected by a communication graph $\graph=(\vertices,\edges)$. Agents collaboratively solve the following composite problem:
\begin{equation}\label{original_problem}
  \min_{\blx}\frac{1}{N}\sum_{i=1}^N\left\{\phi_i(\blx)\triangleq f_i(\blx)+r(\blx)\right\},
\end{equation}
where the decision variable $\blx\in\R^{1\times p}$ is treated as a row vector; $f_i$ is a smooth, possibly nonconvex function known only to agent $i$; and $r$ is a convex, possibly non-smooth regularizer common to all agents. Agents $i$ and $j$ can communicate only if $(i,j)\in\edges$. Many real-world applications in machine learning~\cite{vogels21,ying21,yuan21,chamideh21} and reinforcement learning~\cite{zhang18,chao19} fit the form of~\eqref{original_problem}. Such scenarios differ from the centralized setting~\cite{mcmahan17,canh20}, where the agents are assumed to be able to communicate with one another globally via either a parameter server or a collective communication protocol. This setting arises naturally when data is distributed over a large geographic region or when a centralized communication structure is too costly~\cite{xin21hsgd}.

Utilizing the communication topology induced by $\graph$, we reformulate~\eqref{original_problem} into the following equivalent \textit{decentralized consensus optimization problem}:
\begin{equation}\label{dco_problem}
  \min_{\blx_1,\dots,\blx_N}\frac{1}{N}\sum_{i=1}^N\phi_i(\blx_i),\text{ s.t. }\blx_i=\blx_j,\enskip\forall(i,j)\in\edges.
\end{equation}
Problem~\eqref{dco_problem} allows for agents to maintain and update a local copy of the decision variable by locally computing gradients and performing neighbor communications.

The existence of a non-smooth regularizer $r$ renders many decentralized optimization methods for a smooth objective inappropriate. We assume that $r$ admits an easily computable (e.g. closed form) proximal mapping. Moreover, we are interested in the case where each local function $f_i$ takes the following expectation form:
\begin{equation}\label{f_i}
  f_i(\blx)\triangleq\Exp_{\xi\sim\mathcal{D}_i}\left[f_i(\blx;\xi)\right],
\end{equation}
with a slight abuse of notation for ease of exposition. In such a case, agents locally compute stochastic gradients of $f_i$. We adapt ideas from recent advances of stochastic optimization to the decentralized setting, by combining variance reduction techniques~\cite{johnson13,nguyen17,allenzhu18,wang19spider,cutkosky19,tran22} with gradient tracking~\cite{lorenzo16,nedic17,songtao19,zhang20gradtrack,koloskova21}, to produce an algorithmic framework that achieves the optimal sample complexity bounds established in~\cite{arjevani19} for nonconvex stochastic methods.

Our framework, coined \ourmethod, is a single-loop algorithm with an attractive property that, besides the initial iteration, each agent only needs $m=\bigO{1}$ stochastic samples to compute a gradient estimate. Further, when a diminishing step size is used, even the first iteration does not need a large batch, at the expense of an additional logarithmic factor in the sample complexity result. Intuitively, \ourmethod utilizes a momentum based variance reduction technique~\cite{cutkosky19,xu20,levy21,tran22} to guarantee convergence under a small batch size. The use of momentum simultaneously accelerates the computation and communication complexities over non-momentum based methods in the small batch setting; see Table~\ref{table:related_works} for a comparison. The recent ProxGT-SR-O/E~\cite{xin21} method can also achieve optimal sample complexity for solving~\eqref{dco_problem}, but at the expense of performing a double-loop which requires a large (stochastic) gradient computation every time the inner loop is completed. In scenarios where the batch size is uncontrollable, such as streaming or online learning, \ourmethod is advantageous.

When discussing sample complexity, it is paramount to specify the impact of the communication graph $\graph$. With a constant step size, we show that under a sufficient amount of initial, or \emph{transient}, iterations and proper initialization, \ourmethod behaves similarly to its centralized counterparts~\cite{cutkosky19,levy21,tran22}, while enjoying a linear speed-up with respect to $N$.

\textbf{We summarize the contributions of this work below:}

\begin{itemize}[leftmargin=*]

\item We propose a novel decentralized framework, \ourmethod, for nonconvex stochastic composite optimization problems. We show that \ourmethod achieves the optimal sample complexity with respect to solution accuracy, where each agent needs only $\bigO{1}$ samples to compute a local stochastic gradient. To the best of our knowledge, this is the first decentralized method that achieves optimal sample complexity for solving \emph{stochastic composite} problems by using only small batches.

\item Additionally, we establish convergence guarantees of \ourmethod with both constant and diminishing step sizes. When a constant step size is used, we show that under sufficiently many transient iterations and proper initialization, \ourmethod achieves a linear speed-up with respect to $N$, signifying an advantage over analogous centralized variance reduction methods~\cite{cutkosky19,levy21,tran22}.
  
\end{itemize}

\begin{table*}[t]
  \centering
 \begin{adjustbox}{max width=\textwidth}
	\renewcommand{\arraystretch}{1.6}
	\begin{tabular}{lccc}
	  \toprule
	  Method    &    $r\not\equiv0$    &    Batch size    &     Sample complexity (per agent)      \\
	  \midrule
	  D-PSGD~\cite{lian17}     &     \xmark     &     $\bigO{1}$     &     $\bigO{\max\left\{\frac{1}{N\varepsilon^2},\frac{N^2}{(1-\rho)^2\varepsilon}\right\}}$    \\
	  DSGT~\cite{xin21_dsgt}     &     \xmark     &     $\bigO{1}$     &     $\bigO{\max\left\{\frac{1}{N\varepsilon^2},\frac{\rho N}{(1-\rho)^3\varepsilon}\right\}}$     \\
	  D-GET~\cite{sun20}     &     \xmark     &     $\bigO{\frac{1}{\varepsilon}}$ or $\bigO{\frac{1}{\varepsilon^{0.5}}}$     &     $\bigO{\frac{1}{(1-\rho)^a\varepsilon^{1.5}}}$    \\
	  GT-HSGD~\cite{xin21hsgd}     &     \xmark     &     $\bigO{\frac{1}{\varepsilon^{0.5}}}$ then $\bigO{1}$     &     $\bigO{\max\left\{\frac{1}{N\varepsilon^{1.5}},\frac{\rho^4}{N(1-\rho)^3\varepsilon},\frac{\rho^{1.5}N^{0.5}}{(1-\rho)^{2.25}\varepsilon^{0.75}}\right\}}$    \\
	  SPPDM~\cite{wang21}     &    \cmark     &     $\Omega(\frac{N}{\varepsilon})$     &     $\bigO{\frac{1}{(1-\rho)^b\varepsilon^2}}$    \\
	  ProxGT-SR-O/E~\cite{xin21}     &    \cmark     &     $\bigO{\frac{1}{\varepsilon}}$ or $\bigO{\frac{1}{\varepsilon^{0.5}}}$     &     $\bigO{\frac{1}{N\varepsilon^{1.5}}}^{\dagger}$  \\
	  \midrule
      \textbf{Theorem~\ref{theorem:constant_convergence}}
      &    \cmark     &     $\bigO{\frac{1}{\varepsilon^{0.5}}}$ then $\bigO{1}$     &     $\bigO{\max\left\{\frac{1}{N\varepsilon^{1.5}},\frac{1}{(1-\rho)^2\varepsilon},\frac{N^{0.5}}{\varepsilon^{0.75}}\right\}}^{\ddagger}$    \\
      \textbf{Theorem~\ref{theorem:diminishing_convergence}}
      &    \cmark     &    $\bigO{1}$     &     $\tilde{\mathcal{O}}\left(\frac{1}{\varepsilon^{1.5}}\right)$    \\
	  \bottomrule
	\end{tabular}
	\end{adjustbox}
	\caption{Comparison between \ourmethod (bottom two rows) and representative decentralized stochastic nonconvex methods. The sample complexity takes into account both the stationarity and consensus violation. Since D-GET and SPPDM do not show the dependence on $\rho$, we use unspecified powers $a$ and $b$, following the practice of~\cite{xin21hsgd}. $^\dagger$The sample complexity of ProxGT-SR-O/E is independent of $\rho$ by \emph{requiring} multiple communications per update; this is similar to our result in Theorem~\ref{theorem:diminishing_convergence}. $^\ddagger$With multiple communications and $\varepsilon\le N^{-2}$, Theorem~\ref{theorem:constant_convergence} guarantees our algorithm attains the optimal $\bigO{N^{-1}\varepsilon^{-1.5}}$ sample complexity, but with a smaller batch size than ProxGT-SR-O/E .}
	\label{table:related_works}
\end{table*}

\section{Related works}
A rich body of literature exists for solving the problem~\eqref{dco_problem} in the decentralized setting. We discuss related works below.

\textbf{Nonconvex decentralized methods.} Of particular relevance to this work are methods for nonconvex $f_i$'s. When $f_i$ takes the finite-sum form, deterministic methods (with full gradient computation) such as DGD~\cite{zeng18}, Near-DGD~\cite{iakovidou21}, Prox-PDA~\cite{hong17}, xFILTER~\cite{sun19}, and SONATA~\cite{scutari19} converge to an $\varepsilon$-stationary point in $\bigO{\varepsilon^{-1}}$ iterations. They all work for the case $r\equiv0$ only, except SONATA. For stochastic methods, we summarize a few representative ones in Table~\ref{table:related_works}, including the information of whether they handle $r\not\equiv0$. Note that D-PSGD~\cite{lian17} extends the convergence results of DGD; D$^2$~\cite{tang18} further improves over D-PSGD by relaxing a dissimilarity assumption.

Gradient tracking~\cite{lorenzo16,nedic17} has been introduced as a tool to track the gradient of the global objective and has been studied extensively in the nonconvex and stochastic setting, under different names~\cite{zhang20gradtrack,songtao19,koloskova21,xin21_dsgt}. Many works now utilize this technique to improve the performance of their methods; those that mimic the SARAH~\cite{nguyen17} and Spider~\cite{wang19} updates have become popular for their improved theoretical convergence rates. D-SPIDER-SFO~\cite{pan20} and D-GET~\cite{sun20} are two such methods. When $f_i$ takes the finite-sum form, GT-SARAH~\cite{xin21sarah} and DESTRESS~\cite{li21} improve the analysis of D-GET by obtaining an optimal sample complexity and an optimal communication complexity, respectively. All these methods require computing a stochastic gradient with a large batch size every few iterations. 

GT-HSGD~\cite{xin21hsgd} can be considered a special case of our method. It uses a stochastic gradient estimator of the form proposed in~\cite{cutkosky19,levy21}, requiring a large initial batch size, followed by $\bigO{1}$ batch size subsequently. The convergence analysis of GT-HSGD requires $r\equiv0$; hence part of our work is to extend it to the case of $r\not\equiv0$. Similar extensions have been proposed for other methods; for example, ProxGT-SR-O/E~\cite{xin21} extends D-GET, GT-SARAH, and DESTRESS. Additionally, the primal-dual method SPPDM~\cite{wang21} is shown to converge in $\bigO{\varepsilon^{-1}}$ communications, but it requires a large batch size proportional to $\varepsilon^{-1}$. Using such a batch size can negatively impact the performance on machine learning problems~\cite{keskar17}.

\textbf{Other decentralized methods.} Several other decentralized methods exist for scenarios differing from that considered here. They include methods that work for convex problems only, such as DGD~\cite{yuan16}, EXTRA~\cite{shi15}, ADMM~\cite{shi14}, DIGing~\cite{nedic17}, Acc-DNGD~\cite{qu19}, MSDA~\cite{scaman17}, DPAG~\cite{ye20dapg}, Flex-PD~\cite{mansoori21}, IDEAL~\cite{arjevani20}, PUDA~\cite{alghunaim21}, PMGT-VR~\cite{ye20pmgtvr}, and DPSVRG~\cite{li21dpsvrg}; asynchronous methods, such as AD-PSGD~\cite{lian17_async}, the Asynchronous Primal-Dual method~\cite{wu17}, APPG~\cite{zhang21}, asynchronous ADMM~\cite{wei13, hong17admm}, and AD-OGP~\cite{jiang21}; methods that operate under a time-varying network topology, such as Acc-GT~\cite{li21accgt} and ADOM~\cite{kovalev21}; and methods that focus on providing convergence guarantees when communication compression is used, such as DCD-PSGD~\cite{tang18compression}, SQuARM-SGD~\cite{singh21}, and the Primal-Dual method developed in~\cite{chen21}.

\section{DEEPSTORM framework}\label{sec:algo}

We first state the assumed conditions of each $\phi_i$ and the communication graph $\graph$. They are standard in variance reduction~\cite{cutkosky19,xu20,tran22} and decentralized methods~\cite{lian17,sun20,xin21_dsgt}.

\begin{assumption}\label{assumption:objective_function}
  The following conditions hold.
  \begin{enumerate}[(i), leftmargin=*]
  \item The regularizer function $r$ is convex and admits an easily computable proximal mapping.
  \item Each component function $f_i$ is mean-squared $L$-smooth; i.e. there exists a constant $0<L<\infty$ such that $\forall\,\vec{a},\vec{b}\in\R^{1\times p}$ and $\forall\,i=1,\dots,N$,
	\begin{equation}\label{assumption:smoothness}
	  \Exp_\xi\norm{\nabla f_i(\vec{a};\xi)-\nabla f_i(\vec{b};\xi)}_2^2\le L^2\norm{\vec{a}-\vec{b}}_2^2.
	\end{equation}
  \item There exists $\sigma>0$ such that $\forall\,\vec{a}\in\R^{1\times p}$,
	\begin{equation}\label{assumption:variance}
		\begin{split}
	  &\Exp_\xi[\nabla f_i(\vec{a};\xi)]=\nabla f_i(\vec{a}),\\
	  &\Exp\norm{\nabla f_i(\vec{a};\xi)-\nabla f_i(\vec{a})}_2^2\le\sigma^2.
	  \end{split}
	\end{equation}
  \item The global function $\phi=\frac{1}{N}\sum_{i=1}^N\phi_i$ is lower bounded; i.e. there exists a constant $\phi^*$ such that
	\begin{equation}\label{assumption:lower_bound}
	  -\infty<\phi^*\le \phi(\vec{a}),\enskip\forall\vec{a}\in\R^{1\times p}.
	\end{equation}
  \end{enumerate}
\end{assumption}

\begin{assumption}\label{assumption:mixing_matrix}
  The graph $\graph$ is connected and undirected. It can be represented by a \emph{mixing matrix} $\W\in\R^{N\times N}$ such that:
  \begin{enumerate}[(i), leftmargin=*]
  \item \textbf{(Decentralized property)} $w_{ij}>0$ if $(i,j)\in \edges$ and $w_{ij}=0$ otherwise;
  \item \textbf{(Symmetric property)} $\W=\W^\top$;
  \item \textbf{(Null-space property)} $\mathrm{null}\left(\identity-\W\right)=\mathrm{span}\{\basis\}$, where $\basis\in\R^N$ is the vector of all ones; and
  \item \textbf{(Spectral property)} the eigenvalues of $\W$ lie in the range $(-1,1]$ with
    \begin{equation}\label{spectral_gap}
        \rho\triangleq \norm{\W-\frac{1}{N}\basis\basis^\top}_2<1.
    \end{equation}
  \end{enumerate}
\end{assumption}

Note that the entry values of $\W$ can be flexibly designed as long as Assumption~\ref{assumption:mixing_matrix} holds. One example is $\W=\identity-\mathbf{L}/\tau$, where $\mathbf{L}$ is the combinatorial Laplacian of $\graph$ and $\tau$ is a value greater than half of its largest eigenvalue. It is not hard to see that the consensus constraint $\blx_i=\blx_j$ for all $(i,j)\in\edges$ in~\eqref{dco_problem} is equivalent to $\W\x=\x$, where the $i$-th row of $\x$ is $\blx_i$. The value $\rho$ in~\eqref{spectral_gap} indicates the connectedness of the graph. The quantity $1-\rho$ is sometimes referred to as the \emph{spectral gap}; a higher value suggests that the graph is more connected and consensus of the $\blx_i$'s is easier to achieve.

Under Assumptions~\ref{assumption:objective_function} and~\ref{assumption:mixing_matrix}, we now present the \ourmethod framework. We start with the basic algorithm and later generalize the simple communication (using $\W$) with a more general communication operator, denoted by $\w_T$.

\textbf{Basic algorithm.} Let $\blx_i^{(k)}$ be the $k$-th iterate for agent $i$, and let the matrix $\x^{(k)}$ contain all the $k$-th iterates among agents, stacked as a matrix. We will similarly use such vector and matrix notations for other variables. Our \textbf{DE}c\textbf{E}ntralized \textbf{P}roximal \textbf{STO}chastic \textbf{R}ecursive \textbf{M}omentum framework, \ourmethod, uses a variance reduction variable $\vec{d}_i^{(k)}$ and a gradient tracking variable $\vec{y}_i^{(k)}$ to improve the convergence of $\blx_i^{(k)}$. \ourmethod contains the following steps in each iteration $k$:

\begin{enumerate}[leftmargin=*]
\item Communicate the local variables:
  \begin{equation}\label{algo:comm_update}
    \z^{(k)} = \W \x^{(k)}.
  \end{equation}
  
\item Update each local variable (by using, e.g., proximal mappings):
  \begin{equation}\label{algo:x_update}
	\small\xiknew=\argmin_{\blx_i}\left\{\alpha_kr(\blx_i)+\frac{1}{2}\norm{\blx_i-\left(\blz_i^{(k)}-\alpha_k\vec{y}_i^{(k)}\right)}^2\right\}.
  \end{equation}

\item Update the variance reduction variable:
  \begin{equation}\label{base_var_reduction}
    \vec{d}_i^{(k+1)}=(1-\beta_k)\left(\vec{d}_i^{(k)}+\vec{v}_i^{(k+1)}-\vec{u}_i^{(k+1)}\right)+\beta_k\tilde{\vec{v}}_i^{(k+1)},
  \end{equation}
  where
  \begin{equation}\label{storm_u_v}
  	\begin{split}
    \vec{v}_i^{(k+1)}=&\frac{1}{m}\sum_{\mathclap{\xi\in B_i^{(k+1)}}}\nabla f_i(\blx^{(k+1)}_i;\xi),\\
    \vec{u}_i^{(k+1)}=&\frac{1}{m}\sum_{\mathclap{\xi\in B_i^{(k+1)}}}\nabla f_i(\blx^{(k)}_i;\xi).
    \end{split}
  \end{equation}
  Here, $B_i^{(k+1)}$ is a batch of $m$ samples at the current iteration. Note that while $\vec{v}_i^{(k+1)}$ is evaluated at the current iterate, $\vec{u}_i^{(k+1)}$ is evaluated at the previous iterate. We make the assumption that for all $k$ and all agents $i$ and $j$, $B_i^{(k+1)}$ and $B_j^{(k+1)}$ contain independent and mutually independent random variables. The part $\tilde{\vec{v}}_i^{(k+1)}$ can be any unbiased estimate of $\nabla f_i(\blx^{(k+1)}_i)$ with bounded variance; its details will be elaborated soon.

\item Update the gradient tracking variable via communication:
  \begin{equation}\label{algo:y_update}
    \y^{(k+1)} = \W \left( \y^{(k)} + \d^{(k+1)} - \d^{(k)} \right).
  \end{equation}

\end{enumerate}

The step that updates the variance reduction variable, \eqref{base_var_reduction}, is motivated by Hybrid-SGD~\cite{tran22}, which allows for a single-loop update. Intuitively, this variable is a convex combination of the SARAH~\cite{nguyen17} update and $\tilde{\vec{v}}_i^{(k+1)}$, allowing for strong variance reduction and meanwhile flexibility in design. By doing so, a constant batch size $m$ suffices for convergence. This is a useful property in scenarios of online learning and real-time decision making, where it is unrealistic to obtain and store mega batches for training~\cite{xu20,xin21hsgd}.

\textbf{Examples of $\tilde{\vec{v}}_i^{(k+1)}$.} The vector $\tilde{\vec{v}}_i^{(k+1)}$ in~\eqref{base_var_reduction} can be any unbiased local gradient estimate. In this work, we consider two cases: either $\tilde{\vec{v}}_i^{(k+1)}$ is evaluated on another set of samples $\tilde{B}_i^{(k+1)}$, defined analogously to $B_i^{(k+1)}$ that is used to compute $\vec{v}_i^{(k+1)}$ in~\eqref{storm_u_v}, such that
\begin{equation}\label{unbiased_assumption}
	\begin{split}
  &\text{$\tilde{B}_i^{(k+1)}$ is independent of $B_i^{(k+1)}$ with }\\
  &\Exp\norm{\tilde{\vec{v}}_i^{(k+1)}-\nabla f_i(\blx_i^{(k+1)})}_2^2\le\hat{\sigma}^2;
  \end{split}\tag{v1}
\end{equation}
or simply
\begin{equation}\label{v2_update}
  \tilde{\vec{v}}_i^{(k+1)}=\vec{v}_i^{(k+1)}\text{ with }\Exp\norm{\vec{v}_i^{(k+1)}-\nabla f_i(\blx_i^{(k+1)})}_2^2\le\hat{\sigma}^2, \tag{v2}
\end{equation}
for some $\hat{\sigma}>0$. Two possible unbiased estimators that satisfy~\eqref{unbiased_assumption} are
\begin{align}
  \tilde{\vec{v}}_i^{(k+1)}=&\frac{1}{m}\sum_{\mathclap{\tilde{\xi}\in\tilde{B}_i^{(k+1)}}}\nabla f_i(\blx_i^{(k+1)};\tilde{\xi}),\label{v1_update}\tag{v1-SG}\\
  \tilde{\vec{v}}_i^{(k+1)}=&
  \frac{1}{m}\sum_{\mathclap{\tilde{\xi}\in\tilde{B}_i^{(k+1)}}}\nabla f_i(\blx_i^{(k+1)};\tilde{\xi})\nonumber\\
  &+\frac{1}{m}\sum_{\mathclap{\tilde{\tilde{\xi}}\in \tilde{\tilde{B}}_i^{(\tau_{k+1})}}}\nabla f_i(\blx_i^{(\tau_{k+1})};\tilde{\tilde{\xi}})
  -\frac{1}{m}\sum_{\mathclap{\tilde{\xi}\in\tilde{B}_i^{(k+1)}}}\nabla f_i(\blx_i^{(\tau_{k+1})};\tilde{\xi}), \label{v3_update}\tag{v1-SVRG}
\end{align}
for some $\tau_{k+1} < k+1$. The first estimator is a standard one, evaluated by using a batch $\tilde{B}_i^{(k+1)}$ independent of $B_i^{(k+1)}$. The second estimator, which introduces further variance reduction, uses an additional past-time iterate $\blx_i^{(\tau_{k+1})}$ and a batch $\tilde{\tilde{B}}_i^{(\tau_{k+1})}$, whose size is generally greater than $m$. Such an update is inspired by the SVRG method~\cite{johnson13}. Here, we have $\hat{\sigma}^2=m^{-1}\sigma^2$ for the estimators~\eqref{v1_update} and~\eqref{v2_update}; while $\hat{\sigma}^2=\left(3m^{-1}+6\abs{\tilde{\tilde{B}}_i^{(\tau_{k+1})}}^{-1}\right)\sigma^2$ for~\eqref{v3_update}, where we recall that $\sigma^2$ comes from~\eqref{assumption:variance}. Note that beyond the two examples, our proof techniques hold for any unbiased estimator satisfying~\eqref{unbiased_assumption}, leaving more open designs.

\textbf{Generalized communication.} Steps~\eqref{algo:comm_update} and~\eqref{algo:y_update} use the mixing matrix to perform weighted averaging of neighbor information. The closer $\W$ is to $\frac{1}{N}\basis\basis^\top$, the more uniform the rows of $\x^{(k+1)}$ are, implying agents are closer to consensus. Hence, to improve convergence, we can apply multiple mixing rounds in each iteration. To this end, we generalize the network communication by using an operator $\w_T$, which is a degree-$T$ polynomial in $\W$ that must satisfy Assumption~\ref{assumption:mixing_matrix} parts (\textit{ii})--(\textit{iv}). We adopt Chebyshev acceleration~\cite{auzinger11,scaman17,xin21,li21}, which defines for any input matrix $\B_0$, $\B_T=\w_T(\B_0)$, where $\B_1=\W\B_0$, $\mu_0 = 1$, $\mu_1 = \frac{1}{\rho}$ for $\rho$ defined in~\eqref{spectral_gap}, and recursively,
\begin{equation}
	\begin{split}
  \mu_{t+1} =& \frac{2}{\rho}\mu_t - \mu_{t-1} \text{ and}\\
  \B_{t+1}=&\frac{2\mu_t}{\rho\mu_{t+1}}\W\B_t-\frac{\mu_{t-1}}{\mu_{t+1}}\B_{t-1},\text{ for } t\le T-1.
  \end{split}
\end{equation}

It is not hard to see that $\basis$ is an eigenvector of $\w_T$, associated to eigenvalue 1, whose algebraic multiplicity is 1. Therefore, 
\begin{equation}\label{spectral_gap_cheby}
  \tilde{\rho}\triangleq\norm{\w_T-\avg}_2<1.
\end{equation}
Moreover, $\tilde{\rho}$ converges to zero exponentially with $T$, bringing $\w_T$ rather close to an averaging operator (for details, see Appendix~\ref{appendix:chebyshev}). Notice with $T=1$, $\w_T$ reduces to $\W$.

We summarize the overall algorithm in Algorithm~\ref{algo:}, by replacing $\W$ in~\eqref{algo:comm_update} and~\eqref{algo:y_update} with $\w_T$. Additionally, see the discussions after Theorems~\ref{theorem:constant_convergence} and~\ref{theorem:diminishing_convergence} regarding the probability distribution for choosing the output of Algorithm~\ref{algo:}.

\begin{algorithm}[tb]
\caption{DEEPSTORM}
\label{algo:}
\textbf{Input}: Initial $\x^{(0)}$, mixing rounds $T_0,T$, iteration $K$, and $\{\alpha_k\}$, $\{\beta_k\}$
\begin{algorithmic}[1] 
\STATE Compute $\vec{d}_i^{(0)}=\frac{1}{m_0}\sum_{\xi\in B_i^{(0)}}\nabla f_i(\blx_i^{(0)};\xi)$ $\forall i$
\STATE Communicate to obtain $\y^{(0)}=\w_{T_0}(\d^{(0)})$
\FOR{$k=0,\dots,K-1$}
\STATE Communicate to obtain $\z^{(k)} = \w_T(\x^{(k)})$
\STATE Update local decision variables by~\eqref{algo:x_update}
\STATE Obtain local gradient estimator by~\eqref{base_var_reduction}
\STATE Communicate to update gradient tracking variable $\y^{(k+1)} = \w_T ( \y^{(k)} + \d^{(k+1)} - \d^{(k)})$
\ENDFOR
\end{algorithmic}
\textbf{Output}: $\z^{(\tau)}$ with $\tau$ chosen randomly from $\{0,\dots,K-1\}$
\end{algorithm}

\section{Convergence results}\label{sec:theory}

For the convergence of \ourmethod, we start with the following standard definitions~\cite{xu20,xin21}.

\begin{definition}\label{def:prox_grad_map}
  Given $\blx\in\hbox{\normalfont dom}(r)$, $\bly$, and $\eta>0,$ define the \emph{proximal gradient mapping} of $\bly$ at $\blx$ to be
  \begin{equation}\label{def:prox_grad_map:eqn}
    \textstyle
	P\left(\blx,\bly,\eta\right)\triangleq\frac{1}{\eta}\left(\blx-\hbox{\normalfont prox}_{\eta r}(\blx-\eta\bly)\right),
  \end{equation}
  where $\hbox{\normalfont prox}$ denotes the proximal operator $\hbox{\normalfont prox}_{g}(\mathbf{v})=\argmin_{\mathbf{u}}\left\{g(\mathbf{u})+\frac{1}{2}\|\mathbf{u}-\mathbf{v}\|_2^2\right\}$.
\end{definition}

\begin{definition}\label{def:stationarity}
  A stochastic matrix $\x\in\R^{N\times p}$ is called a \emph{stochastic $\varepsilon$-stationary point} of (\ref{dco_problem}) if 
  \begin{equation}\label{def:stationarity:eqn}
 	\Exp\left[\frac{1}{N}\sum_{i=1}^N\norm{P\left(\blx_i,\nabla f(\blx_i),\eta\right)}_2^2+\frac{L^2}{N}\fronorm{\x_\perp}\right]\le\varepsilon,
  \end{equation}
  where $\eta>0$, $\nabla f\triangleq\frac{1}{N}\sum_{j=1}^N\nabla f_j$, $\blx_i$ is the $i$-th row of $\x$, and $\x_\perp\triangleq\x-\frac{1}{N}\basis\basis^\top\x$ is the difference between all $\blx_i$ and their average $\frac{1}{N}\sum_{j=1}^N\blx_j$.
\end{definition}

Our analyses rely on the construction of two novel Lyapunov functions as indicated by  Theorems~\ref{theorem:constant_convergence} and~\ref{theorem:diminishing_convergence} below. These Lyapunov functions guarantee convergence through the careful design of function coefficients which result from solving non-linear systems of inequalities in either the constant or diminishing step size case. We first consider the use of a constant step size. The convergence rate result is given in the following theorem. Its proof is given in Appendix~\ref{appendix:constant_proof}.

\begin{theorem}\label{theorem:constant_convergence}
  Under Assumptions~\ref{assumption:objective_function} and~\ref{assumption:mixing_matrix}, let $\left\{\left(\x^{(k)},\d^{(k)},\y^{(k)},\z^{(k)}\right)\right\}$ be obtained by Algorithm~\ref{algo:} via~\eqref{algo:x_update},~\eqref{algo:y_update}, and~\eqref{base_var_reduction} such that $\tilde{\vec{v}}_i^{(k+1)}$ is any unbiased gradient estimator that satisfies either~\eqref{unbiased_assumption} or~\eqref{v2_update}. Further, let $\alpha_k$ and $\beta_k$ be chosen as 
 {\small
  \begin{equation}\label{lemma:constant_choice:bounds}
  	\begin{split}
	&\alpha_k=\frac{\alpha}{K^{\frac{1}{3}}},\enskip \beta_k=\frac{144L^2\alpha^2}{NK^{\frac{2}{3}}},
    \text{ with }\\ &\alpha\le\min\left\{\frac{K^{\frac{1}{3}}}{32L},\frac{(1-\tilde{\rho})^2K^{\frac{1}{3}}}{64L}\right\},
    \end{split}
  \end{equation}
  }for all $k=0,\dots,K-1$.
  Then, it holds that $\beta_k\in(0,1)$ for all $k\ge0$ and that
  {\small
  \begin{equation}\label{theorem:constant_convergence:bound}
	\begin{split}
	  &\frac{1}{K}\sum_{k=0}^{K-1}\Exp\left(\frac{1}{N}\sum_{i=1}^N\norm{P\left(\vec{z}_i^{(k)},\nabla f(\vec{z}_i^{(k)}),\alpha_k\right)}_2^2+\frac{L^2}{N}\fronorm{\z_\perp^{(k)}}\right)\\
	  &\le\frac{512}{\alpha K^{\frac{2}{3}}}\left(\Phi^{(0)}-\phi^*\right)+\left(\frac{2048}{L(1-\tilde{\rho})^2K}\right)\frac{144^2L^4\alpha^3\hat{\sigma}^2}{N^2}\\
	  &+\left(\frac{128}{3L^2\alpha K^{\frac{2}{3}}}+\frac{8192\alpha}{K^{\frac{4}{3}}}+\frac{2048\alpha}{NK^{\frac{4}{3}}}\right)\frac{144^2L^4\alpha^3\hat{\sigma}^2}{N^2},
	\end{split}
  \end{equation}
  }for some $\Phi^{(0)}>\phi^*$ that depends on the initialization. Note that $\Phi^{(k)}$ is defined in~\eqref{lemma:base_change_update:lyapunov} in Appendix~\ref{appendix:theory} for any $k\ge0.$
\end{theorem}

\textbf{Network-independent sample complexity, linear speed-up, and communication complexity.} Theorem~\ref{theorem:constant_convergence} establishes convergence based on the sequence $\{\z^{(k)}\}$ defined in~\eqref{algo:comm_update}. As a consequence, if we let each agent start with the same initial variable $\blx^{(0)}$, set $\alpha=\frac{N^{\frac{2}{3}}}{64L}$ and the initial batch size $m_0=\sqrt[3]{NK}$, and choose initial communication rounds $T_0=\tilde{\mathcal{O}}\left((1-\rho)^{-0.5}\right)$ for $\y^{(0)}$, then for all $K\ge\frac{N^2}{(1-\tilde{\rho})^6}$, \ourmethod achieves stochastic $\varepsilon$-stationarity for some iterate $\z^{(\tau)}$, where $\tau$ is selected uniformly from $\{0,\dots,K-1\}$, by using
\begin{equation}\label{corollary:constant_convergence:gradients_and_comms}
  \bigO{\max\left\{\frac{(L\Delta)^{\frac{3}{2}}+\hat{\sigma}^3}{N\varepsilon^{\frac{3}{2}}},\frac{\hat{\sigma}^2}{(1-\tilde{\rho})^2\varepsilon},\frac{\sqrt{N}\hat{\sigma}^{\frac{3}{2}}}{\varepsilon^{\frac{3}{4}}}\right\}}
\end{equation}
local stochastic gradient computations. For the formal statement, see Corollary~\ref{corollary:constant_convergence} in Appendix~\ref{appendix:constant_complexity}. Here, $\Delta=\Phi^{(0)}-\phi^*$ denotes an initial function gap, which is independent of $\tilde{\rho}$, $N$, and $K$. Moreover, when $\varepsilon\le N^{-2}(1-\tilde{\rho})^4$, we see that $\bigO{N^{-1}\varepsilon^{-1.5}}$ dominates in~\eqref{corollary:constant_convergence:gradients_and_comms}; hence, this result manifests a linear speed-up with respect to $N$ over the centralized counterparts~\cite{cutkosky19,tran22} of \ourmethod. Furthermore, if the number of Chebyshev mixing rounds is $T=\lceil\frac{2}{\sqrt{1-\rho}}\rceil$, we have $(1-\tilde{\rho})\ge\frac{1}{\sqrt{2}}$, which suggests that $\varepsilon$ does not need to be small for the linear speed-up to hold. For details, see Lemma~\ref{lemma:chebyshev} and Remark~\ref{remark:complexity_discussion} in the Appendix. The communication cost is $\bigO{T_0+TK}$. 

In parallel, we state a result for the case of diminishing step size. Its proof is given in Appendix~\ref{appendix:diminishing_proof}.

\begin{theorem}\label{theorem:diminishing_convergence}
  Under the same assumptions as Theorem~\ref{theorem:constant_convergence},
  let $\alpha_k$ and $\beta_k$ be chosen as 
  {\small
  \begin{equation}\label{theorem:diminishing_choice:bounds}
    \begin{split}
	&\alpha_k=\frac{\alpha}{(k+k_0)^{\frac{1}{3}}},\enskip \beta_k=1-\frac{\alpha_{k+1}}{\alpha_k}+48L^2\alpha_{k+1}^2,\text{ with }\\
	&\alpha\le\min\left\{\frac{k_0^{\frac{1}{3}}}{32L},\frac{(1-\tilde{\rho})^2k_0^{\frac{1}{3}}}{64L}\right\},
	\end{split}
  \end{equation}
  }for all $k=0,\dots,K-1$, where $k_0\ge\lceil\frac{2}{1-\tilde{\rho}^3}\rceil$.
  Then, it holds that $\beta_k\in(0,1)$ for all $k\ge0$ and that
  {\small
  \begin{equation}\label{theorem:diminishing_convergence:bound}
	\begin{split}
	  &\sum_{k=0}^{K-1}c\alpha_k\Exp\left(\frac{1}{N}\sum_{i=1}^N\norm{P\left(\vec{z}_i^{(k)},\nabla f(\vec{z}_i^{(k)}),\alpha_k\right)}_2^2+\frac{L^2}{N}\fronorm{\z_\perp^{(k)}}\right)\\
	  \le&12\left(\hat{\Phi}^{(0)}-\phi^*\right)+\sum_{k=0}^{K-1}\left(\frac{1}{L^2\alpha_{k+1}}+\frac{48}{L(1-\tilde{\rho})^2}\right)\beta_k^2\hat{\sigma}^2,
	\end{split}
  \end{equation}
  }for some $\hat{\Phi}^{(0)} > \phi^*$ that depends on initialization and $c\triangleq\frac{k_0^{\frac{1}{3}}}{2k_0^{\frac{1}{3}}+(k_0+1)^{\frac{1}{3}}}>\frac{1}{4}$. Note that $\hat{\Phi}^{(k)}$ is defined in~\eqref{lemma:base_change_diminishing:lyapunov} in Appendix~\ref{appendix:theory} for any $k\ge0.$
\end{theorem}

\textbf{Sample complexity.} Theorem~\ref{theorem:diminishing_convergence} establishes the convergence rate of DEEPSTORM with diminishing step sizes. If we choose $k_0=\lceil\frac{2}{(1-\tilde{\rho})^6}\rceil$ in~\eqref{theorem:diminishing_choice:bounds}, then \ourmethod achieves stochastic $\varepsilon$-stationarity for some iterate $\z^{(\tau)}$, where $\tau$ is chosen according to~\eqref{corollary:diminishing_convergence:probability}, by using $\tilde{\mathcal{O}}\left((1-\tilde{\rho})^{-3}\varepsilon^{-1.5}\right)$ local stochastic gradient computations; this sample complexity is network-dependent. However, by using an initialization technique similar to the case of constant step sizes above and letting the initial batch size be $\bigO{1}$, we can set the Chebyshev mixing rounds to be $T=\lceil\frac{2}{\sqrt{1-\rho}}\rceil$, so that $(1-\tilde{\rho})^{-1}\le \sqrt{2}$. This leads to the network-independent sample complexity reported in Table~\ref{table:related_works}. For a full statement of the complexity results, see Corollary~\ref{corollary:diminishing_convergence} in Appendix~\ref{appendix:diminishing_complexity} and Remark~\ref{remark:complexity_discussion_2}.

\begin{figure*}[!h]
  \centering
  \begin{adjustbox}{max width=\textwidth}
	{\small
	\begin{tabular}{lcccc}
	  \toprule
	  Method    &    Train loss    &    Stationarity     &     \% Non-zeros     &     Test accuracy    \\
	  \midrule
	  \multicolumn{5}{c}{a9a}     \\
	  \midrule
	  DSGT     &     0.3308$\pm$1.272e-4	&	0.0003$\pm$1.819e-4	&	74.18$\pm$160.09e-4	&	84.89$\pm$271.02e-4	\\
	  SPPDM    &    0.5457$\pm$20.014e-4	&	0.001$\pm$2.99e-4	&	46.19$\pm$51.04e-4	&	76.38$\pm$0.0e-4	\\
	  ProxGT-SR-E    &    0.545$\pm$85.017e-4	&	0.0491$\pm$64.099e-4	&	98.04$\pm$15.035e-4	&	76.38$\pm$0.0e-4	\\
      \ourmethod v1-SG    &    \underline{0.3306}$\pm$9.46e-4	&	0.0002$\pm$1.292e-4	&	2.99$\pm$60.066e-4	&	\underline{84.96}$\pm$1235.0e-4	\\
      \ourmethod v1-SVRG    &    0.3308$\pm$7.689e-4	&	\textbf{0.0001}$\pm$0.21278e-4	&	\underline{2.86}$\pm$45.018e-4	&	84.94$\pm$929.04e-4	\\
      \ourmethod v2    &    \textbf{0.3277}$\pm$7.461e-4	&	\textbf{0.0001}$\pm$0.8179e-4	&	\textbf{1.92}$\pm$53.073e-4	&	\textbf{85.11}$\pm$478.03e-4	\\
      \midrule
      \multicolumn{5}{c}{MiniBooNE}     \\
      \midrule
      DSGT     &     0.3735$\pm$3.844e-4	&	0.0003$\pm$2.076e-4	&	81.83$\pm$227.0e-4	&	\underline{84.24}$\pm$202.07e-4	\\
      SPPDM     &    0.5699$\pm$61.016e-4	&	0.0025$\pm$5.565e-4	&	35.32$\pm$77.02e-4	&	72.02$\pm$0.0e-4	\\
      ProxGT-SR-E     &    0.5663$\pm$32.027e-4	&	0.0115$\pm$7.57e-4	&	97.88$\pm$17.017e-4	&	72.02$\pm$0.0e-4	\\
      \ourmethod v1-SG    &   \textbf{0.3637}$\pm$19.015e-4	&	\underline{0.0002}$\pm$0.6464e-4	&	\underline{4.34}$\pm$60.07e-4	&	\underline{84.24}$\pm$1902.0e-4	\\
      \ourmethod v1-SVRG    &    0.3653$\pm$23.054e-4	&	\underline{0.0002}$\pm$0.9716e-4	&	4.42$\pm$65.068e-4	&	84.15$\pm$1974.0e-4	\\
      \ourmethod v2    &    \textbf{0.3637}$\pm$18.046e-4	&	\textbf{0.0001}$\pm$0.4136e-4	&	\textbf{4.2}$\pm$61.073e-4	&	\textbf{84.25}$\pm$1752.0e-4	\\
      \midrule
 	  \multicolumn{5}{c}{MNIST}     \\
      \midrule
      DSGT     &    0.1055$\pm$24.03e-4	&	\underline{0.0024}$\pm$3.554e-4	&	51.05$\pm$896.0e-4	&	97.61$\pm$1346.0e-4	\\
      SPPDM    &    0.1851$\pm$55.065e-4	&	0.0051$\pm$2.058e-4	&	66.81$\pm$616.03e-4	&	95.55$\pm$1488.0e-4	\\
      ProxGT-SR-E     &    1.699$\pm$903.07e-4	&	0.21299$\pm$268.0e-4	&	91.4$\pm$70.087e-4	&	52.25$\pm$41480.0e-4	\\
      \ourmethod v1-SG    &    0.081$\pm$33.014e-4	&	0.0027$\pm$5.376e-4	&	\underline{10.31}$\pm$70.031e-4	&	97.97$\pm$1261.0e-4	\\
      \ourmethod v1-SVRG    &   \underline{0.078}$\pm$34.022e-4	&	0.0031$\pm$7.366e-4	&	10.99$\pm$82.095e-4	&	\underline{98.08}$\pm$1485.0e-4	\\
      \ourmethod v2    &    \textbf{0.0768}$\pm$29.095e-4	&	\textbf{0.0016}$\pm$1.83e-4	&	\textbf{7.36}$\pm$50.07e-4	&	\textbf{98.15}$\pm$659.04e-4	\\
	  \bottomrule
	\end{tabular}}
	\end{adjustbox}
 	\caption{Comparisons of different methods by running them with the same number of data passes. Bold values indicate the best results and underlined values indicate the second best.}
    \label{fig:numerical_table}
  \setlength\tabcolsep{1pt}
  	\begin{adjustbox}{width=0.97\textwidth}
  \begin{tabular}{cccc}
	
	\includegraphics[width=0.21\linewidth]{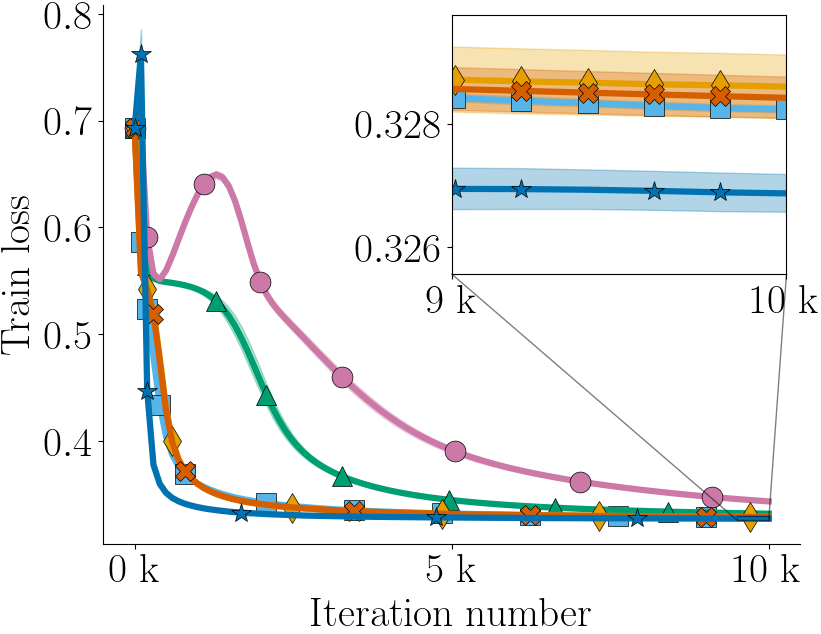} &
	
	\includegraphics[width=0.21\linewidth]{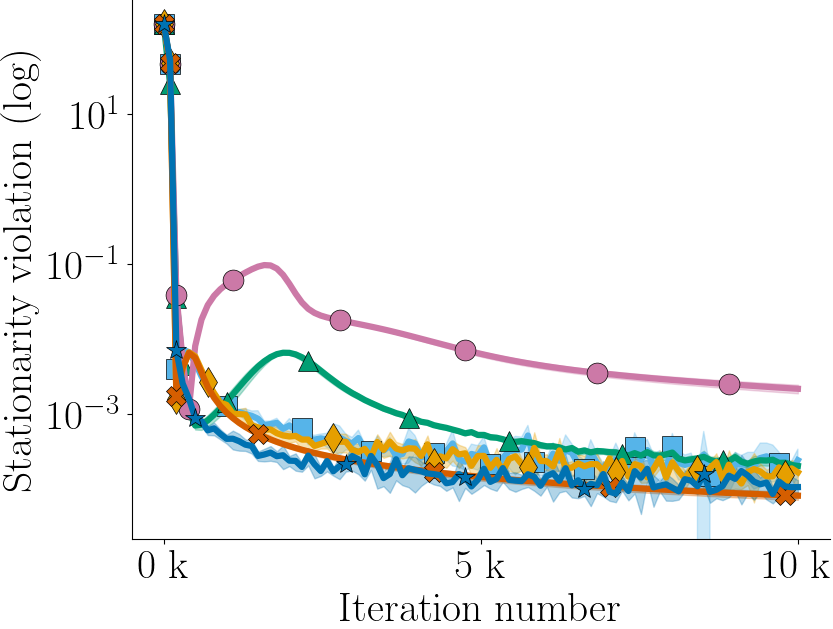} &
	
	\includegraphics[width=0.21\linewidth]{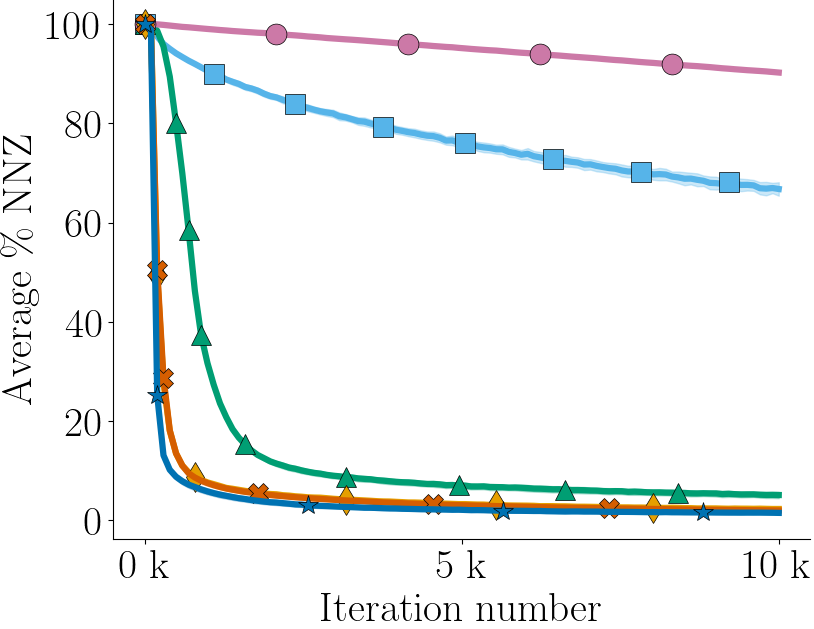} & 
	
	\includegraphics[width=0.21\linewidth]{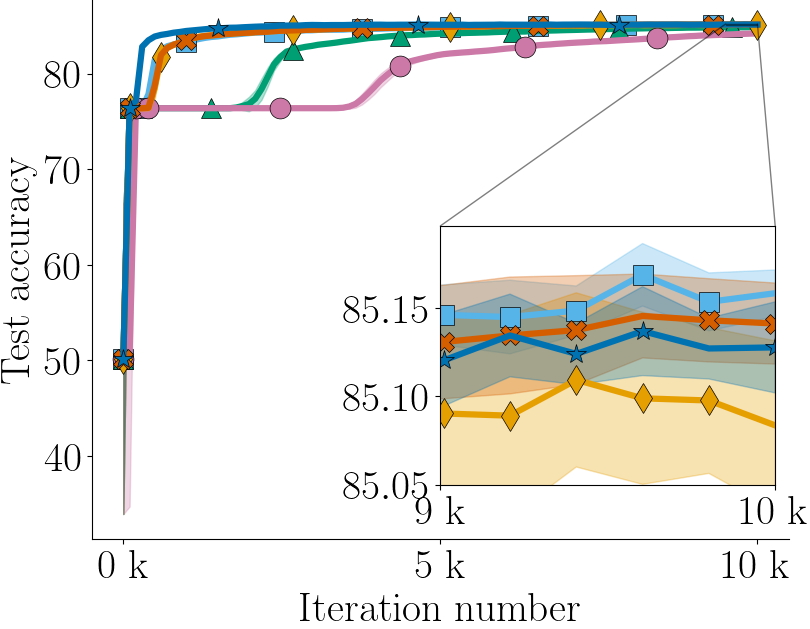} \\
	
	\includegraphics[width=0.21\linewidth]{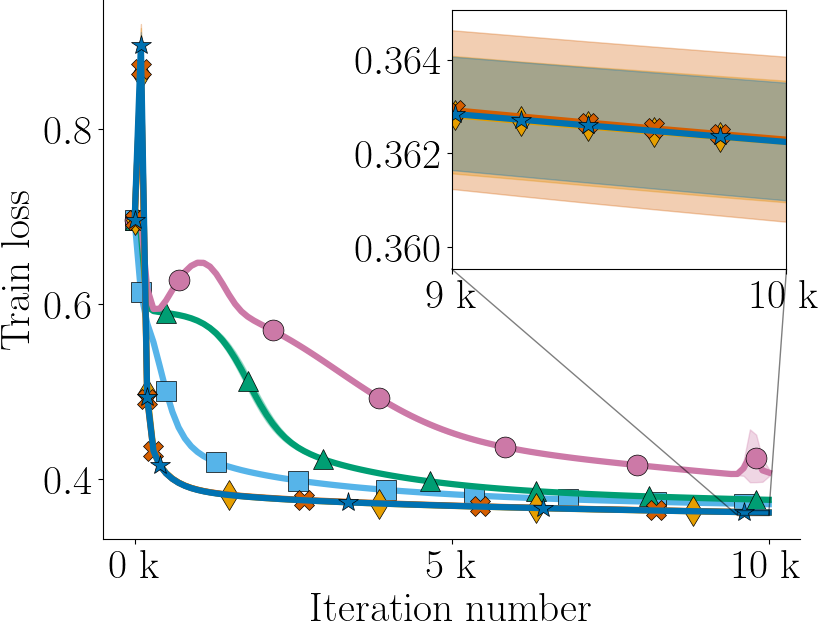} &
	
	\includegraphics[width=0.21\linewidth]{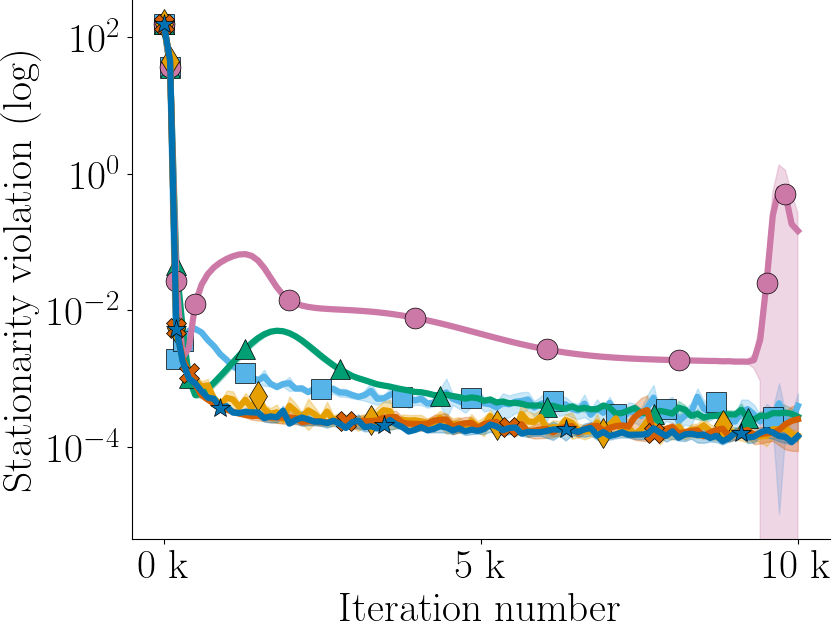} &
	
	\includegraphics[width=0.21\linewidth]{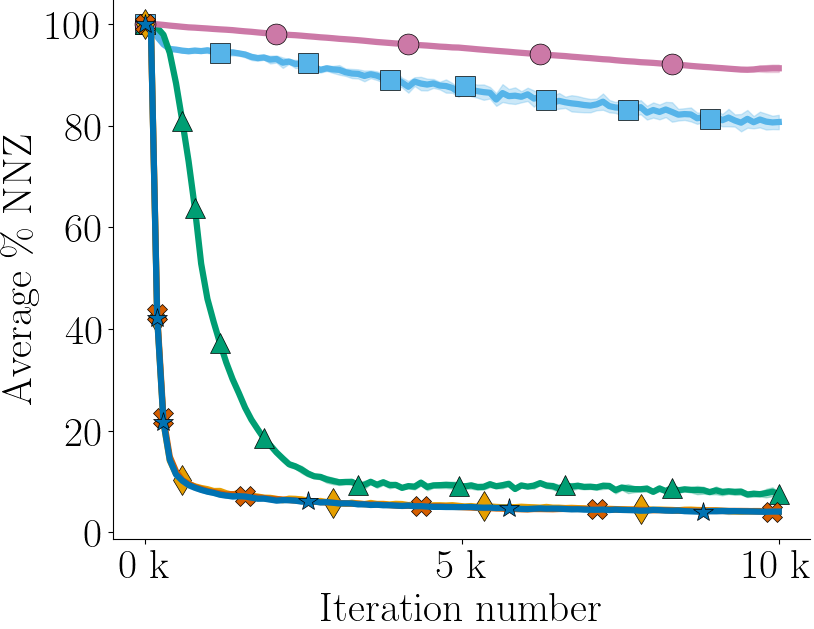} & 
	
	\includegraphics[width=0.21\linewidth]{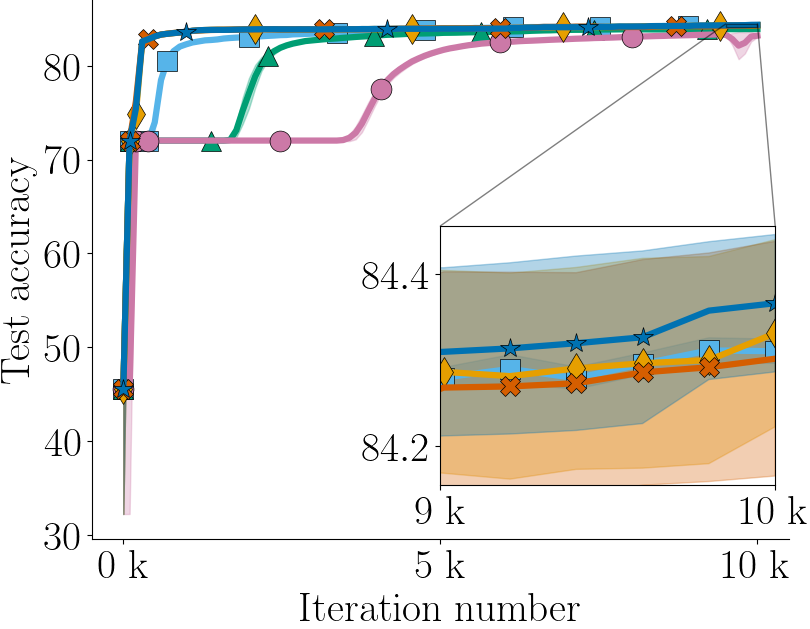} \\
	
	\includegraphics[width=0.21\linewidth]{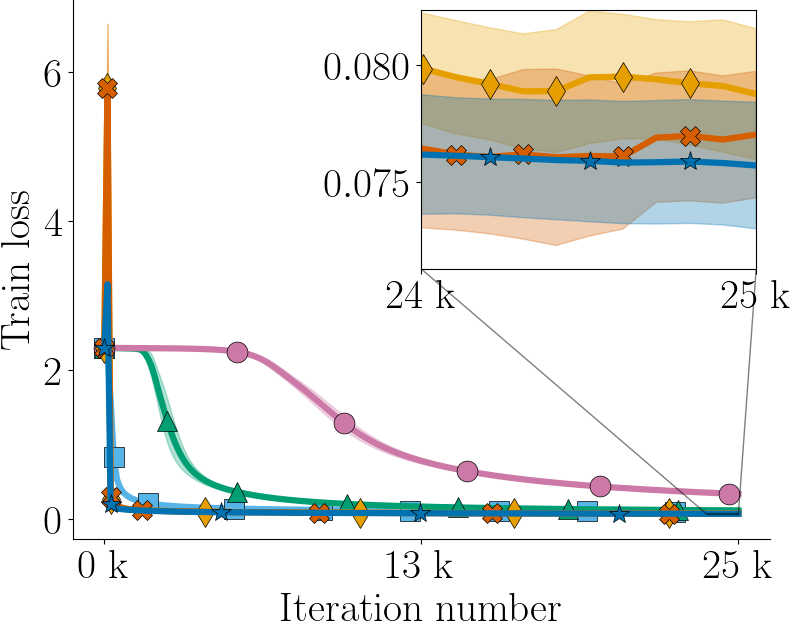} &
	
	\includegraphics[width=0.21\linewidth]{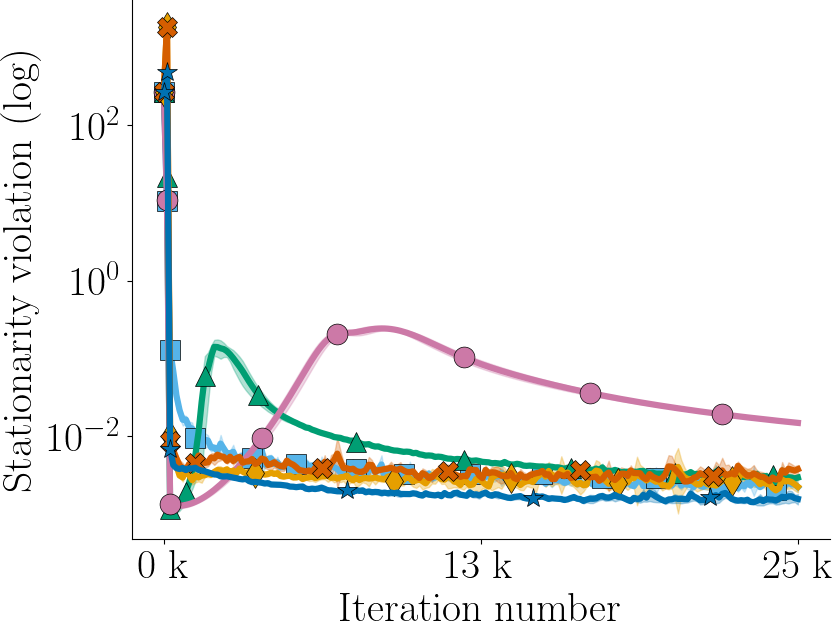} &
	
	\includegraphics[width=0.21\linewidth]{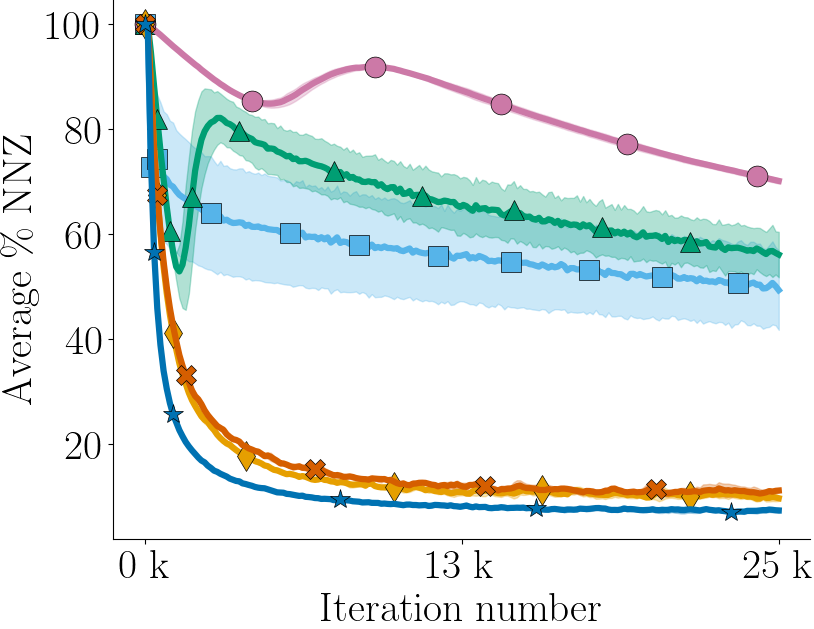} & 
	
	\includegraphics[width=0.21\linewidth]{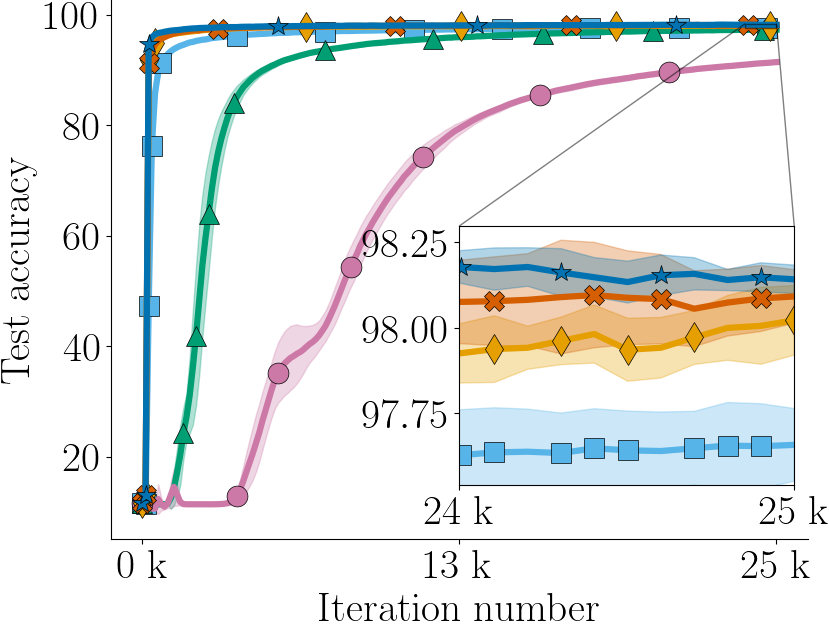} \\
	
	\multicolumn{4}{c}{\includegraphics[width=0.85\linewidth]{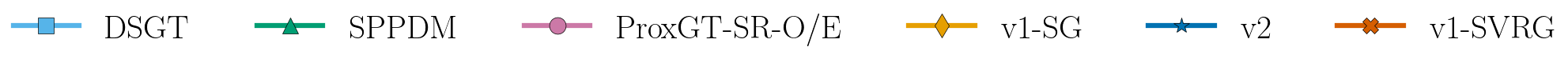}}
	
  \end{tabular}
  \end{adjustbox}
  
  \caption{Comparison of different methods by running them with the same number of iterations. From top to bottom: a9a, MiniBooNE, MNIST. From left to right: training loss, stationarity violation, average percentage non-zeros, testing accuracy. The shaded regions indicate standard deviations (with some being small and unnoticeable).}
  
  \label{fig:numerical}
  
\end{figure*}

\section{Experiments}\label{sec:numerical}

In this section, we empirically validate the convergence theory of \ourmethod and demonstrate its effectiveness in comparison with representative decentralized methods. We compare all versions of \ourmethod with DSGT~\cite{songtao19,zhang20gradtrack,koloskova21,xin21_dsgt}, SPPDM~\cite{wang21}, and ProxGT-SR-O/E~\cite{xin21}. DSGT uses gradient tracking but it is not designed for non-smooth objectives; nevertheless, it outperforms strong competitors (e.g., D-PSGD~\cite{lian17} and D$^{2}$~\cite{tang18}) in practice~\cite{zhang20gradtrack,xin21_dsgt}. SPPDM is a primal-dual method, but it does not utilize gradient tracking and its convergence theory requires a large batch size. ProxGT-SR-O/E is a double-loop algorithm, which requires using a mega-batch to compute the (stochastic) gradient at each outer iteration.  All experiments are conducted using the AiMOS~\footnote{See:~\url{https://cci.rpi.edu/aimos}} supercomputer with eight NVIDIA Tesla V100 GPUs in total, with code implemented in PyTorch (v1.6.0) and OpenMPI (v3.1.4).

\textbf{Problems.} We conduct tests on three classification problems. Each local agent $i$ has the objective $
\phi_i(\blx_i)=\frac{1}{M}\sum_{j=1}^M\ell\left(g\left(\blx_i,\vec{a}_j\right),\vec{b}_j\right)+\lambda\norm{\blx_i}_1,
$ where $g(\blx,\vec{a})$ is the output of a neural network with parameters $\blx$ on data $\vec{a}$, and $\ell$ is the cross-entropy loss function between the output and the true label $\vec{b}$. The data is uniformly randomly split among the agents, each obtaining $M$ training examples. The $L_1$ regularization promotes sparsity of the trained network. The regularization strength $\lambda$ is set to 0.0001 following general practice.

\textbf{Data sets and neural networks.} The three data sets we experiment with are summarized in Table~\ref{table:dataset_information} in Appendix~\ref{appendix:numerical}. Two of them are tabular data and we use the standard multi-layer perceptron for $g$ (one hidden layer with 64 units). The other data set contains images; thus, we use a convolutional neural network. Both neural networks use the tanh activation to satisfy the smoothness condition of the objective function.

\textbf{Communication graphs.} Each data set is paired with a different communication graph, indicated by, and visualized in, Table~\ref{table:dataset_information} in Appendix~\ref{appendix:numerical}. For the ladder and random graphs, the mixing matrix is set as $\W=\identity-\gamma\vec{L}$, where $\gamma$ is reciprocal of the maximum eigenvalue of the combinatorial Laplacian $\vec{L}$. For the ring graph, self-weighting and neighbor weights are set to be $\frac{1}{3}$.

\textbf{Performance metrics.} We evaluate on four metrics: training loss, stationarity violation, solution sparsity, and test accuracy. Further, we compare the methods with respect to data passes and algorithm iterations, which reflect the sample complexity and communication complexity, respectively. Note that for each iteration, all methods except SPPDM communicate two variables. For the training loss, stationarity violation, and test accuracy, we evaluate on the average solution $\bar{\blx}$. The stationarity violation is defined as $\norm{\bar{\blx}-\hbox{prox}_{r}\left(\bar{\blx}-\nabla f(\bar{\blx})\right)}_2^2+\sum_{i=1}^N\norm{\blx_i-\bar{\blx}}_2^2$, which measures both optimality and consensus. For sparsity, we use the average percentage of non-zeros in each $\blx_i$ prior to local communication.

\textbf{Protocols.} For hyperparameter selection, see Appendix~\ref{appendix:numerical}. We perform ten runs with different starting points for each dataset. In several runs for the MNIST dataset, DSGT and SPPDM converge to solutions with $\ll 1$\% non-zero entries, but the training loss and test accuracy are not competitive at all. We remove these runs and keep only the five best runs for reporting the (averaged) performance.

\textbf{Results.} Figure~\ref{fig:numerical_table} summarizes the results for all performance metrics, by using the same number of data passes for all methods when convergence has been observed. For a9a and MiniBooNE, the results are averaged over passes 80 to 100; while for MNIST, over passes 180 to 200. Figure~\ref{fig:numerical} compares different methods by using the same number of algorithm iterations.

Overall, we see that \ourmethod (all variants) generally yields a lower training loss and significantly fewer non-zeros in the solution than the other decentralized algorithms. This observation suggests that \ourmethod indeed solves the optimization problem~\eqref{dco_problem} much more efficiently in terms of both data passes and iterations. Moreover, the test accuracy is also highly competitive, concluding the practical usefulness of \ourmethod.

\section{Conclusion}\label{sec:conclusion}

We have presented a novel decentralized algorithm for solving the nonconvex stochastic composite problem~\eqref{dco_problem} by leveraging variance reduction and gradient tracking. It is the first such work that achieves optimal sample complexity for this class of problems by using $\bigO{1}$ batch sizes. Our algorithm is a framework with an open term (see~\eqref{base_var_reduction}), for which we analyze two examples that allow the framework to achieve network-independent complexity bounds, suggesting no sacrifice over centralized variance reduction methods. Our proof technique can be used to analyze more designs of the open term. While our work is one of the few studies on the nonconvex stochastic composite problem~\eqref{dco_problem}, our analysis is for the synchronous setting with a static communication graph. Analysis (or adaptation of the algorithm) for asynchronous or time-varying settings is an avenue of future investigation.

\section{Acknowledgments}\label{sec:ack}

This work was supported by the Rensselaer-IBM AI Research Collaboration, part of the IBM AI Horizons Network, NSF grants DMS-2053493 and DMS-2208394, and the ONR award N00014-22-1-2573.

{\fontsize{9.0pt}{10.0pt} \selectfont
\bibliography{aaai23.bib}}

\onecolumn
\appendix
\numberwithin{equation}{section}
\numberwithin{lemma}{section}
\numberwithin{remark}{section}
\numberwithin{algorithm}{section}
\section{Reproducibility}\label{appendix:numerical}
\textbf{Data sets and communication graphs.} The data sets and communication graphs are summarized/visualized in Table~\ref{table:dataset_information}.

\begin{table*}[!h]
  \centering
  \begin{tabular}[b]{lccccc}
	\toprule
	Dataset   &    Train     &    Test     &   Features  & Model & Graph  \\
	\midrule
	a9a       & 32{,}561  & 16{,}281 & 123   & MLP   & Ladder \\
	MiniBooNE & 100{,}000 & 30{,}064 & 50    & MLP   & Ring   \\
	MNIST     & 60{,}000  & 10{,}000 & 784   & LENET & Random \\
	\bottomrule
  \end{tabular} \hfill
  \includegraphics[width=0.12\linewidth]{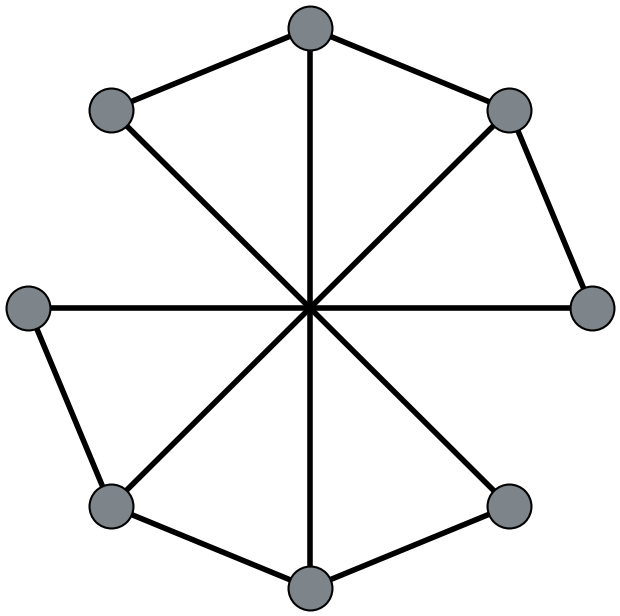}
  \includegraphics[width=0.12\linewidth]{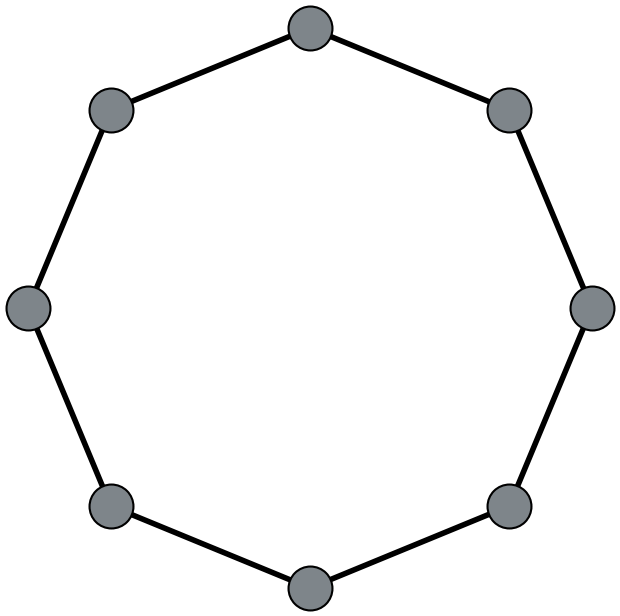}
  \includegraphics[width=0.12\linewidth]{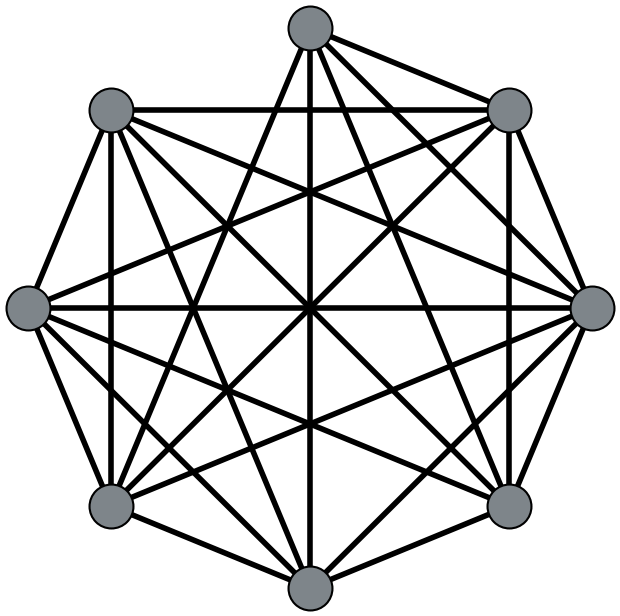}
  \caption{Summary of data sets; all are downloaded from \url{https://www.openml.org}. Graphs from left to right: ladder, ring, and random.}
    \label{table:dataset_information}
\end{table*}

\textbf{Code.} Code for all numerical experiments is available at~\url{https://github.com/gmancino/DEEPSTORM}.

\textbf{Hyperparameter selection.} We choose the batch size according to theoretical guidance, while observing reasonably good performance. For all methods but SPPDM, we set the batch size to be 64, 128, and 64 for a9a, MiniBooNE, and MNIST, respectively. For SPPDM, the respective sizes are 512, 1024, and 128.

For all variants of \ourmethod, we set the number of communication rounds to be $T=1$ such that $\w_1=\W$. We use a diminishing step size as in~\eqref{theorem:diminishing_choice:bounds} with $k_0=\lceil\frac{2}{1-\rho^3}\rceil$ and $\beta_k=1-\frac{\alpha_{k+1}}{\alpha_k}+\beta\alpha_{k+1}^2$ with $\beta<\frac{1}{\alpha_0\alpha_1}$, such that $\beta_0<1$. Such a choice ensures that $\beta_k\in(0,1)$ for all $k$. For the \eqref{v3_update} variant, we compute the snapshot gradient every four passes, by using all local data for a9a and MiniBooNE; whereas for MNIST, we compute the snapshot gradient at the end of every pass, by using 20\% of the local data.

For DSGT, we set the step size to be $\alpha_k=\frac{\alpha}{\sqrt{k+1}}$ for all $k$, according to~\cite{songtao19,xin21_dsgt}. For SPPDM, we follow the choices of many hyperparameters used in the original paper and only tune $c\in\{0.1,1\}$ and $\alpha$. For ProxGT-SR-O/E, we tune the step size $\alpha$ and the frequency of communicating the full local gradient, $q$. We find that $q=32$ yields the most stable results for a9a and MiniBooNE and $q=64$ performs the best for MNIST. For all these methods, $\alpha$ is tuned from $\{10.0, 5.0, 1.0, 0.1, 0.01, 0.005, 0.001\}$. We choose the Pareto optimal $\alpha$ that balances a small stationarity violation and a high test accuracy.

\section{Chebyshev acceleration}\label{appendix:chebyshev}

The Chebyshev mixing protocol~\cite{auzinger11} can be summarized in the following pseudo-code.

\begin{algorithm}[!h]
\caption{Chebyshev mixing protocol $\w_T(\B)$}
\label{algo:cheby}
\textbf{Input}: Mixing matrix $\W$, input $\B$, rounds $T$
\begin{algorithmic}[1] 
\STATE Let $\B_0=\B$ and $\B_1=\W\B_0$
\STATE Compute step sizes $\mu_{0}=1,\mu_{1}=\frac{1}{\rho}$
\FOR{$t=0,\dots,T-1$}
\STATE $\mu_{t+1}\gets\frac{2}{\rho}\mu_{t}-\mu_{t-1}$
\STATE $\B_{t+1}\gets\frac{2\mu_t}{\rho\mu_{t+1}}\W\B_t-\frac{\mu_{t-1}}{\mu_{t+1}}\B_{t-1}$
\ENDFOR
\end{algorithmic}
\textbf{Output}: $\B_T=\w_T(\B_0)$
\end{algorithm}

This method is accompanied with the following convergence result, relating the spectrum of $\w_T$ to the spectrum of $\W$. For a proof, see~\cite{mancino21}.

\begin{lemma}\label{lemma:chebyshev}
  The output of Algorithm~\ref{algo:cheby} can be denoted as $\B_T = \w_T\vec {B}$, where $\w_T$ is a degree-$T$ polynomial of $\W$ and it satisfies Assumptions~\ref{assumption:mixing_matrix}(ii)--(iv). Additionally, we use the bar notation to mean replacing each row of a matrix by the average of its rows; that is, $\bar{\B}=\frac{1}{N}\basis\basis^{\top}\B$. Then, $\bar{\B}_T=\bar{\B}$ for all $T$ and
  \begin{equation}\label{lemma:chebyshev:bound}
	\norm{\B_T-\bar{\B}}_F\le2\left(1-\sqrt{1-\rho}\right)^T\norm{\B-\bar{\B}}_F.
  \end{equation}
\end{lemma}
The analysis in Appendix~\ref{appendix:theory} uses a sufficiently large degree $T$ such that $\tilde{\rho}$ as defined in~\eqref{spectral_gap_cheby} is bounded by a constant, independent of the communication graph. For this, by Corollary 6.1 in~\cite{auzinger11}, it holds that
\begin{equation}\label{spectral_gap_cheby_bound}
  \tilde{\rho}\le2\left(1-\sqrt{1-\rho}\right)^T.
\end{equation}
Hence, by the proof of Theorem 4 in~\cite{mancino21}, we see that when $T=\lceil\frac{2}{\sqrt{1-\rho}}\rceil$, we obtain
\begin{equation}\label{appendix:cheby:upper_bound}
  (1-\tilde{\rho})^2 \ge \frac{1}{2}.
\end{equation}

\section{Convergence results}\label{appendix:theory}

We denote the \emph{global} objective function and the corresponding smooth part to be
	\begin{equation}\label{global_obj}
		\phi\triangleq\frac{1}{N}\sum_{i=1}^N\phi_i\quad\text{ and }\quad f\triangleq\frac{1}{N}\sum_{i=1}^Nf_i
	\end{equation}
respectively. Crucially, our analysis relies on bounding the difference between the local first-order estimators given in~\eqref{base_var_reduction} and the true local gradient; namely we define
	\begin{equation}\label{error_term}
		\vec{r}_i^{(k)}\triangleq\vec{d}_i^{(k)}-\nabla f_i(\blx_i^{(k)}).
	\end{equation}
Additionally, we define the following matrix terms to be used throughout the analysis,
	\begin{align}
		\bar{\vec{A}}&\defined\avg\vec{A},\enskip\forall\vec{A}\in\R^{N\times p},\label{average_matrix}\\
		\x_{\perp}&\defined\x-\bar{\x},\label{x_perp}\\
		\y_{\perp}&\defined\y-\bar{\y},\label{y_perp}\\
		\vec{R}&\defined\d-\nabla F(\x),\label{matrix_error}
	\end{align}
where $\nabla F$ is the gradient of the smooth part of the objective function written in the following matrix form
	\begin{equation}\label{objective_grad}
		\nabla F(\x)\defined\mymatrix{\nabla f_1(\vecx_1)\\\vdots\\\nabla f_N(\vecx_N)}\in\R^{N\times p}.
	\end{equation}
Before beginning with the analysis, we present two preparatory Lemmas. The first is standard in the literature~\cite{ghadimi16}.

\begin{lemma}\label{lemma:prox}
	Let $r:\R^{1\times p}\to\R$ be a closed, convex function, then for any $\vec{a},\vec{b}\in\R^{1\times p},$ it holds that
		\begin{equation}\label{lemma:prox:bound}
			\norm{\hbox{\normalfont prox}_r\left(\vec{a}\right)-\hbox{\normalfont prox}_r\left(\vec{b}\right)}_2\le\norm{\vec{a}-\vec{b}}_2.
		\end{equation}
\end{lemma}

\begin{lemma}\label{lemma:y_d_relation}
		For all $k\ge0$,
			\begin{equation}\label{lemma:y_d_relation:bound}
				\bar{\bly}^{(k)}=\bar{\vec{d}}^{(k)}.
			\end{equation}
\end{lemma}
\begin{proof}
	We proceed by induction. Notice $\w_T$ is a degree-$T$ polynomial of $\W$, so $\basis^\top\w_T=\basis^\top$ and thus $\bar{\bly}^{(0)}=\frac{1}{N}\basis^\top\y^{(0)}=\frac{1}{N}\basis^\top\w_T\d^{(0)}=\frac{1}{N}\basis^\top\d^{(0)}=\bar{\vec{d}}^{(0)}$. For $k\ge0$, we have
		\begin{align*}
			\bar{\bly}^{(k)}=\frac{1}{N}\basis^\top\y^{(k)}&\stack{\eqref{algo:y_update}}{=}\frac{1}{N}\basis^\top\w_T\left(\y^{(k-1)}+\d^{(k)}-\d^{(k-1)}\right)\\
			&=\frac{1}{N}\basis^\top\left(\y^{(k-1)}+\d^{(k)}-\d^{(k-1)}\right)\\
			&=\bar{\bly}^{(k-1)}+\bar{\vec{d}}^{(k)}-\bar{\vec{d}}^{(k-1)}\\
			&=\bar{\vec{d}}^{(k)}
		\end{align*}
	where in the last step we used the inductive hypothesis, $\bar{\bly}^{(k-1)}=\bar{\vec{d}}^{(k-1)}.$
\end{proof}

\subsection{Building blocks for constant and diminishing step size convergence.}

Our analysis begins by building a non-increasing Lyapunov function by relating changes in $\x$ and $\y$ to various quantities.

\begin{lemma}\label{lemma:r_change}
		For all $k\ge0$ and for all $i=1,\dots,N$,
			\begin{equation}\label{lemma:r_change:bound}
			\begin{split}
				&r(\xiknew)-r(\barxk)+\ip{\xiknew-\barxk,\bly_i^{(k)}}\\
				\le&-\frac{1}{2\alpha_k}\left(\norm{\xiknew-\barxk}_2^2+\norm{\xiknew-\vec{z}_i^{(k)}}_2^2-\norm{\barxk-\vec{z}_i^{(k)}}_2^2\right)
			\end{split}
			\end{equation}
\end{lemma}
\begin{proof}
	By~\eqref{algo:x_update}, we have
		\begin{align*}
			\vec{0}&\in\alpha_k\partial r(\xiknew)+\xiknew-\left(\vec{z}_i^{(k)}-\alpha_k\bly_i^{(k)}\right).
		\end{align*}
	Thus, for some $\tilde{\nabla}r(\xiknew)\in\partial r(\xiknew)$, and for any $\blx_i\in\R^{1\times p}$
		\begin{equation}\label{lemma:r_change_eqn1}
			\ip{\xiknew-\blx_i,\tilde{\nabla}r(\xiknew)+\frac{1}{\alpha_k}\left(\xiknew-\vec{z}_i^{(k)}\right)+\bly_i^{(k)}}=0.
		\end{equation}
	By the convexity of $r$, it holds for any $\blx_i\in\R^{1\times p}$,
		\begin{align*}
			&r(\xiknew)-r(\blx_i)+\ip{\xiknew-\barxk,\bly_i^{(k)}}\\
			\le&\ip{\xiknew-\blx_i,\tilde{\nabla}r(\xiknew)+\bly_i^{(k)}}\\
			\stack{\eqref{lemma:r_change_eqn1}}{=}&-\frac{1}{\alpha_k}\ip{\xiknew-\blx_i,\xiknew-\vec{z}_i^{(k)}}\\
			\stack{(a)}{=}&-\frac{1}{2\alpha_k}\left(\norm{\xiknew-\blx_i}_2^2+\norm{\xiknew-\vec{z}_i^{(k)}}_2^2-\norm{\blx_i-\vec{z}_i^{(k)}}_2^2\right),
		\end{align*}
	where (\textit{a}) follows from $\ip{\vec{a},\vec{b}}=\frac{1}{2}\left(\norm{\vec{a}}_2^2+\norm{\vec{b}}_2^2-\norm{\vec{a}-\vec{b}}_2^2\right)$. Letting $\blx_i=\barxk$ completes the proof.
\end{proof}

\begin{lemma}\label{lemma:first_change}
	For all $k\ge0,$
		\begin{equation}\label{lemma:first_change:bound}
			\begin{split}
				&\phi(\barxknew)-\phi(\barxk)\\
				\le&\frac{L}{2}\norm{\barxknew-\barxk}_2^2+\ip{\nabla f(\barxk),\barxknew-\barxk}-\frac{1}{N}\sum_{i=1}^N\ip{\xiknew-\barxk,\bly_i^{(k)}}\\
					&-\frac{1}{2N\alpha_k}\sum_{i=1}^N\left(\norm{\xiknew-\barxk}_2^2+\norm{\xiknew-\vec{z}_i^{(k)}}_2^2-\norm{\barxk-\vec{z}_i^{(k)}}_2^2\right),
			\end{split}
		\end{equation}
	where $\nabla f(\barxk)\triangleq\frac{1}{N}\sum_{i=1}^N\nabla f_i(\barxk)$ comes from~\eqref{global_obj}.
\end{lemma}
\begin{proof}
	From the $L$-smoothness of each $f_i$ and the convexity of $r$, we have
		\begin{align*}
			&\phi(\barxknew)-\phi(\barxk)\\
			=&\frac{1}{N}\sum_{i=1}^N\left(f_i(\barxknew)+r(\barxknew)\right)-\frac{1}{N}\sum_{i=1}^N\left(f_i(\barxk)+r(\barxk)\right)\\
			=&\frac{1}{N}\sum_{i=1}^N\left(f_i(\barxknew)-f_i(\barxk)\right)+r(\barxknew)-r(\barxk)\\
			\le&\frac{1}{N}\sum_{i=1}^N\left(f_i(\barxknew)-f_i(\barxk)\right)+\frac{1}{N}\sum_{i=1}^Nr(\blx_i^{(k+1)})-r(\barxk)\\
			\le&\frac{L}{2}\norm{\barxknew-\barxk}_2^2+\ip{\frac{1}{N}\sum_{i=1}^N\nabla f_i(\barxk),\barxknew-\barxk}+\frac{1}{N}\sum_{i=1}^Nr(\blx_i^{(k+1)})-r(\barxk)\\
			\stack{\eqref{lemma:r_change:bound}}{\le}&\enskip\frac{L}{2}\norm{\barxknew-\barxk}_2^2+\ip{\frac{1}{N}\sum_{i=1}^N\nabla f_i(\barxk),\barxknew-\barxk}-\frac{1}{N}\sum_{i=1}^N\ip{\xiknew-\barxk,\bly_i^{(k)}}\\
				&-\frac{1}{2N\alpha_k}\sum_{i=1}^N\left(\norm{\xiknew-\barxk}_2^2+\norm{\xiknew-\vec{z}_i^{(k)}}_2^2-\norm{\barxk-\vec{z}_i^{(k)}}_2^2\right).
		\end{align*}
	Utilizing~\eqref{global_obj} to have $\nabla f(\barxk)=\frac{1}{N}\sum_{i=1}^N\nabla f_i(\barxk)$ completes the proof.
\end{proof}

\begin{lemma}\label{lemma:first_inner_product}
	For all $k\ge0,$ the following equality holds,
		\begin{equation}\label{lemma:first_inner_product:bound}
			\begin{split}
				&\ip{\nabla f(\barxk),\barxknew-\barxk}-\frac{1}{N}\sum_{i=1}^N\ip{\xiknew-\barxk,\bly_i^{(k)}}\\
				=&\ip{\nabla f(\barxk)-\frac{1}{N}\sum_{i=1}^N\nabla f_i(\xik),\barxknew-\barxk}-\ip{\bar{\vec{r}}^{(k)},\barxknew-\barxk}\\
				&+\frac{1}{N}\sum_{i=1}^N\ip{\bar{\bly}^{(k)}-\bly_i^{(k)},\xiknew-\barxk},
			\end{split}
		\end{equation}
	where $\bar{\vec{r}}^{(k)}=\frac{1}{N}\sum_{i=1}^N\vec{r}_i^{(k)}$ for all $k.$
\end{lemma}
\begin{proof}
	We have,
		\begin{align}
			&\ip{\nabla f(\barxk),\barxknew-\barxk}-\frac{1}{N}\sum_{i=1}^N\ip{\xiknew-\barxk,\bly_i^{(k)}}\nonumber\\
			=&\ip{\nabla f(\barxk),\frac{1}{N}\sum_{i=1}^N\blx_i^{(k+1)}-\barxk}-\frac{1}{N}\sum_{i=1}^N\ip{\xiknew-\barxk,\bly_i^{(k)}}\nonumber\\
			\stack{(a)}{=}&\frac{1}{N}\sum_{i=1}^N\ip{\nabla f(\barxk),\blx_i^{(k+1)}-\barxk}-\frac{1}{N}\sum_{i=1}^N\ip{\xiknew-\barxk,\bly_i^{(k)}}\nonumber\\
			=&\frac{1}{N}\sum_{i=1}^N\ip{\nabla f(\barxk)-\bar{\bly}^{(k)}+\bar{\bly}^{(k)}-\bly_i^{(k)},\xiknew-\barxk}\nonumber\\
			\stack{(b)}{=}&\ip{\nabla f(\barxk)-\bar{\vec{d}}^{(k)},\barxknew-\barxk}+\frac{1}{N}\sum_{i=1}^N\ip{\bar{\bly}^{(k)}-\bly_i^{(k)},\xiknew-\barxk}\label{lemma:first_inner_product:eqn1}
		\end{align}
	where (\textit{a}) utilizes the linearity of the inner product and (\textit{b}) comes from Lemma~\ref{lemma:y_d_relation} in conjunction with the linearity of the inner product. Now,
		\begin{align}
			\ip{\nabla f(\barxk)-\bar{\vec{d}}^{(k)},\barxknew-\barxk}&=\ip{\nabla f(\barxk)-\frac{1}{N}\sum_{i=1}^N\nabla f_i(\xik),\barxknew-\barxk}\nonumber\\
			&+\ip{\frac{1}{N}\sum_{i=1}^N\nabla f_i(\xik)-\frac{1}{N}\sum_{i=1}^N\vec{d}_i^{(k)},\barxknew-\barxk}.\label{lemma:first_inner_product:eqn2}
		\end{align}
	 Plugging~\eqref{lemma:first_inner_product:eqn2} into~\eqref{lemma:first_inner_product:eqn1} and utilizing~\eqref{error_term} completes the proof.
\end{proof}

\begin{lemma}\label{lemma:second_change}
	For all $k\ge0,$ the following inequality holds,
		\begin{equation}\label{lemma:second_change:bound}
			\begin{split}
				\phi(\barxknew)-\phi(\barxk)\le&-\frac{1}{2N}\left(\frac{1}{\alpha_k}-3L\right)\fronorm{\x^{(k+1)}-\bar{\x}^{(k)}}-\ip{\bar{\vec{r}}^{(k)},\barxknew-\barxk}\\
					&+\frac{1}{2N\alpha_k}\fronorm{\bar{\x}^{(k)}-\z^{(k)}}-\frac{1}{2N\alpha_k}\fronorm{\x^{(k+1)}-\z^{(k)}}\\
					&+\frac{L}{2N}\fronorm{\x_\perp^{(k)}}+\frac{1}{2NL}\fronorm{\y_\perp^{(k)}}.
			\end{split}
		\end{equation}
\end{lemma}
\begin{proof}
	From~\eqref{lemma:first_change:bound}, we use~\eqref{lemma:first_inner_product:bound} to have
		\begin{align*}
			&\phi(\barxknew)-\phi(\barxk)\\
			\le&\frac{L}{2}\norm{\barxknew-\barxk}_2^2-\ip{\bar{\vec{r}}^{(k)},\barxknew-\barxk}\\
				&+\ip{\nabla f(\barxk)-\frac{1}{N}\sum_{i=1}^N\nabla f_i(\xik),\barxknew-\barxk}+\frac{1}{N}\sum_{i=1}^N\ip{\bar{\bly}^{(k)}-\bly_i^{(k)},\xiknew-\barxk}\\
				&-\frac{1}{2N\alpha_k}\sum_{i=1}^N\left(\norm{\xiknew-\barxk}_2^2+\norm{\xiknew-\vec{z}_i^{(k)}}_2^2-\norm{\barxk-\vec{z}_i^{(k)}}_2^2\right).
			\end{align*}
	We bound terms individually. By Jensen's inequality, we have
			\begin{align}\label{lemma:second_change:jensens}
				\norm{\barxknew-\barxk}_2^2\le\frac{1}{N}\sum_{i=1}^N\norm{\xiknew-\barxk}_2^2.
			\end{align}
	By the Peter-Paul inequality, we have
		\begin{align*}
			\frac{1}{N}\sum_{i=1}^N\ip{\bar{\bly}^{(k)}-\bly_i^{(k)},\xiknew-\barxk}\le&\frac{1}{N}\sum_{i=1}^N\left(\frac{L}{2}\norm{\xiknew-\barxk}_2^2+\frac{1}{2L}\norm{\bar{\bly}^{(k)}-\bly_i^{(k)}}_2^2\right)\\
		\end{align*}
	and
		\begin{align*}
			&\ip{\nabla f(\barxk)-\frac{1}{N}\sum_{i=1}^N\nabla f_i(\xik),\barxknew-\barxk}\\
			\le&\frac{L}{2}\norm{\barxknew-\barxk}_2^2+\frac{1}{2L}\norm{\nabla f(\barxk)-\frac{1}{N}\sum_{i=1}^N\nabla f_i(\xik)}_2^2\\
			\le&\frac{L}{2}\norm{\barxknew-\barxk}_2^2+\frac{1}{2NL}\sum_{i=1}^N\norm{\nabla f_i(\barxk)-\nabla f_i(\xik)}_2^2\\
			\stack{\eqref{assumption:smoothness},\eqref{lemma:second_change:jensens}}\le&\quad\frac{L}{2N}\sum_{i=1}^N\norm{\xiknew-\barxk}_2^2+\frac{L}{2N}\sum_{i=1}^N\norm{\xik-\barxk}_2^2,
		\end{align*}
	where the second inequality also uses Jensen's inequality. Combining like terms results in
		\begin{align*}
			&\phi(\barxknew)-\phi(\barxk)\\
			\le&-\frac{1}{2N}\left(\frac{1}{\alpha_k}-3L\right)\sum_{i=1}^N\norm{\xiknew-\barxk}_2^2-\ip{\bar{\vec{r}}^{(k)},\barxknew-\barxk}\\
				&+\frac{1}{2N\alpha_k}\sum_{i=1}^N\norm{\barxk-\vec{z}_i^{(k)}}_2^2-\frac{1}{2N\alpha_k}\sum_{i=1}^N\norm{\blx_i^{(k+1)}-\vec{z}_i^{(k)}}_2^2\\
				&+\frac{L}{2N}\sum_{i=1}^N\norm{\xik-\barxk}_2^2+\frac{1}{2NL}\sum_{i=1}^N\norm{\bar{\bly}^{(k)}-\bly_i^{(k)}}_2^2.
		\end{align*}
	We complete the proof by writing the summations of the 2-norms into the equivalent Frobenius norm expressions.
\end{proof}

\begin{lemma}\label{lemma:matrix_bounds}
	For all $k\ge0,$ the followings hold,
		\begin{align}
			\fronorm{\x_\perp^{(k+1)}}&\le\tilde{\rho}\fronorm{\x_\perp^{(k)}}+\frac{\alpha_k^2}{1-\tilde{\rho}}\fronorm{\y_{\perp}^{(k)}},\label{lemma:matrix_bounds:x}
		\end{align}
	and
		\begin{align}
			\Exp\fronorm{\y_\perp^{(k+1)}}&\le\tilde{\rho}\Exp\fronorm{\y_\perp^{(k)}}+\frac{1}{1-\tilde{\rho}}\left(8L^2\Exp\fronorm{\x^{(k+1)}-\x^{(k)}}+4\beta_k^2\Exp\fronorm{\vec{R}^{(k)}}+4N\beta_k^2\hat{\sigma}^2\right),\label{lemma:matrix_bounds:y}
		\end{align}
	where $\tilde{\rho}$ is defined in~\eqref{spectral_gap_cheby} and $\hat{\sigma}^2>0$ is defined in~\eqref{unbiased_assumption}.
\end{lemma}
\begin{proof}
	We first prove~\eqref{lemma:matrix_bounds:x}. First, we use the following identity
		\begin{align*}
			\hbox{prox}_{\alpha_kR}\left(\bar{\x}^{(k)}-\alpha_k\bar{\y}^{(k)}\right)&\triangleq\begin{bmatrix}
				\hbox{prox}_{\alpha_kr}\left(\barxk-\alpha_k\bar{\bly}^{(k)}\right) \\
				\vdots\\
				\hbox{prox}_{\alpha_kr}\left(\barxk-\alpha_k\bar{\bly}^{(k)}\right)
			\end{bmatrix}\in\R^{N\times p}\\
			&=\avg\hbox{prox}_{\alpha_kR}\left(\bar{\x}^{(k)}-\alpha_k\bar{\y}^{(k)}\right),
		\end{align*}
	since each row is identical. Then by~\eqref{algo:x_update} we have
		\begin{align*}
			\fronorm{\x_\perp^{(k+1)}}=&\fronorm{\hbox{prox}_{\alpha_kR}\left(\w_T(\x^{(k)})-\alpha_k\y^{(k)}\right)-\avg\hbox{prox}_{\alpha_kR}\left(\w_T(\x^{(k)})-\alpha_k\y^{(k)}\right)}\\
			\stack{(a)}{=}&\fronorm{\hbox{prox}_{\alpha_kR}\left(\w_T(\x^{(k)})-\alpha_k\y^{(k)}\right)-\hbox{prox}_{\alpha_kR}\left(\bar{\x}^{(k)}-\alpha_k\bar{\y}^{(k)}\right)}\\
				-&\fronorm{\avg\left(\hbox{prox}_{\alpha_kR}\left(\bar{\x}^{(k)}-\alpha_k\bar{\y}^{(k)}\right)-\hbox{prox}_{\alpha_kR}\left(\w_T(\x^{(k)})-\alpha_k\y^{(k)}\right)\right)}\\
			\stack{\eqref{lemma:prox:bound}}{\le}&\enskip\fronorm{\w_T(\x^{(k)})-\bar{\x}^{(k)}-\alpha_k\left(\y^{(k)}-\bar{\y}^{(k)}\right)}\\
			=&\fronorm{\w_T(\x^{(k)})-\bar{\x}^{(k)}}+\fronorm{\alpha_k\left(\y^{(k)}-\bar{\y}^{(k)}\right)}\\
				&-2\ip{\w_T(\x^{(k)})-\bar{\x}^{(k)},\alpha_k\left(\y^{(k)}-\bar{\y}^{(k)}\right)}\\
			\stack{(b)}{\le}&\fronorm{\w_T(\x^{(k)})-\bar{\x}^{(k)}}+\fronorm{\alpha_k\left(\y^{(k)}-\bar{\y}^{(k)}\right)}\\
				&+\delta\fronorm{\w_T(\x^{(k)})-\bar{\x}^{(k)}}+\frac{1}{\delta}\fronorm{\alpha_k\left(\y^{(k)}-\bar{\y}^{(k)}\right)}\\
			\stack{(c)}{=}&(1+\delta)\fronorm{\left(\w_T-\avg\right)\x_\perp^{(k)}}+(1+\frac{1}{\delta})\alpha_k^2\fronorm{\y_\perp^{(k)}},
		\end{align*}
	where (\textit{a}) uses that $\avg$ is a projection operator to have, for any matrix $\vec{A}\in\R^{N\times p},$
		\begin{align*}
			\fronorm{\vec{A}-\avg\vec{A}}=&\fronorm{\vec{A}}-2\ip{\vec{A},\avg\vec{A}}+\fronorm{\avg\vec{A}}\\
			=&\fronorm{\vec{A}}-2\fronorm{\avg\vec{A}}+\fronorm{\avg\vec{A}}\\
			=&\fronorm{\vec{A}}-\fronorm{\avg\vec{A}},
		\end{align*}
	(\textit{b}) uses the Peter-Paul inequality with $\delta>0$, and (\textit{c}) uses $\left(\w_T-\avg\right)=\left(\w_T-\avg\right)\left(\identity-\avg\right)$.  Choosing $\delta=\frac{1-\tilde{\rho}}{\tilde{\rho}}$ with $\tilde{\rho}$ defined in~\eqref{spectral_gap_cheby} and using the compatibility of the Frobenius norm and the 2-norm to have
		\begin{align*}
			\fronorm{\left(\w_T-\avg\right)\x_\perp^{(k)}}\le\norm{\w_T-\avg}_2^2\fronorm{\x_\perp^{(k)}}\stack{\eqref{spectral_gap_cheby}}=\tilde{\rho}^2\fronorm{\x_\perp^{(k)}}
		\end{align*}
	yields~\eqref{lemma:matrix_bounds:x}.\\
	To prove~\eqref{lemma:matrix_bounds:y}, we use Assumption~\ref{assumption:mixing_matrix} parts (\textit{ii}) and (\textit{iii}) to have
		\begin{align*}
			&\fronorm{\y_\perp^{(k+1)}}\\
			\stack{\eqref{algo:y_update}}{=}&\fronorm{\w_T\left(\y^{(k)}+\d^{(k+1)}-\d^{(k)}\right)-\avg\left(\y^{(k)}+\d^{(k+1)}-\d^{(k)}\right)}\\
			\le&(1+c_1)\fronorm{\left(\w_T-\avg\right)\y^{(k)}}+(1+\frac{1}{c_1})\fronorm{\left(\w_T-\avg\right)\left(\d^{(k+1)}-\d^{(k)}\right)}\\
			\stack{(a)}{\le}&(1+c_1)\tilde{\rho}^2\fronorm{\y_\perp^{(k)}}+(1+\frac{1}{c_1})\tilde{\rho}^2\fronorm{\left(\identity-\avg\right)\d^{(k+1)}-\d^{(k)}}\\
			\stack{(b)}{\le}&(1+c_1)\tilde{\rho}^2\fronorm{\y_\perp^{(k)}}+(1+\frac{1}{c_1})\fronorm{\d^{(k+1)}-\d^{(k)}},
		\end{align*}
	where (\textit{a}) utilizes $\left(\w-\avg\right)=\left(\w-\avg\right)\left(\identity-\avg\right)$ coupled with part \textit{(iv)} of Assumption~\ref{assumption:mixing_matrix} and (\textit{b}) uses $\tilde{\rho}^2<1$ and $\norm{\identity-\avg}_2\le1.$ Next, by Young's inequality we have
		\begin{align}
			&\fronorm{\d^{(k+1)}-\d^{(k)}}\nonumber\\
			=&\fronorm{(1-\beta_k)\left(\vec{V}^{(k+1)}-\vec{U}^{(k+1)}\right)+\beta_k\left(\tilde{\vec{V}}^{(k+1)}-\d^{(k)}\right)}\nonumber\\
			\le&4\left((1-\beta_k)^2\fronorm{\vec{V}^{(k+1)}-\vec{U}^{(k+1)}}+\beta_k^2\fronorm{\vec{R}^{(k)}}\right)\nonumber\\
					&+4\left(\beta_k^2\fronorm{\tilde{\vec{V}}^{(k+1)}-\nabla F(\x^{(k+1)})}+\beta_k^2\fronorm{\nabla F(\x^{(k+1)})-\nabla F(\x^{(k)})}\right)\nonumber\\
			\le&4\left(\fronorm{\vec{V}^{(k+1)}-\vec{U}^{(k+1)}}+\beta_k^2\fronorm{\vec{R}^{(k)}}\right)\nonumber\\
					&+4\left(\beta_k^2\fronorm{\tilde{\vec{V}}^{(k+1)}-\nabla F(\x^{(k+1)})}+\fronorm{\nabla F(\x^{(k+1)})-\nabla F(\x^{(k)})}\right),\label{lemma:matrix_bounds:d_diff}
		\end{align}
	where the last inequality comes from the assumption that $\beta_k\in(0,1).$ Hence we have
		\begin{align*}
			&\fronorm{\y_\perp^{(k+1)}}\\
			\stack{\eqref{lemma:matrix_bounds:d_diff}}{\le}&\enskip(1+c_1)\tilde{\rho}^2\fronorm{\y_\perp^{(k)}}+4(1+\frac{1}{c_1})\left(\fronorm{\vec{V}^{(k+1)}-\vec{U}^{(k+1)}}+\beta_k^2\fronorm{\vec{R}^{(k)}}\right)\\
					&+4(1+\frac{1}{c_1})\left(\beta_k^2\fronorm{\tilde{\vec{V}}^{(k+1)}-\nabla F(\x^{(k+1)})}+\fronorm{\nabla F(\x^{(k+1)})-\nabla F(\x^{(k)})}\right).\\
		\end{align*}
	Letting $c_1=\frac{1}{\tilde{\rho}}-1>0$ and then first taking the expectation with respect to the samples and utilizing~\eqref{assumption:smoothness} and~\eqref{unbiased_assumption} on the above two inequalities and then taking the full expectation, completes the proof.
\end{proof}

Our analysis relies on bounding the gradient error term defined in~\eqref{matrix_error}. Hence, we present the following two Lemmas which define a recursive error bound given either~\eqref{unbiased_assumption} or~\eqref{v2_update} holds for the unbiased estimator $\tilde{\vec{v}}_i^{(k+1)}$ in~\eqref{base_var_reduction}.

\begin{lemma}\label{lemma:hsgd_error}
	Suppose $\{\vec{d}_i^{(t)}\}_{t=0}^{(k)}$ is updated by~\eqref{base_var_reduction} such that $\tilde{\vec{v}}_i$ satisfies~\eqref{unbiased_assumption} for all iterates $t=0,\dots,k$, for each agent $i=1,\dots,N$. Then at iteration $k+1$, the following bound holds
		\begin{equation}\label{lemma:hsgd_error:bound}
			\Exp\fronorm{\vec{R}^{(k+1)}}\le N\beta_k^2\hat{\sigma}^2+(1-\beta_k)^2L^2\Exp\fronorm{\x^{(k+1)}-\x^{(k)}}+(1-\beta_k)^2\Exp\fronorm{\vec{R}^{(k)}}.
		\end{equation}
\end{lemma}
\begin{proof}
	The proof follows the same logic as the proof of Lemmas 3 and 4 in~\cite{tran22}, but is included here for the sake of completeness. For sake of brevity, define $B\triangleq B_i^{(k+1)}$ and $\tilde{B}\triangleq \tilde{B}_i^{(k+1)}$. Then for each agent $i$, by~\eqref{unbiased_assumption} and the definition of $\vec{r}_i^{(k+1)}$ in~\eqref{error_term}, it holds that
		\begin{align}
			&\Exp_{(B,\tilde{B})}\norm{\vec{r}_i^{(k+1)}}_2^2\nonumber\\
			\stack{\eqref{base_var_reduction}}{=}&\enskip\Exp_{(B,\tilde{B})}\norm{(1-\beta_k)\vec{d}_i^{(k)}+(1-\beta_k)\left(\vec{v}_i^{(k+1)}-\vec{u}_i^{(k+1)}\right)+\beta_k\tilde{\vec{v}}_i^{(k+1)}-\nabla f_i(\blx_i^{(k+1)})}_2^2\nonumber\\
			=&\Exp_{(B,\tilde{B})}\left((1-\beta_k)^2\norm{\vec{d}_i^{(k)}-\nabla f_i(\blx_i^{(k)})}_2^2+\beta_k^2\norm{\tilde{\vec{v}}_i^{(k+1)}-\nabla f_i(\blx_i^{(k+1)})}_2^2\right)\nonumber\\
				&+\Exp_{(B,\tilde{B})}\left((1-\beta_k)^2\norm{\left(\vec{v}_i^{(k+1)}-\nabla f_i(\blx_i^{(k+1)})\right)-\left(\vec{u}_i^{(k+1)}-\nabla f_i(\blx_i^{(k)})\right)}_2^2\right)\nonumber\\
				&+2(1-\beta_k)^2\Exp_{(B,\tilde{B})}\ip{\vec{d}_i^{(k)}-\nabla f_i(\blx_i^{(k)}),\left(\vec{v}_i^{(k+1)}-\nabla f_i(\blx_i^{(k+1)})\right)-\left(\vec{u}_i^{(k+1)}-\nabla f_i(\blx_i^{(k)})\right)}\nonumber\\
				&+2\beta_k(1-\beta_k)\Exp_{(B,\tilde{B})}\ip{\vec{d}_i^{(k)}-\nabla f_i(\blx_i^{(k)}),\tilde{\vec{v}}_i^{(k+1)}-\nabla f_i(\blx_i^{(k+1)})}\nonumber\\
				&+2\beta_k(1-\beta_k)\Exp_{(B,\tilde{B})}\ip{\tilde{\vec{v}}_i^{(k+1)}-\nabla f_i(\blx_i^{(k+1)}),\left(\vec{v}_i^{(k+1)}-\nabla f_i(\blx_i^{(k+1)})\right)-\left(\vec{u}_i^{(k+1)}-\nabla f_i(\blx_i^{(k)})\right)},\nonumber\\\label{lemma:hsgd_error:eqn1}
		\end{align}
	where the second equality comes from adding and subtracting $\beta_k\nabla f_i(\blx_i^{(k+1)})$ and $(1-\beta_k)\nabla f_i(\blx_i^{(k)})$ and expanding the norm squared. The first two inner products evaluate to zero by the unbiasedness in Assumption~\ref{assumption:objective_function} (\textit{iv}). Next, since all $\xi\in B$ are independent from all $\tilde{\xi}\in\tilde{B}$, it holds by $\Exp_{(B,\tilde{B})}[\cdot]=\Exp_{\tilde{B}}\left[\Exp_{B}\left[\cdot\big|\tilde{B}\right]\right]$ that the final inner product is zero. Using the unbiasedness assumption of $\vec{v}_i^{(k+1)}$ and $\vec{u}_i^{(k+1)}$, we have
		\begin{align}
				&\Exp_{(B,\tilde{B})}\left((1-\beta_k)^2\norm{\left(\vec{v}_i^{(k+1)}-\nabla f_i(\blx_i^{(k+1)})\right)-\left(\vec{u}_i^{(k+1)}-\nabla f_i(\blx_i^{(k)})\right)}_2^2\right)\nonumber\\
				=&\Exp_{(B,\tilde{B})}(1-\beta_k)^2\left(\norm{\vec{v}_i^{(k+1)}-\vec{u}_i^{(k+1)}}_2^2+\norm{\nabla f_i(\blx_i^{(k+1)})-\nabla f_i(\blx_i^{(k)})}_2^2\right)\nonumber\\
						&-\Exp_{(B,\tilde{B})}\left[2(1-\beta_k)^2\ip{\vec{v}_i^{(k+1)}-\vec{u}_i^{(k+1)},\nabla f_i(\blx_i^{(k+1)})-\nabla f_i(\blx_i^{(k)})}\right]\nonumber\\
				=&(1-\beta_k)^2\left(\Exp_{B}\norm{\vec{v}_i^{(k+1)}-\vec{u}_i^{(k+1)}}_2^2-\norm{\nabla f_i(\blx_i^{(k+1)})-\nabla f_i(\blx_i^{(k)})}_2^2\right)\nonumber\\
				\le&(1-\beta_k)^2\Exp_{B}\norm{\vec{v}_i^{(k+1)}-\vec{u}_i^{(k+1)}}_2^2.\label{lemma:hsgd_error:total_exp}
		\end{align}
	Summing over the agents $i=1,\dots,N$, utilizing~\eqref{assumption:smoothness},~\eqref{assumption:variance},~\eqref{unbiased_assumption}, and taking the full expectation completes the proof.
\end{proof}

\begin{lemma}\label{lemma:storm_error}
	Suppose $\{\vec{d}_i^{(t)}\}_{t=0}^{(k)}$ is updated by~\eqref{base_var_reduction} with~\eqref{v2_update} for each agent $i=1,\dots,N$. Then at iteration $k+1$, the following bound holds
		\begin{equation}\label{lemma:storm_error:bound}
			\Exp\fronorm{\vec{R}^{(k+1)}}\le2N\beta_k^2\hat{\sigma}^2+2(1-\beta_k)^2L^2\Exp\fronorm{\x^{(k+1)}-\x^{(k)}}+(1-\beta_k)^2\Exp\fronorm{\vec{R}^{(k)}}.
		\end{equation}
\end{lemma}
\begin{proof}
	The proof follows from Lemma 2 of~\cite{xu20}, but is included here for sake of completeness. Using~\eqref{base_var_reduction} with~\eqref{v2_update} and defining $B\triangleq B_i^{(k+1)}$, for each agent $i$ we have
		\begin{align*}
			&\Exp_{B}\left[\norm{\vec{r}_i^{(k+1)}}^2\right]\\
			&=\Exp_{B}\left[\norm{\vec{v}_i^{(k+1)}-\nabla f_i(\xiknew)+(1-\beta_k)\left(\nabla f_i(\xik)-\vec{u}_i^{(k+1)}\right)+(1-\beta_k)\vec{r}_i^{(k)}}_2^2\right]\\
			&=\Exp_{B}\left[\norm{\vec{v}_i^{(k+1)}-\nabla f_i(\xiknew)+(1-\beta_k)\left(\nabla f_i(\xik)-\vec{u}_i^{(k+1)}\right)}_2^2\right]+(1-\beta_k)^2\norm{\vec{r}_i^{(k)}}_2^2,
		\end{align*}
	where we have used
		\begin{align*}
			\Exp_{B}\left[\ip{\vec{v}_i^{(k+1)}-\nabla f_i(\xiknew),\vec{r}_i^{(k)}}\right]=0\text{ and }\Exp_{B}\left[\ip{\vec{u}_i^{(k+1)}-\nabla f_i(\xik),\vec{r}_i^{(k)}}\right]=0
		\end{align*}
	by the definition of the $\vec{v}_i$ and $\vec{u}_i$ and the unbiasedness in Assumption~\ref{assumption:objective_function} (\textit{iv}). Adding and subtracting $\beta_k\left(\vec{v}_i^{(k+1)}-\nabla f_i(\xiknew)\right)$ inside of the norm of the first term and using Young's inequality results in
		\begin{align*}
					\Exp_{B}\left[\norm{\vec{r}_i^{(k+1)}}^2\right]&\le2\beta_k^2\Exp_{B}\norm{\vec{v}_i^{(k+1)}-\nabla f_i(\xiknew)}_2^2+(1-\beta_k)^2\norm{\vec{r}_i^{(k)}}^2\\
						&\quad+2(1-\beta_k)^2\Exp_{B}\norm{\left(\vec{v}_i^{(k+1)}-\nabla f_i(\xiknew)\right)+\left(\nabla f_i(\xik)-\vec{u}_i^{(k+1)}\right)}_2^2.
		\end{align*}
	Applying~\eqref{lemma:hsgd_error:total_exp} to the last term above, summing over all agents $i=1,\dots,N$, utilizing both~\eqref{assumption:smoothness} and~\eqref{v2_update}, and taking the full expectation completes the proof.
\end{proof}

\begin{lemma}\label{lemma:avg_error}
	Suppose $\{\vec{d}_i^{(t)}\}_{t=0}^{(k)}$ is updated by~\eqref{base_var_reduction} such that $\tilde{\vec{v}}_i$ satisfies~\eqref{unbiased_assumption} for each agent $i=1,\dots,N$. Then at iteration $k+1$, the following bound holds
		\begin{equation}\label{lemma:avg_error:bound}
			\Exp\norm{\bar{\vec{r}}^{(k+1)}}_2^2\le(1-\beta_k)^2\Exp\norm{\bar{\vec{r}}^{(k)}}_2^2+\frac{(1-\beta_k)^2L^2}{N^2}\Exp\fronorm{\x^{(k+1)}-\x^{(k)}}+\frac{\beta_k^2\hat{\sigma}^2}{N}.
		\end{equation}
\end{lemma}
\begin{proof}
	The proof follows from Lemma 3 of~\cite{xin21hsgd}, but is included here for sake of completeness. Following similar notation to the proof of Lemma~\ref{lemma:hsgd_error}, by~\eqref{base_var_reduction}, it holds that
		\begin{align*}
			\vec{r}_i^{(k+1)}\stack{\eqref{error_term}}{=}&\enskip(1-\beta_k)\left(\vec{d}_i^{(k)}+\vec{v}_i^{(k+1)}-\vec{u}_i^{(k+1)}\right)+\beta_k\tilde{\vec{v}}_i^{(k+1)}-\nabla f_i(\blx_i^{(k+1)})\\
			=&\beta_k\left(\tilde{\vec{v}}_i^{(k+1)}-\nabla f_i(\blx_i^{(k+1)})\right)+(1-\beta_k)\left(\vec{d}_i^{(k)}-\nabla f_i(\blx_i^{(k)})+\vec{v}_i^{(k+1)}-\vec{u}_i^{(k+1)}\right)\\
				&+(1-\beta_k)\left(\nabla f_i(\blx_i^{(k)})-\nabla f_i(\blx_i^{(k+1)})\right).\\
		\end{align*}
	Taking the average results in
		\begin{align}
			\bar{\vec{r}}^{(k+1)}=&(1-\beta_k)\bar{\vec{r}}^{(k)}+\frac{\beta_k}{N}\sum_{i=1}^N\left(\tilde{\vec{v}}_i^{(k+1)}-\nabla f_i(\blx_i^{(k+1)})\right)\nonumber\\
				&+\frac{(1-\beta_k)}{N}\sum_{i=1}^N\left(\vec{v}_i^{(k+1)}-\vec{u}_i^{(k+1)}+\nabla f_i(\blx_i^{(k)})-\nabla f_i(\blx_i^{(k+1)})\right).\label{lemma:avg_error:eqn1}
		\end{align}
	Defining $B\triangleq\left\{\xi:\xi\in\bigcup_{i=1}^NB_i^{(k+1)}\right\}$ and $\tilde{B}\triangleq\left\{\xi:\xi\in\bigcup_{i=1}^N\tilde{B}_i^{(k+1)}\right\}$ we take the norm squared and compute $\Exp_{(B,\tilde{B})}$, resulting in
		\begin{align}
					&\Exp_{(B,\tilde{B})}\norm{\bar{\vec{r}}^{(k+1)}}_2^2\nonumber\\
					=&\Exp_{(B,\tilde{B})}\left((1-\beta_k)^2\norm{\bar{\vec{r}}^{(k)}}_2^2+\frac{\beta_k^2}{N^2}\norm{\sum_{i=1}^N\left(\tilde{\vec{v}}_i^{(k+1)}-\nabla f_i(\blx_i^{(k+1)})\right)}_2^2\right)\nonumber\\
						&+\Exp_{(B,\tilde{B})}\left(\frac{(1-\beta_k)^2}{N^2}\norm{\sum_{i=1}^N\left(\left(\vec{v}_i^{(k+1)}-\nabla f_i(\blx_i^{(k+1)})\right)-\left(\vec{u}_i^{(k+1)}-\nabla f_i(\blx_i^{(k)})\right)\right)}_2^2\right)\nonumber\\
						&+\frac{2(1-\beta_k)^2}{N}\Exp_{(B,\tilde{B})}\sum_{i=1}^N\ip{\bar{\vec{r}}^{(k)},\left(\vec{v}_i^{(k+1)}-\nabla f_i(\blx_i^{(k+1)})\right)-\left(\vec{u}_i^{(k+1)}-\nabla f_i(\blx_i^{(k)})\right)}\nonumber\\
						&+\frac{2\beta_k(1-\beta_k)}{N}\Exp_{(B,\tilde{B})}\sum_{i=1}^N\ip{\bar{\vec{r}}^{(k)},\tilde{\vec{v}}_i^{(k+1)}-\nabla f_i(\blx_i^{(k+1)})}\nonumber\\
						&+\frac{2\beta_k(1-\beta_k)}{N^2}\Exp_{(B,\tilde{B})}\sum_{i=1}^N\sum_{j=1}^N\ip{\tilde{\vec{v}}_i^{(k+1)}-\nabla f_i(\blx_i^{(k+1)}),\left(\vec{v}_j^{(k+1)}-\nabla f_j(\blx_j^{(k+1)})\right)-\left(\vec{u}_j^{(k+1)}-\nabla f_j(\blx_j^{(k)})\right)}.\label{lemma:avg_error:eqn2}
		\end{align}
	Similar to the proof of Lemma~\ref{lemma:hsgd_error}, we have the first two inner products evaluate to zero by the unbiasedness in Assumption~\ref{assumption:objective_function} (\textit{iv}). Next, since all $\xi\in B$ are independent from all $\tilde{\xi}\in\tilde{B}$, it holds by $\Exp_{(B,\tilde{B})}[\cdot]=\Exp_{\tilde{B}}\left[\Exp_{B}\left[\cdot\big|\tilde{B}\right]\right]$ that the final inner product is zero. Define
		\begin{align*}
		\hat{\nabla}_i^{(k+1)}\triangleq\left(\vec{v}_i^{(k+1)}-\nabla f_i(\blx_i^{(k+1)})\right)-\left(\vec{u}_i^{(k+1)}-\nabla f_i(\blx_i^{(k)})\right)
		\end{align*}
	for all $i=1,\dots,N$. Hence, by using $\xi\in B$ are independent from $\tilde{\xi}\in\tilde{B}$,
		\begin{align}
			&\Exp_{(B,\tilde{B})}\left(\frac{(1-\beta_k)^2}{N^2}\norm{\sum_{i=1}^N\left(\left(\vec{v}_i^{(k+1)}-\nabla f_i(\blx_i^{(k+1)})\right)-\left(\vec{u}_i^{(k+1)}-\nabla f_i(\blx_i^{(k)})\right)\right)}_2^2\right)\nonumber\\
			=&\Exp_{(B,\tilde{B})}\left(\frac{(1-\beta_k)^2}{N^2}\sum_{i=1}^N\norm{\left(\left(\vec{v}_i^{(k+1)}-\nabla f_i(\blx_i^{(k+1)})\right)-\left(\vec{u}_i^{(k+1)}-\nabla f_i(\blx_i^{(k)})\right)\right)}_2^2\right)\nonumber\\
				&+\frac{(1-\beta_k)^2}{N^2}\sum_{i\ne j}\Exp_{(B,\tilde{B})}\ip{\hat{\nabla}_i^{(k+1)},\hat{\nabla}_j^{(k+1)}}\nonumber\\
			=&\Exp_{(B,\tilde{B})}\left(\frac{(1-\beta_k)^2}{N^2}\sum_{i=1}^N\norm{\left(\left(\vec{v}_i^{(k+1)}-\nabla f_i(\blx_i^{(k+1)})\right)-\left(\vec{u}_i^{(k+1)}-\nabla f_i(\blx_i^{(k)})\right)\right)}_2^2\right)\nonumber\\
			\stack{\eqref{lemma:hsgd_error:total_exp}}{\le}&\enskip\frac{(1-\beta_k)^2}{N^2}\sum_{i=1}^N\Exp_B\norm{\vec{v}_i^{(k+1)}-\vec{u}_i^{(k+1)}}_2^2.\label{lemma:avg_error:eqn3}
		\end{align}
	By similar logic,
		\begin{align}
			&\Exp_{\tilde{B}}\frac{\beta_k^2}{N^2}\norm{\sum_{i=1}^N\left(\tilde{\vec{v}}_i^{(k+1)}-\nabla f_i(\blx_i^{(k+1)})\right)}_2^2\nonumber\\
			=&\frac{\beta_k^2}{N^2}\sum_{i=1}^N\Exp_{\tilde{B}}\norm{\left(\tilde{\vec{v}}_i^{(k+1)}-\nabla f_i(\blx_i^{(k+1)})\right)}_2^2\nonumber\\
				&+\frac{\beta_k^2}{N^2}\sum_{i\ne j}\Exp_{\tilde{B}}\ip{\tilde{\vec{v}}_i^{(k+1)}-\nabla f_i(\blx_i^{(k+1)}),\tilde{\vec{v}}_j^{(k+1)}-\nabla f_j(\blx_j^{(k+1)})}\nonumber\\
			=&\frac{\beta_k^2}{N^2}\sum_{i=1}^N\Exp_{\tilde{B}}\norm{\left(\tilde{\vec{v}}_i^{(k+1)}-\nabla f_i(\blx_i^{(k+1)})\right)}_2^2.\label{lemma:avg_error:eqn4}
		\end{align}
	Plugging~\eqref{lemma:avg_error:eqn3} to~\eqref{lemma:avg_error:eqn4} into~\eqref{lemma:avg_error:eqn2} and taking the full expectation yields
		\begin{align*}
			&\Exp\norm{\bar{\vec{r}}^{(k+1)}}_2^2\\
			\le&(1-\beta_k)^2\Exp\norm{\bar{\vec{r}}^{(k)}}_2^2+\frac{(1-\beta_k)^2}{N^2}\sum_{i=1}^N\Exp\norm{\vec{v}_i^{(k+1)}-\vec{u}_i^{(k+1)}}_2^2+\frac{\beta_k^2\hat{\sigma}^2}{N}\\
			\stack{\eqref{assumption:smoothness}}{\le}&(1-\beta_k)^2\Exp\norm{\bar{\vec{r}}^{(k)}}_2^2+\frac{(1-\beta_k)^2L^2}{N^2}\Exp\fronorm{\x^{(k+1)}-\x^{(k)}}+\frac{\beta_k^2\hat{\sigma}^2}{N},
		\end{align*}
	where we have used~\eqref{assumption:variance} and the equivalence of the Frobenius norm to the sum of the squared 2-norms. This completes the proof.
\end{proof}

\begin{lemma}\label{lemma:avg_error_storm}
	Suppose $\{\vec{d}_i^{(t)}\}_{t=0}^{(k)}$ is updated by~\eqref{base_var_reduction} such that $\tilde{\vec{v}}_i$ satisfies~\eqref{v2_update} for each agent $i=1,\dots,N$. Then at iteration $k+1$, the following bound holds
		\begin{equation}\label{lemma:avg_error:bound_storm}
			\Exp\norm{\bar{\vec{r}}^{(k+1)}}_2^2\le(1-\beta_k)^2\Exp\norm{\bar{\vec{r}}^{(k)}}_2^2+\frac{2(1-\beta_k)^2L^2}{N^2}\Exp\fronorm{\x^{(k+1)}-\x^{(k)}}+\frac{2\beta_k^2\hat{\sigma}^2}{N}.
		\end{equation}
\end{lemma}
\begin{proof}
	The proof follows from Lemma 3 of~\cite{xin21hsgd}, but is included here for sake of completeness. Following similar notation as the proof of Lemma~\ref{lemma:avg_error}, by~\eqref{v2_update}, we define $B\triangleq\left\{\xi:\xi\in\bigcup_{i=1}^NB_i^{(k+1)}\right\}$ to have
				\begin{align*}
					&\Exp_{B}\left[\norm{\bar{\vec{r}}^{(k+1)}}_2^2\right]\\
					&\stack{\eqref{base_var_reduction}}{=}\Exp_{B}\left[\norm{\frac{1}{N}\sum_{i=1}^N\left(\vec{v}_i^{(k+1)}-\nabla f_i(\xiknew)+(1-\beta_k)\left(\nabla f_i(\xik)-\vec{u}_i^{(k+1)}\right)\right)+(1-\beta_k)\bar{\vec{r}}^{(k)}}_2^2\right]\\
					&=\Exp_{B}\left[\norm{\frac{1}{N}\sum_{i=1}^N\left(\vec{v}_i^{(k+1)}-\nabla f_i(\xiknew)+(1-\beta_k)\left(\nabla f_i(\xik)-\vec{u}_i^{(k+1)}\right)\right)}_2^2\right]+(1-\beta_k)^2\norm{\bar{\vec{r}}^{(k)}}_2^2
				\end{align*}
			where we have used, for all $i=1,\dots,N$,
				\begin{align*}
					\Exp_{B}\left[\ip{\vec{v}_i^{(k+1)}-\nabla f_i(\xiknew),\bar{\vec{r}}^{(k)}}\right]=0\text{ and }\Exp_{B}\left[\ip{\vec{u}_i^{(k+1)}-\nabla f_i(\xik),\bar{\vec{r}}^{(k)}}\right]=0
				\end{align*}
			by the definition of the $\vec{v}_i$ and $\vec{u}_i$ and the unbiasedness in Assumption~\ref{assumption:objective_function} (\textit{iv}). Adding and subtracting $\frac{\beta_k}{N}\sum_{i=1}^N\left(\vec{v}_i^{(k+1)}-\nabla f_i(\xiknew)\right)$ inside of the norm of the first term and using Young's inequality results in
				\begin{align*}
							\Exp_{B}\left[\norm{\bar{\vec{r}}^{(k+1)}}_2^2\right]\le&\frac{2\beta_k^2}{N^2}\Exp_{B}\norm{\sum_{i=1}^N\left(\vec{v}_i^{(k+1)}-\nabla f_i(\xiknew)\right)}_2^2+(1-\beta_k)^2\norm{\bar{\vec{r}}^{(k)}}_2^2\\
								&+\frac{2(1-\beta_k)^2}{N^2}\Exp_{B}\norm{\sum_{i=1}^N\left(\left(\vec{v}_i^{(k+1)}-\nabla f_i(\xiknew)\right)+\left(\nabla f_i(\xik)-\vec{u}_i^{(k+1)}\right)\right)}_2^2.
				\end{align*}
			Applying~\eqref{lemma:avg_error:eqn3} and~\eqref{lemma:avg_error:eqn4}, utilizing both~\eqref{assumption:smoothness} and~\eqref{v2_update}, and taking the full expectation completes the proof.
\end{proof}

Before presenting our convergence results, we give the following Lemma which relates relevant terms to a stochastic $\varepsilon$-stationary point as defined in Definition~\ref{def:stationarity}.

\begin{lemma}\label{lemma:stationary_bound}
	For all $k\ge0,$ the following bound holds,
		\begin{equation}\label{lemma:stationary_bound:bound}
			\begin{split}
				&\frac{1}{N}\sum_{i=1}^N\norm{P\left(\vec{z}_i^{(k)},\nabla f(\vec{z}_i^{(k)}),\alpha_k\right)}_2^2+\frac{L^2}{N}\fronorm{\x_\perp^{(k)}}+\frac{L^2}{N}\fronorm{\z_\perp^{(k)}}\\
						\le&\frac{6}{N}\fronorm{\y_\perp^{(k)}}+6\norm{\bar{\vec{r}}^{(k)}}_2^2+\frac{32L^2}{N}\fronorm{\x_\perp^{(k)}}+\frac{2}{N\alpha_k^2}\fronorm{\x^{(k+1)}-\z^{(k)}}.
			\end{split}
		\end{equation}
\end{lemma}
\begin{proof}
	Notice that by~\eqref{def:prox_grad_map:eqn}, for all $i=1,\dots,N$, we have
		\begin{align}
			\norm{P\left(\vec{z}_i^{(k)},\vec{y}_i^{(k)},\alpha_k\right)}_2^2=&\norm{\frac{1}{\alpha_k}\left(\vec{z}_i^{(k)}-\hbox{\normalfont prox}_{\alpha_kr}\left(\vec{z}_i^{(k)}-\alpha_k\vec{y}_i^{(k)}\right)\right)}_2^2\nonumber\\
			\stack{\eqref{algo:x_update}}{=}&\frac{2}{\alpha_k^2}\norm{\vec{x}_i^{(k+1)}-\vec{z}_i^{(k)}}_2^2\label{lemma:stationary_bound:prox_grad_relation}
		\end{align}
	and by~\eqref{lemma:prox:bound}, we further have
		\begin{align}
			&\norm{P\left(\vec{z}_i^{(k)},\vec{y}_i^{(k)},\alpha_k\right)- P\left(\vec{z}_i^{(k)},\nabla f(\vec{z}_i^{(k)}),\alpha_k\right)}_2^2\nonumber\\
			=&\frac{1}{\alpha_k^2}\norm{\hbox{\normalfont prox}_{\alpha_kr}\left(\vec{z}_i^{(k)}-\alpha_k\vec{y}_i^{(k)}\right)-\hbox{\normalfont prox}_{\alpha_kr}\left(\vec{z}_i^{(k)}-\alpha_k\nabla f(\vec{z}_i^{(k)})\right)}_2^2\nonumber\\
			\le&\norm{\vec{y}_i^{(k)}-\nabla f(\vec{z}_i^{(k)})}_2^2.\label{lemma:stationary_bound:eqn1}
		\end{align}
	By Young's inequality and~\eqref{lemma:stationary_bound:prox_grad_relation}, we have
		\begin{align}
			\frac{1}{2}\norm{P\left(\vec{z}_i^{(k)},\nabla f(\vec{z}_i^{(k)}),\alpha_k\right)}_2^2\le&\frac{1}{\alpha_k^2}\norm{\vec{x}_i^{(k+1)}-\vec{z}_i^{(k)}}_2^2\nonumber\\
			&+\norm{P\left(\vec{z}_i^{(k)},\vec{y}_i^{(k)},\alpha_k\right)- P\left(\vec{z}_i^{(k)},\nabla f(\vec{z}_i^{(k)}),\alpha_k\right)}_2^2.\label{lemma:stationary_bound:eqn2}
		\end{align}
	Plugging~\eqref{lemma:stationary_bound:eqn1} into~\eqref{lemma:stationary_bound:eqn2} and summing over $i=1,\dots,N$ yields
		\begin{align}
			&\frac{1}{2}\sum_{i=1}^N\norm{P\left(\vec{z}_i^{(k)},\nabla f(\vec{z}_i^{(k)}),\alpha_k\right)}_2^2\nonumber\\
			\le&\frac{1}{\alpha_k^2}\sum_{i=1}^N\norm{\vec{x}_i^{(k+1)}-\vec{z}_i^{(k)}}_2^2+\sum_{i=1}^N\norm{\vec{y}_i^{(k)}-\nabla f(\vec{z}_i^{(k)})}_2^2\nonumber\\
			=&\frac{1}{\alpha_k^2}\fronorm{\x^{(k+1)}-\z^{(k)}}+\sum_{i=1}^N\norm{\vec{y}_i^{(k)}-\nabla f(\vec{z}_i^{(k)})}_2^2,\label{lemma:stationary_bound:eqn3}
		\end{align}
	where we have utilized the definition of the Frobenius norm. Next, we bound
		\begin{align}
			&\sum_{i=1}^N\norm{\vec{y}_i^{(k)}-\nabla f(\vec{z}_i^{(k)})}_2^2\nonumber\\
			\stack{\eqref{lemma:y_d_relation:bound}}{=}&\enskip\sum_{i=1}^N\norm{\vec{y}_i^{(k)}-\bar{\vec{y}}^{(k)}+\bar{\vec{d}}^{(k)}-\frac{1}{N}\sum_{j=1}^N\nabla f_j(\blx_j^{(k)})+\frac{1}{N}\sum_{j=1}^N\nabla f_j(\blx_j^{(k)})-\nabla f(\vec{z}_i^{(k)})}_2^2\nonumber\\
			\le&3\sum_{i=1}^N\left(\norm{\vec{y}_i^{(k)}-\bar{\vec{y}}^{(k)}}_2^2+\norm{\bar{\vec{d}}^{(k)}-\frac{1}{N}\sum_{j=1}^N\nabla f_j(\blx_j^{(k)})}_2^2+\norm{\frac{1}{N}\sum_{j=1}^N\nabla f_j(\blx_j^{(k)})-\nabla f(\vec{z}_i^{(k)})}_2^2\right).\label{lemma:stationary_bound:y_term}
		\end{align}
	Looking at terms individually, we have
		\begin{align}
			\norm{\bar{\vec{d}}^{(k)}-\frac{1}{N}\sum_{j=1}^N\nabla f_j(\blx_j^{(k)})}_2^2=\norm{\frac{1}{N}\sum_{j=1}^N\left(\vec{d}_j^{(k)}-\nabla f_j(\blx_j^{(k)})\right)}_2^2=\norm{\bar{\vec{r}}^{(k)}}_2^2.\label{lemma:stationary_bound:d_part}
		\end{align}
	By Jensen's inequality, we have
		\begin{align}
			&\norm{\frac{1}{N}\sum_{j=1}^N\nabla f_j(\blx_j^{(k)})-\nabla f(\vec{z}_i^{(k)})}_2^2\nonumber\\
			=&\norm{\frac{1}{N}\sum_{j=1}^N\left(\nabla f_j(\blx_j^{(k)})-\nabla f_j(\bar{\blx}^{(k)})+\nabla f_j(\bar{\blx}^{(k)})-\nabla f_j(\vec{z}_i^{(k)})\right)}_2^2\nonumber\\
			\le&\frac{1}{N}\sum_{j=1}^N\norm{\nabla f_j(\blx_j^{(k)})-\nabla f_j(\bar{\blx}^{(k)})+\nabla f_j(\bar{\blx}^{(k)})-\nabla f_j(\vec{z}_i^{(k)})}_2^2\nonumber\\
			\stack{\eqref{assumption:smoothness}}{\le}&\frac{2L^2}{N}\sum_{j=1}^N\left(\norm{\blx_j^{(k)}-\bar{\blx}^{(k)}}_2^2+\norm{\bar{\blx}^{(k)}-\vec{z}_i^{(k)}}_2^2\right)\nonumber\\
			=&\frac{2L^2}{N}\fronorm{\x^{(k)}-\bar{\x}^{(k)}}+\frac{2L^2}{N}\sum_{j=1}^N\norm{\bar{\blx}^{(k)}-\vec{z}_i^{(k)}}_2^2\nonumber\\
			=&\frac{2L^2}{N}\fronorm{\x^{(k)}-\bar{\x}^{(k)}}+2L^2\norm{\bar{\blx}^{(k)}-\vec{z}_i^{(k)}}_2^2.\label{lemma:stationary_bound:grad_part}
		\end{align}
	Plugging~\eqref{lemma:stationary_bound:d_part} and~\eqref{lemma:stationary_bound:grad_part} into~\eqref{lemma:stationary_bound:y_term} yields
		\begin{align}
			&\sum_{i=1}^N\norm{\vec{y}_i^{(k)}-\nabla f(\vec{z}_i^{(k)})}_2^2\nonumber\\
			\le&3\fronorm{\y_\perp^{(k)}}+3N\norm{\bar{\vec{r}}^{(k)}}_2^2+6L^2\fronorm{\x_\perp^{(k)}}+6L^2\fronorm{\bar{\x}^{(k)}-\z^{(k)}}.\label{lemma:stationary_bound:eqn4}
		\end{align}
 Adding $\frac{L^2}{2}\fronorm{\x_\perp^{(k)}}$ to both sides of~\eqref{lemma:stationary_bound:eqn3}, applying~\eqref{lemma:stationary_bound:eqn4}, and dividing by $N$ results in
		\begin{align}
			&\frac{1}{2N}\sum_{i=1}^N\norm{P\left(\vec{z}_i^{(k)},\nabla f(\vec{z}_i^{(k)}),\alpha_k\right)}_2^2+\frac{L^2}{2N}\fronorm{\x_\perp^{(k)}}\nonumber\\
			\le&\frac{3}{N}\fronorm{\y_\perp^{(k)}}+3\norm{\bar{\vec{r}}^{(k)}}_2^2+\frac{13L^2}{2N}\fronorm{\x_\perp^{(k)}}+\frac{1}{N\alpha_k^2}\fronorm{\x^{(k+1)}-\z^{(k)}}+\frac{6L^2}{N}\fronorm{\bar{\x}^{(k)}-\z^{(k)}}.\label{lemma:stationarity_bound:eqn5}
		\end{align}
	Notice that by line 3 of Algorithm~\ref{algo:} and Assumption~\ref{assumption:mixing_matrix} (\textit{iv}), $\bar{\z}^{(k)}=\avg\w_T(\x^{(k)})=\bar{\x}^{(k)}$. Hence
		\begin{equation}\label{lemma:stationarity_bound:eqn6}
			\fronorm{\z_\perp^{(k)}}\stack{\eqref{algo:comm_update}}{=}\fronorm{\w_T(\x^{(k)})-\bar{\x}^{(k)}}=\fronorm{\left(\w_T-\avg\right)\x_\perp^{(k)}}\stack{\eqref{spectral_gap_cheby}}{\le}\tilde{\rho}^2\fronorm{\x_\perp^{(k)}}<\fronorm{\x_\perp^{(k)}},
		\end{equation}
	so adding $\frac{L^2}{2N}\fronorm{\z_\perp^{(k)}}$ to both sides of~\eqref{lemma:stationarity_bound:eqn5} and using~\eqref{lemma:stationarity_bound:eqn6} results in
		\begin{align*}
			&\frac{1}{2N}\sum_{i=1}^N\norm{P\left(\vec{z}_i^{(k)},\nabla f(\vec{z}_i^{(k)}),\alpha_k\right)}_2^2+\frac{L^2}{2N}\fronorm{\x_\perp^{(k)}}+\frac{L^2}{2N}\fronorm{\z_\perp^{(k)}}\\
			\le&\frac{3}{N}\fronorm{\y_\perp^{(k)}}+3\norm{\bar{\vec{r}}^{(k)}}_2^2+\frac{16L^2}{N}\fronorm{\x_\perp^{(k)}}+\frac{1}{N\alpha_k^2}\fronorm{\x^{(k+1)}-\z^{(k)}}.
		\end{align*}
	Finally, we multiply both sides by 2, which completes the proof.
\end{proof}

We are now in position to define a lower bounded Lyapunov function and use this to show the convergence of \ourmethod v1 and v2. Notice that until this point, the analyses of v1 and v2 of our method only differ slightly, i.e. in terms of the constants involved in Lemmas~\ref{lemma:hsgd_error},~\ref{lemma:storm_error},~\ref{lemma:avg_error}, and~\ref{lemma:avg_error_storm}. Since the bound established in~\eqref{lemma:storm_error:bound} is larger than~\eqref{lemma:hsgd_error:bound}, we upper bound~\eqref{lemma:hsgd_error:bound} by~\eqref{lemma:storm_error:bound}. Additionally, we notice that Lemma~\ref{lemma:avg_error_storm} provides an upper bound on the results from Lemma~\ref{lemma:avg_error} and hence use~\eqref{lemma:avg_error:bound_storm} in the following Lemma.

\subsection{Constant step size}

\begin{lemma}\label{lemma:base_change}
	For all $k\ge0,$ the following inequality holds
		\begin{equation}\label{lemma:base_change:bound}
			\begin{split}
								&\underbrace{\left(\frac{1}{4N\alpha_k}-\frac{3L}{2N}\right)}_{(A)}\Exp\fronorm{\x^{(k+1)}-\bar{\x}^{(k)}}\\
								&+\underbrace{\left(\frac{1}{2N\alpha_k}-\frac{72L^2\alpha_k}{N}-\frac{8L^2\alpha_k(1-\beta_k)^2}{N^2}-\frac{16L^2\gamma_1^{(k)}}{1-\tilde{\rho}}-(\gamma_3^{(k)}+\frac{\gamma_4^{(k)}}{N^2})4L^2(1-\beta_k)^2\right)}_{(B)}\Exp\fronorm{\x^{(k+1)}-\z^{(k)}}\\
								&+\underbrace{\left(\gamma_2^{(k-1)}-\tilde{\rho}\gamma_2^{(k)}-\frac{L}{2N}-\frac{\tilde{\rho}^2}{2N\alpha_k}\right)}_{(C)}\Exp\fronorm{\x_\perp^{(k)}}\\
								&+\underbrace{\left(-(9+\frac{(1-\beta_k)^2}{N})\frac{32L^2\alpha_k}{N}-\frac{64L^2\gamma_1^{(k)}}{1-\tilde{\rho}}-(\gamma_3^{(k)}+\frac{\gamma_4^{(k)}}{N^2})16L^2(1-\beta_k)^2\right)}_{(C)}\Exp\fronorm{\x_\perp^{(k)}}\\
								&+\underbrace{\left(\gamma_1^{(k-1)}-\tilde{\rho}\gamma_1^{(k)}-\frac{1}{2NL}-\frac{\gamma_2^{(k)}\alpha_k^2}{1-\tilde{\rho}}\right)}_{(D)}\Exp\fronorm{\y_\perp^{(k)}}\\
								&+\underbrace{\left(\gamma_3^{(k-1)}-\frac{16\alpha_k\beta_k^2}{N}-\frac{4\gamma_1^{(k)}\beta_k^2}{1-\tilde{\rho}}-\gamma_3^{(k)}(1-\beta_k)^2\right)}_{(E)}\Exp\fronorm{\vec{R}^{(k)}}\\
								&+\underbrace{\left(\gamma_4^{(k-1)}-\gamma_4^{(k)}(1-\beta_k)^2-2\alpha_k(1-\beta_k)^2\right)}_{(F)}\Exp\norm{\bar{\vec{r}}^{(k)}}_2^2\\
								\le&\Exp\left[\Phi^{(k)}-\Phi^{(k+1)}\right]+\left(\frac{4\gamma_1^{(k)}N}{1-\tilde{\rho}}+2N\gamma_3^{(k)}+\frac{2\gamma_4^{(k)}}{N}+16\alpha_k+\frac{4\alpha_k}{N}\right)\hat{\sigma}^2\beta_k^2,
			\end{split}
		\end{equation}
	where $\gamma_1^{(k)},\gamma_2^{(k)},\gamma_3^{(k)},\gamma_4^{(k)}$ are strictly positive values and
		\begin{equation}\label{lemma:base_change_update:lyapunov}
			\Phi^{(k)}\triangleq\phi(\barxk)+\gamma_1^{(k-1)}\fronorm{\y_\perp^{(k)}}+\gamma_2^{(k-1)}\fronorm{\x_\perp^{(k)}}+\gamma_3^{(k-1)}\fronorm{\vec{R}^{(k)}}+\gamma_4^{(k-1)}\norm{\bar{\vec{r}}^{(k)}}_2^2
		\end{equation}
	is a lower bounded Lyapunov function.
\end{lemma}
\begin{proof}
	We start by using part \textit{(iv)} of Assumption~\ref{assumption:mixing_matrix} and~\eqref{algo:comm_update} to note that
			\begin{align}\label{lemma:base_change:eqn1}
				\frac{1}{2N\alpha_k}\fronorm{\bar{\x}^{(k)}-\z^{(k)}}&=\frac{1}{2N\alpha_k}\fronorm{\left(\w_T-\avg\right)\left(\identity-\avg\right)\x^{(k)}}\nonumber\\
				\stack{\eqref{spectral_gap_cheby}}{\le}&\frac{\tilde{\rho}^2}{2N\alpha_k}\fronorm{\x_\perp^{(k)}}.
			\end{align}
		Next, we utilize the Peter-Paul inequality and Jensen's inequality to have
			\begin{align}\label{lemma:base_change:eqn2}
				&-\ip{\bar{\vec{r}}^{(k)},\barxknew-\barxk}\nonumber\\
				=&\ip{-\bar{\vec{r}}^{(k+1)}+\bar{\vec{r}}^{(k+1)}-\bar{\vec{r}}^{(k)},\barxknew-\barxk}\nonumber\\
				\le&\alpha_k\norm{-\bar{\vec{r}}^{(k+1)}+\bar{\vec{r}}^{(k+1)}-\bar{\vec{r}}^{(k)}}_2^2+\frac{1}{4\alpha_k}\norm{\barxknew-\barxk}_2^2\nonumber\\
				\le&2\alpha_k\norm{\bar{\vec{r}}^{(k+1)}}_2^2+2\alpha_k\norm{\bar{\vec{r}}^{(k+1)}-\bar{\vec{r}}^{(k)}}_2^2+\frac{1}{4N\alpha_k}\sum_{i=1}^N\norm{\xiknew-\barxk}_2^2\nonumber\\
				=&2\alpha_k\norm{\bar{\vec{r}}^{(k+1)}}_2^2+2\alpha_k\norm{\bar{\vec{r}}^{(k+1)}-\bar{\vec{r}}^{(k)}}_2^2+\frac{1}{4N\alpha_k}\fronorm{\x^{(k+1)}-\bar{\x}^{(k)}}.
			\end{align}
	Further, by Young's inequality it holds that
		\begin{align}\label{lemma:base_change:eqn3}
			&2\alpha_k\norm{\bar{\vec{r}}^{(k+1)}-\bar{\vec{r}}^{(k)}}_2^2\nonumber\\
			=&2\alpha_k\norm{\bar{\vec{d}}^{(k+1)}-\bar{\vec{d}}^{(k)}+\frac{1}{N}\sum_{i=1}^N\nabla f_i(\blx_i^{(k+1)})-\frac{1}{N}\sum_{i=1}^N\nabla f_i(\blx_i^{(k)})}_2^2\nonumber\\
			\le&\frac{4\alpha_k}{N}\sum_{i=1}^N\norm{\vec{d}_i^{(k+1)}-\vec{d}_i^{(k)}}_2^2+\frac{4\alpha_k}{N}\sum_{i=1}^N\norm{\nabla f_i(\blx_i^{(k+1)})-\nabla f_i(\blx_i^{(k)})}_2^2\nonumber\\
			\stack{\eqref{assumption:smoothness}}{\le}&\frac{4\alpha_k}{N}\fronorm{\d^{(k+1)}-\d^{(k)}}+\frac{4L^2\alpha_k}{N}\sum_{i=1}^N\norm{\blx_i^{(k+1)}-\blx_i^{(k)}}_2^2\nonumber\\
			=&\frac{4\alpha_k}{N}\fronorm{\d^{(k+1)}-\d^{(k)}}+\frac{4L^2\alpha_k}{N}\fronorm{\x^{(k+1)}-\x^{(k)}}.
		\end{align}
	Taking the expectation conditioned on the local samples and then taking the full expectation yields
		\begin{equation}\label{lemma:base_change:eqn4}
			\frac{4\alpha_k}{N}\Exp\fronorm{\d^{(k+1)}-\d^{(k)}}\enskip\stack{\eqref{lemma:matrix_bounds:d_diff}}{\le}\enskip\frac{4\alpha_k}{N}\left(8L^2\Exp\fronorm{\x^{(k+1)}-\x^{(k)}}+4\beta_k^2\Exp\fronorm{\vec{R}^{(k)}}+4N\beta_k^2\hat{\sigma}^2\right),
		\end{equation}
	where we have also used~\eqref{assumption:smoothness} and~\eqref{unbiased_assumption}. Plugging~\eqref{lemma:base_change:eqn4} into~\eqref{lemma:base_change:eqn3} and using~\eqref{lemma:base_change:eqn2} yields
		\begin{align}\label{lemma:base_change:eqn5}
			&-\Exp\ip{\bar{\vec{r}}^{(k)},\barxknew-\barxk}\nonumber\\
			\le&2\alpha_k\Exp\norm{\bar{\vec{r}}^{(k+1)}}_2^2+\frac{1}{4N\alpha_k}\Exp\fronorm{\x^{(k+1)}-\bar{\x}^{(k)}}\nonumber\\
				&+\frac{36L^2\alpha_k}{N}\Exp\fronorm{\x^{(k+1)}-\x^{(k)}}+\frac{16\alpha_k\beta_k^2}{N}\Exp\fronorm{\vec{R}^{(k)}}+16\alpha_k\beta_k^2\hat{\sigma}^2.
		\end{align}
	Using~\eqref{lemma:base_change:eqn1} and~\eqref{lemma:base_change:eqn5} in~\eqref{lemma:second_change:bound} and taking the full expectation results in
		\begin{align}
			\Exp\left[\phi(\barxknew)-\phi(\barxk)\right]\le&-\frac{1}{2N}\left(\frac{1}{2\alpha_k}-3L\right)\Exp\fronorm{\x^{(k+1)}-\bar{\x}^{(k)}}\nonumber\\
						&+\left(\frac{L}{2N}+\frac{\tilde{\rho}^2}{2N\alpha_k}\right)\Exp\fronorm{\x_\perp^{(k)}}-\frac{1}{2N\alpha_k}\Exp\fronorm{\x^{(k+1)}-\z^{(k)}}\nonumber\\
						&+\frac{1}{2NL}\Exp\fronorm{\y_\perp^{(k)}}+2\alpha_k\Exp\norm{\bar{\vec{r}}^{(k+1)}}_2^2\nonumber\\
						&+\frac{36L^2\alpha_k}{N}\Exp\fronorm{\x^{(k+1)}-\x^{(k)}}+\frac{16\alpha_k\beta_k^2}{N}\Exp\fronorm{\vec{R}^{(k)}}+16\alpha_k\beta_k^2\hat{\sigma}^2.
		\end{align}
	Noticing that the right-hand side of~\eqref{lemma:storm_error:bound} is larger than the right-hand side of~\eqref{lemma:hsgd_error:bound}, we add $\gamma_1^{(k)}\Exp\fronorm{\y_\perp^{(k+1)}}$, $\gamma_2^{(k)}\Exp\fronorm{\x_\perp^{(k+1)}}$, $\gamma_3^{(k)}\Exp\fronorm{\vec{R}^{(k+1)}}$, $\gamma_4^{(k)}\Exp\norm{\bar{\vec{r}}^{(k+1)}}_2^2$ to both sides of the above inequality and use the results from Lemmas~\ref{lemma:matrix_bounds},~\ref{lemma:storm_error}, and~\ref{lemma:avg_error_storm} with~\eqref{lemma:base_change_update:lyapunov}, and subtract $\gamma_1^{(k-1)}\Exp\fronorm{\y_\perp^{(k)}}$, $\gamma_2^{(k-1)}\Exp\fronorm{\x_\perp^{(k)}}$, $\gamma_3^{(k-1)}\Exp\fronorm{\vec{R}^{(k)}}$, $\gamma_4^{(k-1)}\Exp\norm{\bar{\vec{r}}^{(k)}}_2^2$ from both sides of the above inequality yields
		\begin{align}
			&\Exp\left[\Phi^{(k+1)}-\Phi^{(k)}\right]\nonumber\\
			\le&-\frac{1}{2N}\left(\frac{1}{2\alpha_k}-3L\right)\Exp\fronorm{\x^{(k+1)}-\bar{\x}^{(k)}}\nonumber\\
					&+\left(\frac{L}{2N}+\frac{\tilde{\rho}^2}{2N\alpha_k}\right)\Exp\fronorm{\x_\perp^{(k)}}-\frac{1}{2N\alpha_k}\Exp\fronorm{\x^{(k+1)}-\z^{(k)}}\nonumber\\
					&+\frac{1}{2NL}\Exp\fronorm{\y_\perp^{(k)}}+2\alpha_k\Exp\norm{\bar{\vec{r}}^{(k+1)}}_2^2\nonumber\\
					&+\frac{36L^2\alpha_k}{N}\Exp\fronorm{\x^{(k+1)}-\x^{(k)}}+\frac{16\alpha_k\beta_k^2}{N}\Exp\fronorm{\vec{R}^{(k)}}+16\alpha_k\beta_k^2\hat{\sigma}^2\nonumber\\
					&+\gamma_1^{(k)}\left(\tilde{\rho}\Exp\fronorm{\y_\perp^{(k)}}+\frac{1}{1-\tilde{\rho}}\left(8L^2\Exp\fronorm{\x^{(k+1)}-\x^{(k)}}+4\beta_k^2\Exp\fronorm{\vec{R}^{(k)}}+4N\beta_k^2\hat{\sigma}^2\right)\right)\nonumber\\
					&+\gamma_2^{(k)}\left(\tilde{\rho}\Exp\fronorm{\x_\perp^{(k)}}+\frac{\alpha_k^2}{1-\tilde{\rho}}\Exp\fronorm{\y_{\perp}^{(k)}}\right)\nonumber\\
					&+\gamma_3^{(k)}\left(2N\beta_k^2\hat{\sigma}^2+2(1-\beta_k)^2L^2\Exp\fronorm{\x^{(k+1)}-\x^{(k)}}+(1-\beta_k)^2\Exp\fronorm{\vec{R}^{(k)}}\right)\nonumber\\
					&+\gamma_4^{(k)}\left((1-\beta_k)^2\Exp\norm{\bar{\vec{r}}^{(k)}}_2^2+\frac{2(1-\beta_k)^2L^2}{N^2}\Exp\fronorm{\x^{(k+1)}-\x^{(k)}}+\frac{2\beta_k^2\hat{\sigma}^2}{N}\right)\nonumber\\
					&-\gamma_1^{(k-1)}\Exp\fronorm{\y_\perp^{(k)}}-\gamma_2^{(k-1)}\Exp\fronorm{\x_\perp^{(k)}}-\gamma_3^{(k-1)}\Exp\fronorm{\vec{R}^{(k)}}-\gamma_4^{(k-1)}\Exp\norm{\bar{\vec{r}}^{(k)}}_2^2.\label{lemma:base_change:eqn6}
		\end{align}
	Grouping like terms in~\eqref{lemma:base_change:eqn6} results in
		\begin{align*}
		&\Exp\left[\Phi^{(k+1)}-\Phi^{(k)}\right]\\
			\le&-\frac{1}{2N}\left(\frac{1}{2\alpha_k}-3L\right)\Exp\fronorm{\x^{(k+1)}-\bar{\x}^{(k)}}-\frac{1}{2N\alpha_k}\Exp\fronorm{\x^{(k+1)}-\z^{(k)}}\\
				&+\left(-\gamma_2^{(k-1)}+\tilde{\rho}\gamma_2^{(k)}+\frac{L}{2N}+\frac{\tilde{\rho}^2}{2N\alpha_k}\right)\Exp\fronorm{\x_\perp^{(k)}}\\
				&+\left(-\gamma_1^{(k-1)}+\tilde{\rho}\gamma_1^{(k)}+\frac{1}{2NL}+\frac{\gamma_2^{(k)}\alpha_k^2}{1-\tilde{\rho}}\right)\Exp\fronorm{\y_\perp^{(k)}}\\
				&+\left(-\gamma_3^{(k-1)}+\gamma_3^{(k)}(1-\beta_k)^2+\frac{16\alpha_k\beta_k^2}{N}+\frac{4\gamma_1^{(k)}\beta_k^2}{1-\tilde{\rho}}\right)\Exp\fronorm{\vec{R}^{(k)}}\\
				&+\left(-\gamma_4^{(k-1)}+\gamma_4^{(k)}(1-\beta_k)^2\right)\Exp\norm{\bar{\vec{r}}^{(k)}}_2^2\\
				&+\left(\frac{36L^2\alpha_k}{N}+\frac{8L^2\gamma_1^{(k)}}{1-\tilde{\rho}}+2L^2\gamma_3^{(k)}(1-\beta_k)^2+\frac{2L^2\gamma_4^{(k)}(1-\beta_k)^2}{N^2}\right)\Exp\fronorm{\x^{(k+1)}-\x^{(k)}}\\
				&+\left(16\alpha_k+\frac{4N\gamma_1^{(k)}}{1-\tilde{\rho}}+2N\gamma_3^{(k)}+\frac{2\gamma_4^{(k)}}{N}\right)\beta_k^2\hat{\sigma}^2\\
				&+2\alpha_k\Exp\norm{\bar{\vec{r}}^{(k+1)}}_2^2.
		\end{align*}
	Next, we use that the right-hand side of~\eqref{lemma:avg_error:bound_storm} is larger than the right-hand side of~\eqref{lemma:avg_error:bound} to have,
		\begin{align*}
		&\Exp\left[\Phi^{(k+1)}-\Phi^{(k)}\right]\\
			\stack{\eqref{lemma:avg_error:bound_storm}}{\le}&-\frac{1}{2N}\left(\frac{1}{2\alpha_k}-3L\right)\Exp\fronorm{\x^{(k+1)}-\bar{\x}^{(k)}}-\frac{1}{2N\alpha_k}\Exp\fronorm{\x^{(k+1)}-\z^{(k)}}\\
					&+\left(-\gamma_2^{(k-1)}+\tilde{\rho}\gamma_2^{(k)}+\frac{L}{2N}+\frac{\tilde{\rho}^2}{2N\alpha_k}\right)\Exp\fronorm{\x_\perp^{(k)}}\\
					&+\left(-\gamma_1^{(k-1)}+\tilde{\rho}\gamma_1^{(k)}+\frac{1}{2NL}+\frac{\gamma_2^{(k)}\alpha_k^2}{1-\tilde{\rho}}\right)\Exp\fronorm{\y_\perp^{(k)}}\\
					&+\left(-\gamma_3^{(k-1)}+\gamma_3^{(k)}(1-\beta_k)^2+\frac{16\alpha_k\beta_k^2}{N}+\frac{4\gamma_1^{(k)}\beta_k^2}{1-\tilde{\rho}}\right)\Exp\fronorm{\vec{R}^{(k)}}\\
					&+\left(-\gamma_4^{(k-1)}+\gamma_4^{(k)}(1-\beta_k)^2\right)\Exp\norm{\bar{\vec{r}}^{(k)}}_2^2\\
					&+\left(\frac{36L^2\alpha_k}{N}+\frac{8L^2\gamma_1^{(k)}}{1-\tilde{\rho}}+2L^2\gamma_3^{(k)}(1-\beta_k)^2+\frac{2L^2\gamma_4^{(k)}(1-\beta_k)^2}{N^2}\right)\Exp\fronorm{\x^{(k+1)}-\x^{(k)}}\\
					&+\left(16\alpha_k+\frac{4N\gamma_1^{(k)}}{1-\tilde{\rho}}+2N\gamma_3^{(k)}+\frac{2\gamma_4^{(k)}}{N}\right)\beta_k^2\hat{\sigma}^2\\
					&+2\alpha_k\left((1-\beta_k)^2\Exp\norm{\bar{\vec{r}}^{(k)}}_2^2+\frac{2(1-\beta_k)^2L^2}{N^2}\Exp\fronorm{\x^{(k+1)}-\x^{(k)}}+\frac{2\beta_k^2\hat{\sigma}^2}{N}\right).\\
		\end{align*}
	Further using
		\begin{align*}
				\fronorm{\x^{(k+1)}-\x^{(k)}}&=\fronorm{\x^{(k+1)}-\z^{(k)}+\z^{(k)}-\x^{(k)}}\\
					&\le2\fronorm{\x^{(k+1)}-\z^{(k)}}+2\fronorm{\w_T(\x^{(k)})-\w_T\bar{\x}^{(k)}+\w_T\bar{\x}^{(k)}-\x^{(k)}}\\
					&=2\fronorm{\x^{(k+1)}-\z^{(k)}}+2\fronorm{\left(\identity-\w_T\right)\x_\perp^{(k)}}\\
					&\le2\fronorm{\x^{(k+1)}-\z^{(k)}}+8\fronorm{\x_\perp^{(k)}}
		\end{align*}
	and combining like terms completes the proof.
\end{proof}

\subsubsection{Proof of Theorem~\ref{theorem:constant_convergence}}\label{appendix:constant_proof}

\begin{proof}
	We approach this proof in phases; first, we note that $\beta_k\in(0,1)$ by $0<\alpha\le\frac{K^{\frac{1}{3}}}{32L}$. Second, let
		\begin{equation}\label{theorem:constant_convergence:params}
			\gamma_1^{(k)}\triangleq\frac{1}{NL(1-\tilde{\rho})}\text{, }\gamma_2^{(k)}\triangleq\frac{16K^{\frac{1}{3}}}{N(1-\tilde{\rho})\alpha}\text{, }\gamma_3^{(k)}\triangleq\frac{K^{\frac{1}{3}}}{48NL^2\alpha}\text{, and }\gamma_4^{(k)}\triangleq\frac{NK^{\frac{1}{3}}}{48L^2\alpha}
		\end{equation}
	 in~\eqref{lemma:base_change:bound} to have
		\begin{align}
			&\underbrace{\left(\frac{K^{\frac{1}{3}}}{4N\alpha}-\frac{3L}{2N}\right)}_{(A')}\Exp\fronorm{\x^{(k+1)}-\bar{\x}^{(k)}}\nonumber\\
				&+\underbrace{\left(\frac{K^{\frac{1}{3}}}{2N\alpha}-\frac{72L^2\alpha}{NK^{\frac{1}{3}}}-\frac{8L^2\alpha(1-\beta_k)^2}{N^2K^{\frac{1}{3}}}-\frac{16L}{N(1-\tilde{\rho})^2}-\frac{(1-\beta_k)^2K^{\frac{1}{3}}}{6N\alpha}\right)}_{(B')}\Exp\fronorm{\x^{(k+1)}-\z^{(k)}}\nonumber\\
				&+\underbrace{\left(\frac{16K^{\frac{1}{3}}}{N\alpha}-\frac{L}{2N}-\frac{\tilde{\rho}^2K^{\frac{1}{3}}}{2N\alpha}-\frac{288L^2\alpha}{NK^{\frac{1}{3}}}-\frac{32L^2\alpha(1-\beta_k)^2}{N^2K^{\frac{1}{3}}}-\frac{64L}{N(1-\tilde{\rho})^2}-\frac{2(1-\beta_k)^2K^{\frac{1}{3}}}{3N\alpha}\right)}_{(C')}\Exp\fronorm{\x_\perp^{(k)}}\nonumber\\
				&+\underbrace{\left(\frac{1}{NL}-\frac{1}{2NL}-\frac{16\alpha}{N(1-\tilde{\rho})^2K^{\frac{1}{3}}}\right)}_{(D')}\Exp\fronorm{\y_\perp^{(k)}}\nonumber\\
				&+\underbrace{\left(\frac{K^{\frac{1}{3}}}{48NL^2\alpha}-\frac{(1-\beta_k)^2K^{\frac{1}{3}}}{48NL^2\alpha}-\frac{16\alpha\beta_k^2}{NK^{\frac{1}{3}}}-\frac{4\beta_k^2}{NL(1-\tilde{\rho})^2}\right)}_{(E')}\Exp\fronorm{\vec{R}^{(k)}}\nonumber\\
				&+\underbrace{\left(\frac{NK^{\frac{1}{3}}}{48L^2\alpha}-\frac{(1-\beta_k)^2NK^{\frac{1}{3}}}{48L^2\alpha}-\frac{2\alpha(1-\beta_k)^2}{K^{\frac{1}{3}}}\right)}_{(F')}\Exp\norm{\bar{\vec{r}}^{(k)}}_2^2\nonumber\\
					\le&\Exp\left[\Phi^{(k)}-\Phi^{(k+1)}\right]+\left(\frac{4}{L(1-\tilde{\rho})^2}+\frac{K^{\frac{1}{3}}}{12L^2\alpha}+\frac{16\alpha}{K^{\frac{1}{3}}}+\frac{4\alpha}{NK^{\frac{1}{3}}}\right)\hat{\sigma}^2\beta_k^2.\label{theorem:constant_convergence:relation}
		\end{align}
	Next, we lower bound ($A'$) - ($E'$). For ($A'$), we have
		\begin{align*}
			(A')=\frac{K^{\frac{1}{3}}}{4N\alpha}-\frac{3L}{2N}\ge&\frac{K^{\frac{1}{3}}}{4N\alpha}-\frac{3K^{\frac{1}{3}}}{64N\alpha}\\
			>&\frac{K^{\frac{1}{3}}}{5N\alpha},
		\end{align*}
	where the first inequality uses $\alpha\le\frac{K^{\frac{1}{3}}}{32L}.$ For ($B'$), we use $(1-\beta_k)^2<1$ and $\alpha\le\min\{\frac{K^{\frac{1}{3}}}{32L},\frac{(1-\tilde{\rho})^2K^{\frac{1}{3}}}{64L}\}$ to have
		\begin{align*}
				(B')=&\frac{K^{\frac{1}{3}}}{2N\alpha}-\frac{72L^2\alpha}{NK^{\frac{1}{3}}}-\frac{8L^2\alpha(1-\beta_k)^2}{N^2K^{\frac{1}{3}}}-\frac{16L}{N(1-\tilde{\rho})^2}-\frac{(1-\beta_k)^2K^{\frac{1}{3}}}{6N\alpha}\\
				=&\frac{K^{\frac{1}{3}}}{N\alpha}\left(\frac{1}{2}-\frac{72L^2\alpha^2}{K^{\frac{2}{3}}}-\frac{8L^2\alpha^2(1-\beta_k)^2}{NK^{\frac{2}{3}}}-\frac{16L\alpha}{(1-\tilde{\rho})^2K^{\frac{1}{3}}}-\frac{(1-\beta_k)^2}{6}\right)\\
				>&\frac{K^{\frac{1}{3}}}{N\alpha}\left(\frac{1}{2}-\frac{72}{1024}-\frac{8}{1024N}-\frac{16}{64}-\frac{1}{6}\right)\\
				>&\frac{K^{\frac{1}{3}}}{256N\alpha}.
		\end{align*}
	For ($C'$), we again use $(1-\beta_k)<1$ and $\tilde{\rho}^2<1$ to have
		\begin{align*}
			(C')=&\frac{16K^{\frac{1}{3}}}{N\alpha}-\frac{L}{2N}-\frac{\tilde{\rho}^2K^{\frac{1}{3}}}{2N\alpha}-\frac{288L^2\alpha}{NK^{\frac{1}{3}}}-\frac{32L^2\alpha(1-\beta_k)^2}{N^2K^{\frac{1}{3}}}-\frac{64L}{N(1-\tilde{\rho})^2}-\frac{2(1-\beta_k)^2K^{\frac{1}{3}}}{3N\alpha}\\
			>&\frac{K^{\frac{1}{3}}}{N\alpha}\left(16-\frac{L\alpha}{2K^{\frac{1}{3}}}-\frac{1}{2}-\frac{288L^2\alpha^2}{K^{\frac{2}{3}}}-\frac{32L^2\alpha^2}{NK^{\frac{2}{3}}}-\frac{64L\alpha}{(1-\tilde{\rho})^2K^{\frac{1}{3}}}-\frac{2}{3}\right)\\
			\ge&\frac{K^{\frac{1}{3}}}{N\alpha}\left(16-\frac{1}{64}-\frac{1}{2}-\frac{288}{1024}-\frac{32}{1024N}-1-\frac{2}{3}\right)\\
			>&\frac{12K^{\frac{1}{3}}}{N\alpha},
		\end{align*}
	where the second to last inequality uses $\alpha\le\min\{\frac{K^{\frac{1}{3}}}{32L},\frac{(1-\tilde{\rho})^2K^{\frac{1}{3}}}{64L}\}.$ For ($D'$), it holds that
		\begin{align*}
			(D')=\frac{1}{NL}-\frac{1}{2NL}-\frac{16\alpha}{N(1-\tilde{\rho})^2K^{\frac{1}{3}}}=&\frac{1}{2NL}-\frac{16\alpha}{N(1-\tilde{\rho})^2K^{\frac{1}{3}}}\\
			\ge&\frac{1}{2NL}-\frac{1}{4NL}\\
			=&\frac{1}{4NL},
		\end{align*}
	where we have used $\alpha\le\frac{(1-\tilde{\rho})^2K^{\frac{1}{3}}}{64L}.$ For ($E'$), we expand $1-(1-\beta_k)^2=2\beta_k-\beta_k^2>\beta_k$ and use~\eqref{lemma:constant_choice:bounds} and $\beta_k^2<\beta_k$ to have
		\begin{align*}
			(E')=&\frac{K^{\frac{1}{3}}}{48NL^2\alpha}-\frac{(1-\beta_k)^2K^{\frac{1}{3}}}{48NL^2\alpha}-\frac{16\alpha\beta_k^2}{NK^{\frac{1}{3}}}-\frac{4\beta_k^2}{NL(1-\tilde{\rho})^2}\\
			>&\frac{3\alpha}{N^2K^{\frac{1}{3}}}-\frac{16\cdot144^2L^4\alpha^5}{N^3K^{\frac{5}{3}}}-\frac{4\cdot144^2L^3\alpha^4}{N^3(1-\tilde{\rho})^2K^{\frac{4}{3}}}\\
			\ge&\frac{\alpha}{N^2K^{\frac{1}{3}}}\left(3-\frac{1}{2}-\frac{4}{3}\right)\\
			>&\frac{\alpha}{N^2K^{\frac{1}{3}}},
		\end{align*}
	where the second inequality uses $\alpha^4\le\frac{K^{\frac{4}{3}}}{32^4L^4}$ and $\alpha^3\le\frac{(1-\tilde{\rho})^2K}{64\cdot 32^2L^3}.$ For ($F'$), we also expand $1-(1-\beta_k)^2=2\beta_k-\beta_k^2>\beta_k$ and use~\eqref{lemma:constant_choice:bounds} to have
		\begin{align*}
			(F')=&\frac{NK^{\frac{1}{3}}}{48L^2\alpha}-\frac{(1-\beta_k)^2NK^{\frac{1}{3}}}{48L^2\alpha}-\frac{2\alpha(1-\beta_k)^2}{K^{\frac{1}{3}}}\\
			>&\frac{3\alpha}{K^{\frac{1}{3}}}-\frac{2\alpha}{K^{\frac{1}{3}}}\\
			=&\frac{\alpha}{K^{\frac{1}{3}}}.
		\end{align*}
	Next, we sum~\eqref{theorem:constant_convergence:relation} over $k=0$ to $K-1$ and divide by $K$; using the established lower bounds to have 
		\begin{align}
			&\frac{1}{K}\sum_{k=0}^{K-1}\left(\frac{K^{\frac{1}{3}}}{5N\alpha}\Exp\fronorm{\x^{(k+1)}-\bar{\x}^{(k)}}+\frac{K^{\frac{1}{3}}}{256N\alpha}\Exp\fronorm{\x^{(k+1)}-\z^{(k)}}+\frac{\alpha}{K^{\frac{1}{3}}}\Exp\norm{\bar{\vec{r}}^{(k)}}_2^2\right)\nonumber\\
				&+\frac{1}{K}\sum_{k=0}^{K-1}\left(\frac{12K^{\frac{1}{3}}}{N\alpha}\Exp\fronorm{\x_\perp^{(k)}}+\frac{1}{4NL}\Exp\fronorm{\y_\perp^{(k)}}+\frac{\alpha}{N^2K^{\frac{1}{3}}}\Exp\fronorm{\vec{R}^{(k)}}\right)\nonumber\\
				&\quad\le\frac{1}{K}\left(\Phi^{(0)}-\phi^*\right)+\left(\frac{4}{L(1-\tilde{\rho})^2K^{\frac{4}{3}}}+\frac{1}{12L^2\alpha K}+\frac{16\alpha}{K^{\frac{5}{3}}}+\frac{4\alpha}{NK^{\frac{5}{3}}}\right)\frac{144^2L^4\alpha^4\hat{\sigma}^2}{N^2},\label{theorem:constant_convergence:pre_bound}
		\end{align}
	where we have used $\phi^*\le\Phi^{(k)}$ for any $k\ge0.$\\
	The final phase of the proof uses Lemma~\ref{lemma:stationary_bound} to provide a concise convergence statement. To do so, multiply both sides of~\eqref{lemma:stationary_bound:bound} by $\frac{\alpha}{K^{\frac{1}{3}}}$, sum from $k=0,\dots,K-1$ and divide by $K$, and take the expectation to have
		\begin{align}
			&\frac{\alpha}{K^{\frac{1}{3}}\cdot K}\sum_{k=0}^{K-1}\left(\frac{1}{N}\sum_{i=1}^N\Exp\norm{P\left(\vec{z}_i^{(k)},\nabla f(\vec{z}_i^{(k)}),\alpha_k\right)}_2^2+\frac{L^2}{N}\Exp\fronorm{\x_\perp^{(k)}}+\frac{L^2}{N}\Exp\fronorm{\z_\perp^{(k)}}\right)\nonumber\\
				\le&\frac{\alpha}{K^{\frac{1}{3}}\cdot K}\sum_{k=0}^{K-1}\left(\frac{6}{N}\Exp\fronorm{\y_\perp^{(k)}}+6\Exp\norm{\bar{\vec{r}}^{(k)}}_2^2+\frac{32L^2}{N}\Exp\fronorm{\x_\perp^{(k)}}+\frac{2}{N\alpha_k^2}\Exp\fronorm{\x^{(k+1)}-\z^{(k)}}\right).\label{theorem:constant_convergence:eqn1}
		\end{align}
 We relate each of the terms on the right-hand side of~\eqref{theorem:constant_convergence:eqn1} to 512 times of the left-hand side of~\eqref{theorem:constant_convergence:pre_bound}. Since $\alpha\le\frac{K^{\frac{1}{3}}}{32L}$ it holds that 
		\begin{equation}\label{theorem:constant_convergence:y_bound}
			\frac{\alpha}{K^{\frac{1}{3}}}\cdot\frac{6}{NK}\sum_{k=0}^{K-1}\fronorm{\y_\perp^{(k)}}\le\frac{128}{NL}\cdot\frac{1}{K}\sum_{k=0}^{K-1}\fronorm{\y_\perp^{(k)}}.
		\end{equation}
	Using $\alpha\le\frac{K^{\frac{1}{3}}}{32L}$ we have,
		\begin{equation}\label{theorem:constant_convergence:x_bound}
			\frac{\alpha}{K^{\frac{1}{3}}}\cdot\frac{32L^2}{NK}\sum_{k=0}^{K-1}\fronorm{\x_\perp^{(k)}}\le\frac{6144K^{\frac{1}{3}}}{\alpha}\cdot\frac{1}{NK}\sum_{k=0}^{K-1}\fronorm{\x_\perp^{(k)}}.
		\end{equation}
	Combining~\eqref{theorem:constant_convergence:y_bound} and~\eqref{theorem:constant_convergence:x_bound} in conjunction with 512 times of~\eqref{theorem:constant_convergence:pre_bound} and~\eqref{theorem:constant_convergence:eqn1} yields:
		\begin{align*}
					&\frac{\alpha}{K^{\frac{1}{3}}\cdot K}\sum_{k=0}^{K-1}\frac{1}{N}\sum_{i=1}^N\Exp\norm{P\left(\vec{z}_i^{(k)},\nabla f(\vec{z}_i^{(k)}),\alpha_k\right)}_2^2\\
					&+\frac{\alpha}{K^{\frac{1}{3}}\cdot K}\sum_{k=0}^{K-1}\frac{L^2}{N}\Exp\fronorm{\x_\perp^{(k)}}+\frac{\alpha}{K^{\frac{1}{3}}\cdot K}\sum_{k=0}^{K-1}\frac{L^2}{N}\Exp\fronorm{\z_\perp^{(k)}}\\
					\le&\frac{512}{K}\left(\Phi^{(0)}-\phi^*\right)+\left(\frac{2048}{L(1-\tilde{\rho})^2K^{\frac{4}{3}}}+\frac{128}{3L^2\alpha K}+\frac{8192\alpha}{K^{\frac{5}{3}}}+\frac{2048\alpha}{NK^{\frac{5}{3}}}\right)\frac{144^2L^4\alpha^4\hat{\sigma}^2}{N^2}.
		\end{align*}
	Multiplying both sides by $\frac{K^{\frac{1}{3}}}{\alpha}$ and using $\fronorm{\x_\perp^{(k)}}\ge0$ and $\fronorm{\vec{R}^{(k)}}\ge0$ for all $k\ge0$ completes the proof.
\end{proof}

\subsubsection{Complexity analysis}\label{appendix:constant_complexity}

Before presenting the complexity analysis for \ourmethod with a constant step size, we first present a preparatory Lemma.

\begin{lemma}\label{lemma:initial_comms}
	For any real numbers $x\in(0,1)$, it holds that
		\begin{equation}\label{lemma:initial_comms:bound}
			\frac{1}{x}\ge\frac{-1}{\ln(1-x)}.
		\end{equation}
\end{lemma}
\begin{proof}
	We prove
		\begin{align}\label{lemma:initial_comms:rearrange}
					g(x)\triangleq x+\ln(1-x)\le0,
		\end{align}
	which is equivalent to~\eqref{lemma:initial_comms:bound}. Computing the first derivative results in
		\begin{align*}
					\frac{d}{dx}g(x)=1-\frac{1}{1-x}<0,\enskip\forall x\in(0,1)
		\end{align*}
	since $(1-x)<1$ for all $x\in(0,1).$ Hence $g(x)$ is decreasing on $(0,1)$. Computing $g(0+)=0$, we have~\eqref{lemma:initial_comms:rearrange} and hence~\eqref{lemma:initial_comms:bound}.
\end{proof}

We make the following remark in order to aid in the discussion of presenting final complexity results for Algorithm~\ref{algo:} with a constant step size.

\begin{remark}\label{remark:batch}
	Notice that the convergence of Algorithm~\ref{algo:} depends upon $\Phi^{(0)}$ (see~\eqref{lemma:base_change_update:lyapunov}) which in turn depends upon $\gamma_{2}^{(0)}$, $\gamma_3^{(0)},$ and $\gamma_4^{(0)}$, all of which are $\bigO{K^{\frac{1}{3}}}.$ In order to obtain the best possible convergence rate which is also independent of the communication network, we need $\fronorm{\vec{R}^{(0)}}=\bigO{K^{-\frac{1}{3}}}$, which occurs when the initial batch size, denoted as $m_0$, is large enough; by Jensen's inequality, this will in turn make $\norm{\bar{\vec{r}}^{(0)}}_2^2=\bigO{K^{-\frac{1}{3}}}$. Additionally, we make a standard assumption~\cite{lian17,tang18,xin21hsgd} that $\blx_i^{(0)}=\blx_j^{(0)}$ for all $i$ and $j$; this eliminates the $\fronorm{\x_\perp^{(0)}}$ error. For the gradient tracking term $\fronorm{\y_\perp^{(0)}}$, we need to perform sufficiently many \emph{initial} communications so that $\gamma_1^{(0)}\fronorm{\y_\perp^{(0)}}=\bigO{1}$ independent of $\tilde{\rho}$. Notice that $\frac{1}{NL(1-\tilde{\rho})}=\bigO{1}$ if $NL$ is sufficiently large and $\tilde{\rho}$ is not too close to 1; in the following Corollary, we assume the worst case so that $\frac{1}{NL(1-\tilde{\rho})}=\bigO{\frac{1}{(1-\tilde{\rho})}}.$
\end{remark}

\begin{corollary}\label{corollary:constant_convergence}
	Let $\varepsilon>0$ be given and assume that $L\ge1$. Under the same conditions as in Theorem~\ref{theorem:constant_convergence}, let $\blx_i^{(0)}=\blx_j^{(0)}$ for all $i,j=1,\dots,N$, let the initial batch size $m_0=\sqrt[3]{NK}$ for all $i=1,\dots,N$, and perform $T_0=\left\lceil\frac{-2\ln(1-\tilde{\rho})}{\sqrt{1-\rho}}\right\rceil$ communications by Algorithm~\ref{algo:cheby} for the initial gradient tracking update in line 1 of Algorithm~\ref{algo:}. Let $\tilde{\vec{v}}_i^{(k+1)}$ be any unbiased gradient estimator such that either~\eqref{unbiased_assumption} or~\eqref{v2_update} holds, let the local batch size $m=\bigO{1}$ for all remaining iterations, and choose $\alpha$ such that
		\begin{equation}\label{corollary:constant_convergence:alpha}
			\alpha=\frac{N^{\frac{2}{3}}}{64L}.
		\end{equation}
	Then, provided $K\ge \frac{N^2}{(1-\tilde{\rho})^6}$, Algorithm~\ref{algo:} produces a stochastic $\varepsilon$-stationary point as defined in Definition~\ref{def:stationarity} in
		\begin{equation}\label{corollary:constant_convergence:gradients_and_comms}
			K=\bigO{\max\left\{\frac{(L\Delta)^{\frac{3}{2}}+\hat{\sigma}^3}{N\varepsilon^{\frac{3}{2}}},\frac{\hat{\sigma}^2}{(1-\tilde{\rho})^2\varepsilon},\frac{\sqrt{N}\hat{\sigma}^{\frac{3}{2}}}{\varepsilon^{\frac{3}{4}}}\right\}}
		\end{equation}
	local stochastic gradient computations and $T_0+ T(K-1)$ neighbor communications for any $T\ge1$, where $\Delta\triangleq\Phi^{(0)}-\phi^*$, with $\Phi^{(0)}$ defined in~\eqref{lemma:base_change_update:lyapunov}.
\end{corollary}
\begin{proof}
	First, notice that if $K\ge\frac{N^2}{(1-\tilde{\rho})^6},$ then $K^{\frac{1}{3}}\ge \frac{N^{\frac{2}{3}}}{(1-\tilde{\rho})^2},$ so the choice of $\alpha$ in~\eqref{corollary:constant_convergence:alpha} satisfies $\alpha\le\frac{(1-\rho)^2K^{\frac{1}{3}}}{64L}$ since
		\begin{align*}
			\alpha=\frac{N^{\frac{2}{3}}}{64L}\le\frac{(1-\rho)^2K^{\frac{1}{3}}}{64L}.
		\end{align*}
	Next, by~\eqref{lemma:base_change_update:lyapunov} and~\eqref{theorem:constant_convergence:params}, we have
		\begin{equation}\label{corollary:constant_convergence:lyapunov}
			\Phi^{(0)}=\phi(\bar{\blx}^{(0)})+\frac{1}{NL(1-\tilde{\rho})}\fronorm{\y_\perp^{(0)}}+\frac{1024LK^{\frac{1}{3}}}{N^{\frac{5}{3}}(1-\tilde{\rho})}\fronorm{\x_\perp^{(0)}}+\frac{4K^{\frac{1}{3}}}{3N^{\frac{5}{3}}L}\fronorm{\vec{R}^{(0)}}+\frac{4N^{\frac{1}{3}}K^{\frac{1}{3}}}{3L}\norm{\bar{\vec{r}}^{(0)}}_2^2.
		\end{equation}
	Notice that
		\begin{align*}
			\frac{1}{N}\Exp\fronorm{\vec{R}^{(0)}}=\frac{1}{N}\sum_{i=1}^N\Exp\norm{\vec{d}_i^{(0)}-\nabla f_i(\blx_i^{(0)})}_2^2\le\frac{\sigma^2}{m_0}
		\end{align*}
	by~\eqref{assumption:variance} since $\vec{d}_i^{(0)}=\frac{1}{m_0}\sum_{\xi\in B_i^{(0)}}\nabla f_i(\blx_i^{(0)},\xi)$ for all $i=1,\dots,N$. Hence with $m_0=\sqrt[3]{NK}$, it holds that
		\begin{align*}
			\frac{4K^{\frac{1}{3}}}{3N^{\frac{5}{3}}L}\fronorm{\vec{R}^{(0)}}\le\frac{4\sigma^2}{3NL}\le\frac{4\sigma^2}{3L}
		\end{align*}
	which is independent of $K$. By Jensen's inequality, we further have
		\begin{align*}
			\frac{4N^{\frac{1}{3}}K^{\frac{1}{3}}}{3L}\norm{\bar{\vec{r}}^{(0)}}_2^2\le\frac{4K^{\frac{1}{3}}}{N^{\frac{2}{3}}3L}\fronorm{\vec{R}^{(0)}}\le\frac{4\sigma^2}{3L},
		\end{align*}
	which is also independent of $K.$ Next, notice that $\y^{(0)}=\w_{T_0}(\d^{(0)})$ by line 1 in Algorithm~\ref{algo:}; hence it holds
		\begin{align}\label{corollary:constant_convergence:initial_y}
			\frac{1}{NL(1-\tilde{\rho})}\fronorm{\y_\perp^{(0)}}\stack{\eqref{lemma:y_d_relation:bound}}{=}\frac{1}{NL(1-\tilde{\rho})}\fronorm{\w_{T_0}(\d^{(0)})-\bar{\d}^{(0)}}\stack{\eqref{lemma:chebyshev:bound}}{\le}\frac{4\left(1-\sqrt{1-\rho}\right)^{2T_0}}{(1-\tilde{\rho})}\fronorm{\d^{(0)}-\bar{\d}^{(0)}}
		\end{align}
	where we have also used $N,L\ge1.$ By the choice of $T_0=\left\lceil\frac{-2\ln(1-\tilde{\rho})}{\sqrt{1-\rho}}\right\rceil$, we have
		\begin{align}\label{corollary:constant_convergence:initial_comms}
			T_0=\left\lceil\frac{-2\ln(1-\tilde{\rho})}{\sqrt{1-\rho}}\right\rceil\ge\frac{-2\ln(1-\tilde{\rho})}{\sqrt{1-\rho}}\stack{\eqref{lemma:initial_comms:bound}}{\ge}\frac{2\ln\left(1-\tilde{\rho}\right)}{\ln\left(1-\sqrt{1-\rho}\right)}
		\end{align}
	where we have used Lemma~\eqref{lemma:initial_comms} with $x=\sqrt{1-\rho}$ and $\ln(1-\tilde{\rho})\le0$. By~\eqref{corollary:constant_convergence:initial_comms}, it holds that
		\begin{align}\label{corollary:constant_convergence:initial_y_perp}
			\frac{4\left(1-\sqrt{1-\rho}\right)^{2T_0}}{(1-\tilde{\rho})}\le\frac{4\left(1-\sqrt{1-\rho}\right)^{\frac{4\ln\left(1-\tilde{\rho}\right)}{\ln\left(1-\sqrt{1-\rho}\right)}}}{(1-\tilde{\rho})}\le4
		\end{align}
	since $4\ln(1-\tilde{\rho})\le\ln(1-\tilde{\rho})$ as $\tilde{\rho}\in[0,1).$ Thus, by $\blx_i^{(0)}=\blx_j^{(0)}$ for all $i,j=1,\dots,N$, we have that $\Delta$ is independent of $\tilde{\rho},N$, and $K$. Hence, for $\tau$ chosen uniformly at random from $\{0,\dots,K-1\}$, we have
		\begin{align*}
					\frac{1}{N}\sum_{i=1}^N\Exp\norm{P\left(\vec{z}_i^{(\tau)},\nabla f(\vec{z}_i^{(\tau)}),\alpha_\tau\right)}_2^2+\frac{L^2}{N}\Exp\fronorm{\z_\perp^{(\tau)}}\le\varepsilon,
		\end{align*}
	provided
		\begin{equation}\label{corollary:constant_convergence:k_bound}
			K=\bigO{\max\left\{\frac{\Delta^{\frac{3}{2}}}{\left(\alpha\varepsilon\right)^{\frac{3}{2}}},\frac{L^3\hat{\sigma}^3\alpha^3}{N^3\varepsilon^{\frac{3}{2}}},\frac{L^3\alpha^3\hat{\sigma}^2}{N^2(1-\tilde{\rho})^2\varepsilon},\frac{L^3\alpha^3\hat{\sigma}^{\frac{3}{2}}}{N^{\frac{3}{2}}\varepsilon^{\frac{3}{4}}}\right\}}.
		\end{equation}
	Plugging $\alpha=\frac{N^{\frac{2}{3}}}{64L}$ into~\eqref{corollary:constant_convergence:k_bound} results in
		\begin{align*}
			K=\bigO{\max\left\{\frac{(L\Delta)^{\frac{3}{2}}+\hat{\sigma}^3}{N\varepsilon^{\frac{3}{2}}},\frac{\hat{\sigma}^2}{(1-\tilde{\rho})^2\varepsilon},\frac{\sqrt{N}\hat{\sigma}^{\frac{3}{2}}}{\varepsilon^{\frac{3}{4}}}\right\}}.
		\end{align*}
	Hence, the number of gradient evaluations is
		\begin{align*}
			\mathcal{K}\triangleq m(K-1)=\bigO{\max\left\{\frac{(L\Delta)^{\frac{3}{2}}+\hat{\sigma}^3}{N\varepsilon^{\frac{3}{2}}},\frac{\hat{\sigma}^2}{(1-\tilde{\rho})^2\varepsilon},\frac{\sqrt{N}\hat{\sigma}^{\frac{3}{2}}}{\varepsilon^{\frac{3}{4}}}\right\}},
		\end{align*}
	which yields a total number of gradient evaluations $\lceil\mathcal{K}+\sqrt[3]{NK}\rceil$, provided $K\ge\frac{N^2}{(1-\tilde{\rho})^6}$. Since $\mathcal{K}=\Omega\left(\sqrt[3]{NK}\right)$, we drop the $\sqrt[3]{NK}$ and obtain~\eqref{corollary:constant_convergence:gradients_and_comms}.
\end{proof}

\begin{remark}\label{remark:complexity_discussion}
	Similar to other works~\cite{lian17,xin21_dsgt,xin21hsgd}, we have a minimum requirement on the number of iterations, called \emph{transient iterations}, in order to achieve the complexity results in~\eqref{corollary:constant_convergence:gradients_and_comms}. Further, we notice that if the connectivity of the original network is poor, i.e. $\rho\approx1$ for $\rho$ defined in~\eqref{spectral_gap}, and we only perform one neighbor communication during lines 3 and 6 in Algorithm~\ref{algo:} so that $\tilde{\rho}=\rho$, then it could be that $\frac{\hat{\sigma}^2}{(1-\rho)^2\varepsilon}$ dominates in~\eqref{corollary:constant_convergence:gradients_and_comms}, meaning \ourmethod is network-dependent. Similar to~\cite{xin21hsgd}, we can place a requirement that $\varepsilon\le N^{-2}(1-\rho)^4$, in which case \ourmethod achieves the optimal complexity result and is independent of the communication network. In order to relax this requirement to $\varepsilon\le N^{-2}$ (which can be significantly greater than $N^{-2}(1-\rho)^4$), we perform Algorithm~\ref{algo:cheby} during the neighbor communications (lines 3 and 6 in Algorithm~\ref{algo:}) such that for $T=\lceil\frac{2}{\sqrt{1-\rho}}\rceil$, $\tilde{\rho}\le\frac{1}{2}$ by~\eqref{appendix:cheby:upper_bound}, so that the number of local gradient computations becomes
		\begin{equation}\label{constant_step:chebyshev_grads}
			\bigO{\max\left\{\frac{(L\Delta)^{\frac{3}{2}}+\hat{\sigma}^3}{N\varepsilon^{\frac{3}{2}}},\frac{\sqrt{N}\hat{\sigma}^{\frac{3}{2}}}{\varepsilon^{\frac{3}{4}}}\right\}}
		\end{equation}
	which is independent of $\rho$. Additionally, the number of local neighbor communications becomes
		\begin{equation}\label{constant_step:chebyshev_comms}
					T_0+TK=\bigO{\frac{1}{\sqrt{1-\rho}}\max\left\{\frac{(L\Delta)^{\frac{3}{2}}+\hat{\sigma}^3}{N\varepsilon^{\frac{3}{2}}},\frac{\sqrt{N}\hat{\sigma}^{\frac{3}{2}}}{\varepsilon^{\frac{3}{4}}}\right\}}
		\end{equation}
	which is optimal in terms of the dependence upon $\rho$~\cite{scaman17}.
\end{remark}

\subsection{Diminishing step size}\label{appendix:diminishing_convergence}

The Lyapunov function and relation defined in~\eqref{lemma:base_change} are specially designed for the constant step size proof. Here, we make analogous designs for the diminishing step size proof.

\begin{lemma}\label{lemma:base_change_diminishing}
	For all $k\ge0,$ the following inequality holds
		\begin{equation}\label{lemma:base_change_diminishing:bound}
			\begin{split}
			&\Exp\left[\hat{\Phi}^{(k+1)}-\hat{\Phi}^{(k)}\right]\\
								\le&\underbrace{\left(4\gamma_3^{(k)}L^2(1-\beta_k)^2-\frac{1}{2N}\left(\frac{1}{2\alpha_k}-3L\right)\right)}_{(A)}\Exp\fronorm{\x^{(k+1)}-\bar{\x}^{(k)}}\\
								&+\underbrace{\left(\frac{16\gamma_1^{(k)}L^2}{1-\tilde{\rho}}-\frac{1}{2N\alpha_k}\right)}_{(B)}\Exp\fronorm{\x^{(k+1)}-\z^{(k)}}\\
								&+\underbrace{\left(\frac{L}{2N}+\frac{\tilde{\rho}^2}{2N\alpha_k}+\frac{64\gamma_1^{(k)}L^2}{1-\tilde{\rho}}+\gamma_2^{(k)}\tilde{\rho}+4\gamma_3^{(k)}L^2(1-\beta_k)^2-\gamma_2^{(k-1)}\right)}_{(C)}\Exp\fronorm{\x_\perp^{(k)}}\\
								&+\underbrace{\left(\frac{1}{2NL}+\gamma_1^{(k)}\tilde{\rho}+\gamma_2^{(k)}\frac{\alpha_k^2}{1-\tilde{\rho}}-\gamma_1^{(k-1)}\right)}_{(D)}\Exp\fronorm{\y_\perp^{(k)}}\\
								&+\underbrace{\left(\frac{\alpha_k}{N}+\frac{4\gamma_1^{(k)}\beta_k^2}{1-\tilde{\rho}}+\gamma_3^{(k)}(1-\beta_k)^2-\gamma_3^{(k-1)}\right)}_{(E)}\Exp\fronorm{\vec{R}^{(k)}}\\
								&+\left(\gamma_3^{(k)}+\frac{2\gamma_1^{(k)}}{1-\tilde{\rho}}\right)2\hat{\sigma}^2N\beta_k^2
			\end{split}
		\end{equation}
	where $\gamma_1^{(k)},\gamma_2^{(k)},\gamma_3^{(k)}$ are strictly positive values and
		\begin{equation}\label{lemma:base_change_diminishing:lyapunov}
			\hat{\Phi}^{(k)}\triangleq\phi(\barxk)+\gamma_1^{(k-1)}\fronorm{\y_\perp^{(k)}}+\gamma_2^{(k-1)}\fronorm{\x_\perp^{(k)}}+\gamma_3^{(k-1)}\fronorm{\vec{R}^{(k)}}
		\end{equation}
	is a lower bounded Lyapunov function.
\end{lemma}
\begin{proof}
	We start by using part \textit{(iv)} of Assumption~\ref{assumption:mixing_matrix} and~\eqref{algo:comm_update} to note that
		\begin{align}\label{lemma:base_change_diminishing:eqn1}
			\frac{1}{2N\alpha_k}\fronorm{\bar{\x}^{(k)}-\z^{(k)}}&=\frac{1}{2N\alpha_k}\fronorm{\left(\w_T-\avg\right)\left(\identity-\avg\right)\x^{(k)}}\nonumber\\
			\stack{\eqref{spectral_gap_cheby}}{\le}&\frac{\tilde{\rho}^2}{2N\alpha_k}\fronorm{\x_\perp^{(k)}}.
		\end{align}
	Next, we utilize the Peter-Paul inequality and Jensen's inequality to have
		\begin{align}\label{lemma:base_change_diminishing:eqn2}
			-\ip{\bar{\vec{r}}^{(k)},\barxknew-\barxk}\le&\alpha_k\norm{\bar{\vec{r}}^{(k)}}_2^2+\frac{1}{4\alpha_k}\norm{\barxknew-\barxk}_2^2\nonumber\\
			\le&\frac{\alpha_k}{N}\sum_{i=1}^N\norm{\vec{r}_i^{(k)}}_2^2+\frac{1}{4N\alpha_k}\sum_{i=1}^N\norm{\xiknew-\barxk}_2^2\nonumber\\
			=&\frac{\alpha_k}{N}\fronorm{\vec{R}^{(k)}}+\frac{1}{4N\alpha_k}\fronorm{\x^{(k+1)}-\bar{\x}^{(k)}}.
		\end{align}
	Applying~\eqref{lemma:base_change_diminishing:eqn1} and~\eqref{lemma:base_change_diminishing:eqn2} to~\eqref{lemma:second_change:bound} results in
		\begin{align}
			&\phi(\barxknew)-\phi(\barxk)\nonumber\\
			\le&-\frac{1}{2N}\left(\frac{1}{\alpha_k}-3L-\frac{1}{2\alpha_k}\right)\fronorm{\x^{(k+1)}-\bar{\x}^{(k)}}-\frac{1}{2N\alpha_k}\fronorm{\x^{(k+1)}-\z^{(k)}}\nonumber\\
								&+\left(\frac{L}{2N}+\frac{\tilde{\rho}^2}{2N\alpha_k}\right)\fronorm{\x_\perp^{(k)}}+\frac{1}{2NL}\fronorm{\y_\perp^{(k)}}+\frac{\alpha_k}{N}\fronorm{\vec{R}^{(k)}}.\label{lemma:base_change_diminishing:eqn3}
		\end{align}
	Next, from Lemma~\eqref{lemma:storm_error}, we use Young's inequality to have
		\begin{align}
			&\Exp\fronorm{\vec{R}^{(k+1)}}\nonumber\\
			\stack{\eqref{lemma:storm_error:bound}}{\le}&\enskip2N\beta_k^2\hat{\sigma}^2+2(1-\beta_k)^2L^2\Exp\fronorm{\x^{(k+1)}-\x^{(k)}}+(1-\beta_k)^2\Exp\fronorm{\vec{R}^{(k)}}\nonumber\\
			\le&2N\beta_k^2\hat{\sigma}^2+4(1-\beta_k)^2L^2\left(\Exp\fronorm{\x^{(k+1)}-\bar{\x}^{(k)}}+\fronorm{\x_\perp^{(k)}}\right)+(1-\beta_k)^2\Exp\fronorm{\vec{R}^{(k)}}.\label{lemma:base_change_diminishing:eqn4}
		\end{align}
	Adding $\gamma_1^{(k)}\fronorm{\y_\perp^{(k+1)}}$, $\gamma_2^{(k)}\fronorm{\x_\perp^{(k+1)}}$, $\gamma_3^{(k)}\fronorm{\vec{R}^{(k+1)}}$ to both sides of~\eqref{lemma:base_change_diminishing:eqn3} and taking the full expectation in conjunction with the results from Lemmas~\ref{lemma:matrix_bounds} and~\eqref{lemma:base_change_diminishing:eqn4} yields
		\begin{align*}
			&\Exp\left[\phi(\barxknew)-\phi(\barxk)\right]+\gamma_1^{(k)}\Exp\fronorm{\y_\perp^{(k+1)}}+\gamma_2^{(k)}\Exp\fronorm{\x_\perp^{(k+1)}}+\gamma_3^{(k)}\Exp\fronorm{\vec{R}^{(k+1)}}\\
								\le&\left(4\gamma_3^{(k)}L^2(1-\beta_k)^2-\frac{1}{2N}\left(\frac{1}{2\alpha_k}-3L\right)\right)\Exp\fronorm{\x^{(k+1)}-\bar{\x}^{(k)}}\\
								&+\frac{8\gamma_1^{(k)}L^2}{1-\tilde{\rho}}\Exp\fronorm{\x^{(k+1)}-\x^{(k)}}-\frac{1}{2N\alpha_k}\Exp\fronorm{\x^{(k+1)}-\z^{(k)}}\\
								&+\left(\frac{L}{2N}+\frac{\tilde{\rho}^2}{2N\alpha_k}+\gamma_2^{(k)}\tilde{\rho}+4\gamma_3^{(k)}L^2(1-\beta_k)^2\right)\Exp\fronorm{\x_\perp^{(k)}}\\
								&+\left(\frac{1}{2NL}+\gamma_1^{(k)}\tilde{\rho}+\gamma_2^{(k)}\frac{\alpha_k^2}{1-\tilde{\rho}}\right)\Exp\fronorm{\y_\perp^{(k)}}\\
								&+\left(\frac{\alpha_k}{N}+\frac{4\gamma_1^{(k)}\beta_k^2}{1-\tilde{\rho}}+\gamma_3^{(k)}(1-\beta_k)^2\right)\Exp\fronorm{\vec{R}^{(k)}}\\
								&+\left(\gamma_3^{(k)}+\frac{2\gamma_1^{(k)}}{1-\tilde{\rho}}\right)2\hat{\sigma}^2N\beta_k^2.
		\end{align*}
Next we apply the following bound to the above relation, 
		\begin{align*}
				\fronorm{\x^{(k+1)}-\x^{(k)}}&=\fronorm{\x^{(k+1)}-\z^{(k)}+\z^{(k)}-\x^{(k)}}\\
						&\le2\fronorm{\x^{(k+1)}-\z^{(k)}}+2\fronorm{\w_T(\x^{(k)})-\w_T\bar{\x}^{(k)}+\w_T\bar{\x}^{(k)}-\x^{(k)}}\\
						&=2\fronorm{\x^{(k+1)}-\z^{(k)}}+2\fronorm{\left(\identity-\w_T\right)\x_\perp^{(k)}}\\
						&\le2\fronorm{\x^{(k+1)}-\z^{(k)}}+8\fronorm{\x_\perp^{(k)}},
		\end{align*}
	where we have used Assumption~\ref{assumption:mixing_matrix} part \textit{(iv)} to bound $\norm{\identity-\w_T}_2\le2$. Finally, we subtract $\gamma_1^{(k-1)}\Exp\fronorm{\y_\perp^{(k)}},$ $\gamma_2^{(k-1)}\Exp\fronorm{\x_\perp^{(k)}},$ and $\gamma_3^{(k-1)}\Exp\fronorm{\vec{R}^{(k)}}$ from both sides to complete the proof. The lower boundedness of~\eqref{lemma:base_change_diminishing:lyapunov} follows from the non-negativity of the Frobenius norm and Assumption~\ref{assumption:objective_function} (\textit{iv}).
\end{proof}

\begin{lemma}\label{lemma:diminishing_prep}
	Let $x\in[0,1).$ Then for any $y\ge\lceil\frac{2}{1-x^3}\rceil$, it holds that
		\begin{equation}\label{lemma:diminishing_prep:bound}
			\left(\frac{y-1}{y}\right)^{\frac{1}{3}}-x>\frac{1-x}{2}.
		\end{equation}
\end{lemma}
\begin{proof}
	The proof begins by analyzing $h(x)\triangleq x^3-x^2-x+1$ for $x\in[0,1).$ Notice that $h(x)$ is decreasing on $[0,1)$ by $h'(x)=3x^2-2x-1<0$. Since $h(0)=1$ and $h(1-)=0$, it holds that $h(x)>0$ for $x\in[0,1).$ Hence
		\begin{align*}
			3+3x^3>&3x+3x^2.
		\end{align*}
	Adding $1+x^3$ to both sides results in
		\begin{align*}
			4+4x^3>&(1+x)^3.
		\end{align*}
	Dividing by $4$, adding $1$ to both sides, and rearranging results in
		\begin{align*}
			2-\frac{1}{4}(1+x)^3>&1-x^3.
		\end{align*}
	Since $x<1$, we divide both sides by $(1-x^3)$ and use $\lceil a\rceil\ge a$ for any $a\in\R$ to have
		\begin{align*}
			y\left(1-\left(\frac{1+x}{2}\right)^{3}\right)\ge\lceil\frac{2}{1-x^3}\rceil\left(1-\left(\frac{1+x}{2}\right)^{3}\right)\ge&\left(\frac{2}{1-x^3}\right)\left(1-\left(\frac{1+x}{2}\right)^{3}\right)>1
		\end{align*}
	where we have used $y\ge\lceil\frac{2}{1-x^3}\rceil$. Rearranging 
		\begin{align*}
			y\left(1-\left(\frac{1+x}{2}\right)^{3}\right)&>1
		\end{align*}
	results in 
		\begin{align*}
			\frac{y-1}{y}>\left(\frac{1+x}{2}\right)^{3}.
		\end{align*}
	Taking the cube-root and subtracting $x$ from both sides completes the proof.
\end{proof}

\subsubsection{Proof of Theorem~\ref{theorem:diminishing_convergence}}\label{appendix:diminishing_proof}

\begin{proof}
	The proof follows similar steps as the proof for Theorem~\ref{theorem:constant_convergence}. We frequently use~\eqref{theorem:diminishing_choice:bounds} to have $\alpha_k\le\min\{\frac{1}{32L},\frac{(1-\tilde{\rho})^2}{64L}\}.$ First, we show $\beta_k\in(0,1)$ for all $k\ge0$. By $\alpha\le\frac{(1-\tilde{\rho})^2k_0^{\frac{1}{3}}}{64L}<\frac{k_0^{\frac{1}{3}}}{\sqrt{480}L}$ it holds
			\begin{align}
				L^2\alpha_{k+1}^2<\frac{1}{480}=\frac{1}{10\cdot48}&<\frac{1}{48}\left(\frac{1}{4}+\left(\frac{2}{3}\right)^{\frac{1}{3}}-1\right)\nonumber\\
				&\le\frac{1}{48}\left(\frac{1}{4}+\left(\frac{k_0}{k_0+1}\right)^{\frac{1}{3}}-1\right)\nonumber\\
				&<\frac{1}{48}\left(\frac{(k_0+1)^{\frac{1}{3}}}{2k_0^{\frac{1}{3}}+(k_0+1)^{\frac{1}{3}}}+\frac{\alpha_{k+1}}{\alpha_k}-1\right)\label{theorem:diminishing_convergence:beta_proof}
			\end{align}
	where the first inequality uses $k_0\ge\lceil\frac{2}{1-\tilde{\rho}^3}\rceil\ge2$ and the last uses $2k_0^{\frac{1}{3}}<3(k_0+1)^{\frac{1}{3}}$ for all $k_0\ge2$. Rearranging~\eqref{theorem:diminishing_convergence:beta_proof} results in
		\begin{equation}\label{theorem:diminishing_convergence:beta_term}
			\beta_k<\frac{(k_0+1)^{\frac{1}{3}}}{2k_0^{\frac{1}{3}}+(k_0+1)^{\frac{1}{3}}}<1
		\end{equation}
	where the right most inequality uses $2k_0^{\frac{1}{3}}>0$ by $k_0\ge2.$ Since $1-\frac{\alpha_{k+1}}{\alpha_k}>0$ by $\alpha_k>\alpha_{k+1}$, we also have $0<\beta_k.$\\
	Second, let 
		\begin{equation}\label{theorem:diminishing_convergence:params}
				\gamma_1^{(k)}\triangleq\frac{1}{NL(1-\tilde{\rho})}\text{, }\gamma_2^{(k)}\triangleq\frac{16}{N(1-\tilde{\rho})\alpha_k}\text{, and }\gamma_3^{(k)}\triangleq\frac{1}{24NL^2\alpha_{k+1}}
		\end{equation}
	in~\eqref{lemma:base_change_diminishing:bound} to have
		\begin{align}
					&\underbrace{\left(\frac{1}{4N\alpha_k}-\frac{(1-\beta_k)^2}{6N\alpha_{k+1}}-\frac{3L}{2N}\right)}_{(A')}\Exp\fronorm{\x^{(k+1)}-\bar{\x}^{(k)}}\nonumber\\
					&+\underbrace{\left(\frac{1}{2N\alpha_k}-\frac{16L}{N(1-\tilde{\rho})^2}\right)}_{(B')}\Exp\fronorm{\x^{(k+1)}-\z^{(k)}}\nonumber\\
					&+\underbrace{\left(\frac{16}{N(1-\tilde{\rho})\alpha_{k-1}}-\frac{16\tilde{\rho}}{N(1-\tilde{\rho})\alpha_{k}}-\frac{L}{2N}-\frac{\tilde{\rho}^2}{2N\alpha_k}-\frac{64L}{N(1-\tilde{\rho})^2}-\frac{(1-\beta_k)^2}{6N\alpha_{k+1}}\right)}_{(C')}\Exp\fronorm{\x_\perp^{(k)}}\nonumber\\
					&+\underbrace{\left(\frac{1}{NL}-\frac{1}{2NL}-\frac{16\alpha_k}{N(1-\tilde{\rho})^2}\right)}_{(D')}\Exp\fronorm{\y_\perp^{(k)}}\nonumber\\
					&+\underbrace{\left(\frac{1}{24NL^2\alpha_k}-\frac{(1-\beta_k)^2}{24NL^2\alpha_{k+1}}-\frac{4\beta_k^2}{NL(1-\tilde{\rho})^2}-\frac{\alpha_k}{N}\right)}_{(E')}\Exp\fronorm{\vec{R}^{(k)}}\nonumber\\
						\le&\Exp\left[\hat{\Phi}^{(k)}-\hat{\Phi}^{(k+1)}\right]+\left(\frac{1}{24NL^2\alpha_{k+1}}+\frac{2}{NL(1-\tilde{\rho})^2}\right)2N\beta_k^2\hat{\sigma}^2.\label{theorem:diminishing_convergence:relation}
				\end{align}
	Next we lower bound ($A'$) - ($E'$). For ($A'$), since $(1-\beta_k)^2<1$ we have
		\begin{align*}
			(A')=\frac{1}{4N\alpha_k}-\frac{(1-\beta_k)^2}{6N\alpha_{k+1}}-\frac{3L}{2N}&>\frac{1}{N\alpha_k}\left(\frac{1}{4}-\frac{1}{6}\left(\frac{k+k_0+1}{k+k_0}\right)^{\frac{1}{3}}-\frac{3L\alpha}{2k_0^{\frac{1}{3}}}\right)\\
			&\ge\frac{1}{N\alpha_k}\left(\frac{1}{4}-\frac{1}{6}\left(\frac{3}{2}\right)^{\frac{1}{3}}-\frac{3}{64}\right)\\
			&>\frac{1}{96N\alpha_k}
		\end{align*}
	where we have used $k_0\ge2$ and $\alpha\le\frac{k_0^{\frac{1}{3}}}{32L}$. For ($B'$), we use $\alpha\le\frac{(1-\tilde{\rho})^2k_0^{\frac{1}{3}}}{64L}$ so $\alpha_k\le\frac{(1-\tilde{\rho})^2}{64L}$ to have
		\begin{align*}
			(B')=\frac{1}{2N\alpha_k}-\frac{16L}{N(1-\tilde{\rho})^2}&\ge\frac{1}{N\alpha_k}\left(\frac{1}{2}-\frac{1}{4}\right)=\frac{1}{4N\alpha_k}.
		\end{align*}
	For ($C'$), we have
		\begin{align*}
			(C')=&\frac{16}{N(1-\tilde{\rho})\alpha_{k-1}}-\frac{16\tilde{\rho}}{N(1-\tilde{\rho})\alpha_{k}}-\frac{L}{2N}-\frac{\tilde{\rho}^2}{2N\alpha_k}-\frac{64L}{N(1-\tilde{\rho})^2}-\frac{(1-\beta_k)^2}{6N\alpha_{k+1}}\\
			=&\frac{1}{N\alpha_k}\left(\frac{16}{(1-\tilde{\rho})}\left(\frac{\alpha_k}{\alpha_{k-1}}-\tilde{\rho}\right)-\frac{L\alpha_k}{2}-\frac{\tilde{\rho}^2}{2}-\frac{64L\alpha_k}{(1-\tilde{\rho})^2}-\frac{(1-\beta_k)^2\alpha_k}{6\alpha_{k+1}}\right)\\
			\ge&\frac{1}{N\alpha_k}\left(\frac{16}{(1-\tilde{\rho})}\left(\left(\frac{k_0-1}{k_0}\right)^{\frac{1}{3}}-\tilde{\rho}\right)-\frac{L\alpha_k}{2}-\frac{\tilde{\rho}^2}{2}-\frac{64L\alpha_k}{(1-\tilde{\rho})^2}-\frac{(1-\beta_k)^2\alpha_k}{6\alpha_{k+1}}\right).
		\end{align*}
	Next, using Lemma~\ref{lemma:diminishing_prep}, since $\tilde{\rho}\in[0,1)$ and $k_0\ge\lceil\frac{2}{1-\tilde{\rho}^3}\rceil$, it holds that
		\begin{align*}
			\left(\frac{k_0-1}{k_0}\right)^{\frac{1}{3}}-\tilde{\rho}\ge\frac{1-\tilde{\rho}}{2}.
		\end{align*}
	Thus, using $(1-\beta_k)^2<1$ and $\alpha_k\le\min\{\frac{1}{32L},\frac{(1-\tilde{\rho})^2}{64L}\}$, it holds
		\begin{align*}
				(C')>&\frac{1}{N\alpha_k}\left(8-\frac{L\alpha_k}{2}-\frac{\tilde{\rho}^2}{2}-\frac{64L\alpha_k}{(1-\tilde{\rho})^2}-\frac{(1-\beta_k)^2\alpha_k}{6\alpha_{k+1}}\right)\\
				>&\frac{1}{N\alpha_k}\left(8-\frac{1}{64}-\frac{1}{2}-1-\frac{1}{6}\left(\frac{k+k_0+1}{k+k_0}\right)^{\frac{1}{3}}\right)\\
                >&\frac{1}{N\alpha_k}\left(8-\frac{1}{64}-\frac{1}{2}-1-\frac{1}{6}\left(\frac{3}{2}\right)^{\frac{1}{3}}\right)\\
				>&\frac{4}{N\alpha_k}.
		\end{align*}
	For ($D'$), we use $\alpha\le\frac{(1-\tilde{\rho})^2k_0^{\frac{1}{3}}}{64L}$ so $\alpha_k\le\frac{(1-\tilde{\rho})^2}{64L}$ to have
		\begin{align*}
			(D')=\frac{1}{NL}-\frac{1}{2NL}-\frac{16\alpha_k}{N(1-\tilde{\rho})^2}&\ge\frac{1}{2NL}-\frac{16\alpha}{N(1-\tilde{\rho})^2k_0^{\frac{1}{3}}}\\
			\ge&\frac{1}{4NL}.
		\end{align*}
	For ($E'$), we factor out $\frac{1}{24NL^2\alpha_{k+1}}$ and expand $(1-\beta_k)^2$ to have
		\begin{align*}
			(E')=&\frac{1}{24NL^2\alpha_{k+1}}\left(\frac{\alpha_{k+1}}{\alpha_{k}}-1+2\beta_k-\beta_k^2-24L^2\alpha_{k+1}\alpha_{k}-\frac{96L\alpha_{k+1}\beta_k^2}{(1-\tilde{\rho})^2}\right)\\
			\ge&\frac{1}{24NL^2\alpha_{k+1}}\left(\frac{\alpha_{k+1}}{\alpha_{k}}-1+2\beta_k-\beta_k^2-48L^2\alpha_{k+1}^2-\frac{96L\alpha_{k+1}\beta_k^2}{(1-\tilde{\rho})^2}\right),
		\end{align*}
	where we have used $\alpha_k\le2\alpha_{k+1}$. Plugging in the definition of $\beta_k$ from~\eqref{theorem:diminishing_choice:bounds} and using $\beta_k\ge48L^2\alpha_{k+1}^2$ gives
		\begin{align*}
			(E')\ge&\frac{\beta_k}{24NL^2\alpha_{k+1}}\left(1-\beta_k\left(1+\frac{96L\alpha_{k+1}}{(1-\tilde{\rho})^2}\right)\right)\\
			\ge&\frac{\beta_k}{24NL^2\alpha_{k+1}}\left(1-\beta_k\left(1+\frac{96}{64}\left(\frac{k_0}{k_0+1}\right)^{\frac{1}{3}}\right)\right)\\
			\stack{\eqref{theorem:diminishing_convergence:beta_term}}{\ge}&\frac{\beta_k}{24NL^2\alpha_{k+1}}\left(1-\frac{(k_0+1)^{\frac{1}{3}}}{2k_0^{\frac{1}{3}}+(k_0+1)^{\frac{1}{3}}}\left(1+\frac{3}{2}\left(\frac{k_0}{k_0+1}\right)^{\frac{1}{3}}\right)\right)\\
			=&\frac{\beta_k}{24NL^2\alpha_{k+1}}\left(\frac{1}{2}\cdot\frac{k_0^{\frac{1}{3}}}{2k_0^{\frac{1}{3}}+(k_0+1)^{\frac{1}{3}}}\right)\\
			\ge&\frac{2\alpha_{k+1}}{N}\left(\frac{1}{2}\cdot\frac{k_0^{\frac{1}{3}}}{2k_0^{\frac{1}{3}}+(k_0+1)^{\frac{1}{3}}}\right)\\
			\ge&\frac{\alpha_{k}}{2N}\left(\frac{k_0^{\frac{1}{3}}}{2k_0^{\frac{1}{3}}+(k_0+1)^{\frac{1}{3}}}\right)\\
		\end{align*}
	where the second inequality uses $\alpha_{k+1}\le\frac{\alpha}{(k_0+1)^{\frac{1}{3}}}$ and $\alpha\le\frac{(1-\tilde{\rho})^2k_0^{\frac{1}{3}}}{64L}$ and the last inequality uses $\alpha_k\le2\alpha_{k+1}$. Next, we sum~\eqref{theorem:diminishing_convergence:relation} over $k=0$ to $K-1$; using the established lower bounds to have
		\begin{align}
			&\sum_{k=0}^{K-1}\frac{1}{96N\alpha_k}\Exp\fronorm{\x^{(k+1)}-\bar{\x}^{(k)}}+\sum_{k=0}^{K-1}\frac{1}{4N\alpha_k}\Exp\fronorm{\x^{(k+1)}-\z^{(k)}}\nonumber\\
			&+\sum_{k=0}^{K-1}\frac{4}{N\alpha_k}\Exp\fronorm{\x_\perp^{(k)}}+\sum_{k=0}^{K-1}\frac{1}{4NL}\Exp\fronorm{\y_\perp^{(k)}}\nonumber\\
			&+\sum_{k=0}^{K-1}\left(\frac{k_0^{\frac{1}{3}}}{2k_0^{\frac{1}{3}}+(k_0+1)^{\frac{1}{3}}}\right)\frac{\alpha_{k}}{2N}\Exp\fronorm{\vec{R}^{(k)}}\nonumber\\
				\le&\left(\hat{\Phi}^{(0)}-\phi^*\right)+\sum_{k=0}^{K-1}\left(\frac{1}{12L^2\alpha_{k+1}}+\frac{4}{L(1-\tilde{\rho})^2}\right)\beta_k^2\hat{\sigma}^2.\label{theorem:diminishing_convergence:pre_bound}
		\end{align}
	 where we have used $\phi^*\le\hat{\Phi}^{(k)}$ for any $k\ge0.$ The final phase of the proof uses Lemma~\ref{lemma:stationary_bound} to provide a concise convergence statement. To do so, we use Jensen's inequality to have
	 	\begin{align*}
	 		\norm{\bar{\vec{r}}^{(k)}}_2^2\le\frac{1}{N}\sum_{i=1}^N\norm{\vec{r}_i^{(k)}}_2^2=\frac{1}{N}\fronorm{\vec{R}^{(k)}}.
	 	\end{align*}
	 Applying this to~\eqref{lemma:stationary_bound:bound}, multiply both sides of~\eqref{lemma:stationary_bound:bound} by $\left(\frac{k_0^{\frac{1}{3}}}{2k_0^{\frac{1}{3}}+(k_0+1)^{\frac{1}{3}}}\right)\alpha_k$, sum from $k=0,\dots,K-1$, and take the expectation to have
	 		\begin{align}
	 			&\left(\frac{k_0^{\frac{1}{3}}}{2k_0^{\frac{1}{3}}+(k_0+1)^{\frac{1}{3}}}\right)\sum_{k=0}^{K-1}\alpha_k\left(\frac{1}{N}\sum_{i=1}^N\Exp\norm{P\left(\vec{z}_i^{(k)},\nabla f(\vec{z}_i^{(k)}),\alpha_k\right)}_2^2\right)\nonumber\\
	 			&+\left(\frac{k_0^{\frac{1}{3}}}{2k_0^{\frac{1}{3}}+(k_0+1)^{\frac{1}{3}}}\right)\sum_{k=0}^{K-1}\alpha_k\left(\frac{L^2}{N}\Exp\fronorm{\x_\perp^{(k)}}+\frac{L^2}{N}\Exp\fronorm{\z_\perp^{(k)}}\right)\nonumber\\
	 				\le&\left(\frac{k_0^{\frac{1}{3}}}{2k_0^{\frac{1}{3}}+(k_0+1)^{\frac{1}{3}}}\right)\sum_{k=0}^{K-1}\alpha_k\left(\frac{6}{N}\Exp\fronorm{\y_\perp^{(k)}}+\frac{6}{N}\Exp\fronorm{\vec{R}^{(k)}}\right)\nonumber\\
	 				&+\left(\frac{k_0^{\frac{1}{3}}}{2k_0^{\frac{1}{3}}+(k_0+1)^{\frac{1}{3}}}\right)\sum_{k=0}^{K-1}\alpha_k\left(\frac{32L^2}{N}\Exp\fronorm{\x_\perp^{(k)}}+\frac{2}{N\alpha_k^2}\Exp\fronorm{\x^{(k+1)}-\z^{(k)}}\right).\label{theorem:diminishing_convergence:eqn1}
	 		\end{align}
	  We relate each of the terms on the right-hand side of~\eqref{theorem:diminishing_convergence:eqn1} to 12 times of the left-hand side of~\eqref{theorem:diminishing_convergence:pre_bound}.  Since $\alpha_k\le\frac{1}{32L}$ and $\left(\frac{k_0^{\frac{1}{3}}}{2k_0^{\frac{1}{3}}+(k_0+1)^{\frac{1}{3}}}\right)\le1$, it holds for all $k\ge0,$
		\begin{equation}\label{theorem:diminishing_convergence:y_bound}
			\left(\frac{k_0^{\frac{1}{3}}}{2k_0^{\frac{1}{3}}+(k_0+1)^{\frac{1}{3}}}\right)\frac{6\alpha_k}{N}\fronorm{\y_\perp^{(k)}}\le\frac{3}{NL}\fronorm{\y_\perp^{(k)}}.
		\end{equation}
	Next, by $\alpha_k\le\frac{1}{32L}$ and $\left(\frac{k_0^{\frac{1}{3}}}{2k_0^{\frac{1}{3}}+(k_0+1)^{\frac{1}{3}}}\right)\le1$, we have
		\begin{equation}\label{theorem:diminishing_convergence:x_bound}
			\left(\frac{k_0^{\frac{1}{3}}}{2k_0^{\frac{1}{3}}+(k_0+1)^{\frac{1}{3}}}\right)\frac{32L^2\alpha_k^2}{N}\fronorm{\x_\perp^{(k)}}\le\frac{48}{N}\fronorm{\x_\perp^{(k)}}.
		\end{equation}
	Additionally, since $\left(\frac{k_0^{\frac{1}{3}}}{2k_0^{\frac{1}{3}}+(k_0+1)^{\frac{1}{3}}}\right)\le1$, it holds that
		\begin{equation}\label{theorem:diminishing_convergence:other_x_bound}
			\left(\frac{k_0^{\frac{1}{3}}}{2k_0^{\frac{1}{3}}+(k_0+1)^{\frac{1}{3}}}\right)\frac{2}{N\alpha_k}\fronorm{\x^{(k+1)}-\z^{(k)}}\le\frac{3}{N\alpha_k}\fronorm{\x^{(k+1)}-\z^{(k)}}.
		\end{equation}
	Combining~\eqref{theorem:diminishing_convergence:y_bound} -~\eqref{theorem:diminishing_convergence:other_x_bound} in conjunction with 12 times of~\eqref{theorem:diminishing_convergence:pre_bound} and~\eqref{theorem:diminishing_convergence:eqn1} results in
		\begin{align*}
			&\left(\frac{k_0^{\frac{1}{3}}}{2k_0^{\frac{1}{3}}+(k_0+1)^{\frac{1}{3}}}\right)\sum_{k=0}^{K-1}\alpha_k\left(\frac{1}{N}\sum_{i=1}^N\Exp\norm{P\left(\vec{z}_i^{(k)},\nabla f(\vec{z}_i^{(k)}),\alpha_k\right)}_2^2\right)\nonumber\\
			&+\left(\frac{k_0^{\frac{1}{3}}}{2k_0^{\frac{1}{3}}+(k_0+1)^{\frac{1}{3}}}\right)\sum_{k=0}^{K-1}\alpha_k\left(\frac{L^2}{N}\Exp\fronorm{\x_\perp^{(k)}}+\frac{L^2}{N}\Exp\fronorm{\z_\perp^{(k)}}\right)\nonumber\\
				\le&12\left(\hat{\Phi}^{(0)}-\phi^*\right)+\sum_{k=0}^{K-1}\left(\frac{1}{L^2\alpha_{k+1}}+\frac{48}{L(1-\tilde{\rho})^2}\right)\beta_k^2\hat{\sigma}^2.
		\end{align*}
	Using $\fronorm{\x_\perp^{(k)}}\ge0$ completes the proof.
\end{proof}

\subsubsection{Complexity analysis}\label{appendix:diminishing_complexity}

\begin{remark}\label{remark:batch_2}
	Similar to Remark~\ref{remark:batch}, the convergence of Algorithm~\ref{algo:} depends upon $\hat{\Phi}^{(0)}$ (see~\eqref{lemma:base_change_diminishing:lyapunov}), but in this setting we do not need a big initial batch, in terms of dependence upon $\varepsilon$. The initial variables must be equal for all agents and the number of initial communications to have $\y^{(0)}$ must be sufficiently large, as in Corollary~\ref{corollary:constant_convergence}.
\end{remark}

\begin{corollary}\label{corollary:diminishing_convergence}
	Let $\varepsilon>0$ be given and assume that $L\ge1$. Under the same conditions as in Theorem~\ref{theorem:diminishing_convergence}, let $\blx_i^{(0)}=\blx_j^{(0)}$ for all $i,j=1,\dots,N$, let the local batch size $m=\bigO{1}$ for all iterations, choose $k_0=\lceil\frac{2}{(1-\tilde{\rho})^6}\rceil$, and perform $T_0=\left\lceil\frac{-2\ln(1-\tilde{\rho})}{\sqrt{1-\rho}}\right\rceil$ communications by Algorithm~\ref{algo:cheby} for the initial gradient tracking update in line 1 of Algorithm~\ref{algo:}. Then choose $\alpha$ such that
		\begin{equation}\label{corollary:diminishing_convergence:alpha}
			\alpha=\frac{1}{64L}.
		\end{equation}
	Then for all
		\begin{equation}\label{corollary:diminishing_convergence:transient_iterations}
					K\ge2^{\frac{3}{2}}k_0,
		\end{equation}
	Algorithm~\ref{algo:} produces a stochastic $\varepsilon$-stationary point as defined in Definition~\ref{def:stationarity} in
		\begin{equation}\label{corollary:diminishing_convergence:grad_and_comms}
			K=\bigO{\max\left\{\frac{\left(L\delta\right)^{\frac{3}{2}}+\hat{\sigma}^3+k_0^{\frac{1}{2}}\sigma^3}{\varepsilon^{\frac{3}{2}}},\frac{\hat{\sigma}^3\left(\abs{\ln\varepsilon}+\abs{\ln\hat{\sigma}}\right)^{\frac{3}{2}}}{\varepsilon^{\frac{3}{2}}}\right\}}
		\end{equation}
	local stochastic gradient computations and $T_0+ T(K-1)$ neighbor communications for any $T\ge1$.
\end{corollary}
\begin{proof}
	First, notice that $k_0=\lceil\frac{2}{(1-\tilde{\rho})^6}\rceil\ge\lceil\frac{2}{(1-\tilde{\rho}^3}\rceil$ by $(1-\tilde{\rho})^6\le(1-\tilde{\rho})^3\le1-\tilde{\rho}^3$ for all $\tilde{\rho}\in[0,1).$ Hence $k_0$ satisfies the requirements of Theorem~\ref{theorem:diminishing_convergence}. Next, we have $c\triangleq\frac{k_0^{\frac{1}{3}}}{2k_0^{\frac{1}{3}}+(k_0+1)^{\frac{1}{3}}}>\frac{1}{4}$ for all $k_0\ge2$. Hence it holds that
		\begin{align*}
		c\sum_{k=0}^{K-1}\alpha_k&=c\alpha\sum_{k=0}^{K-1}\frac{1}{(k+k_0)^{\frac{1}{3}}}\\
		&\ge c\alpha\int_{0}^K\frac{1}{(x+k_0)^{\frac{1}{3}}}dx\\
		&=\frac{3c\alpha}{2}\left((K+k_0)^{\frac{2}{3}}-k_0^{\frac{2}{3}}\right)\\
		&>\frac{3\alpha}{8}\left((K+k_0)^{\frac{2}{3}}-k_0^{\frac{2}{3}}\right)
		\end{align*}
	Notice $(K+k_0)^{\frac{2}{3}}-k_0^{\frac{2}{3}}>\frac{K^{\frac{2}{3}}}{2}$ as long as $K$ satisfies~\eqref{corollary:diminishing_convergence:transient_iterations}. Hence for some iterate $\tau\in\{0,1,\dots,K-1\}$ chosen with probability
		\begin{equation}\label{corollary:diminishing_convergence:probability}
			\prob{\tau=k}=\frac{\frac{c\alpha}{(k+k_0)^{\frac{1}{3}}}}{\sum_{j=0}^{K-1}\frac{c\alpha}{(j+k_0)^{\frac{1}{3}}}},\enskip k=0,\dots,K-1,
		\end{equation}
	 it holds that
		\begin{align*}
			&\frac{1}{N}\sum_{i=1}^N\Exp\norm{P\left(\vec{z}_i^{(\tau)},\nabla f(\vec{z}_i^{(\tau)}),\alpha_\tau\right)}_2^2+\frac{L^2}{N}\Exp\fronorm{\z_\perp^{(\tau)}}\\
			\le&\frac{64\hat{\Delta}}{\alpha K^{\frac{2}{3}}}+\frac{16}{3\alpha K^{\frac{2}{3}}}\sum_{k=0}^{K-1}\left(\frac{1}{L^2\alpha_{k+1}}+\frac{48}{L(1-\tilde{\rho})^2}\right)\beta_k^2\hat{\sigma}^2.
		\end{align*}
	Expanding the summation on the right yields
		\begin{align}\label{corollary:diminishing_convergence:upper_bound}
			&\sum_{k=0}^{K-1}\left(\frac{1}{L^2\alpha_{k+1}}+\frac{48}{L(1-\tilde{\rho})^2}\right)\beta_k^2\hat{\sigma}^2\nonumber\\
			=&\hat{\sigma}^2\sum_{k=0}^{K-1}\left(\frac{1}{L^2\alpha_{k+1}}+\frac{48}{L(1-\tilde{\rho})^2}\right)\left(1-\frac{\alpha_{k+1}}{\alpha_k}+48L^2\alpha_{k+1}^2\right)^2\nonumber\\
			\le&2\hat{\sigma}^2\sum_{k=0}^{K-1}\left(\frac{1}{L^2\alpha_{k+1}}+\frac{48}{L(1-\tilde{\rho})^2}\right)\left(\left(1-\frac{\alpha_{k+1}}{\alpha_k}\right)^2+\left(48L^2\alpha_{k+1}^2\right)^2\right)
		\end{align}
	where the last inequality uses $(a+b)^2\le2a^2+2b^2$ for any $a,b\in\R.$ Utilizing $a^3-b^3=(a-b)(a^2+ab+b^2)$ for any $a,b\in\R$, it holds that
		\begin{equation}\label{corollary:diminishing_convergence:alpha_k_relation}
			1-\frac{\alpha_{k+1}}{\alpha_k}=1-\frac{(k+k_0)^{\frac{1}{3}}}{(k+k_0+1)^{\frac{1}{3}}}=\frac{(k+k_0+1)^{-\frac{1}{3}}}{(k+k_0+1)^{\frac{2}{3}}+(k+k_0+1)^{\frac{1}{3}}(k+k_0)^{\frac{1}{3}}+(k+k_0)^{\frac{2}{3}}}.
		\end{equation}
	Notice that for all $k\ge1$ and $k_0\ge2$, it holds that
		\begin{align*}
		2(k+k_0+1)^{\frac{2}{3}}\le(k+k_0+1)^{\frac{2}{3}}+(k+k_0+1)^{\frac{1}{3}}(k+k_0)^{\frac{1}{3}}+(k+k_0)^{\frac{2}{3}}.
		\end{align*}
	Squaring both sides of the above inequality and rearranging results in
		\begin{equation}\label{corollary:diminishing_convergence:eqn1}
			\frac{1}{\left((k+k_0+1)^{\frac{2}{3}}+(k+k_0+1)^{\frac{1}{3}}(k+k_0)^{\frac{1}{3}}+(k+k_0)^{\frac{2}{3}}\right)^2}\le\frac{1}{4(k+k_0+1)^{\frac{4}{3}}}
		\end{equation}
	Utilizing~\eqref{corollary:diminishing_convergence:alpha_k_relation} and multiplying both sides of~\eqref{corollary:diminishing_convergence:eqn1} by $(k+k_0+1)^{-\frac{1}{3}}$, we further bound
		\begin{align}
			&2\hat{\sigma}^2\sum_{k=0}^{K-1}\frac{1}{L^2\alpha_{k+1}}\left(1-\frac{\alpha_{k+1}}{\alpha_k}\right)^2\nonumber\\
			=&\frac{2\hat{\sigma}^2}{L^2\alpha}\sum_{k=0}^{K-1}\frac{(k+k_0+1)^{-\frac{1}{3}}}{\left((k+k_0+1)^{\frac{2}{3}}+(k+k_0+1)^{\frac{1}{3}}(k+k_0)^{\frac{1}{3}}+(k+k_0)^{\frac{2}{3}}\right)^2}\nonumber\\
			\le&\frac{\hat{\sigma}^2}{2L^2\alpha }\sum_{k=0}^{K-1}\frac{1}{(k+k_0+1)^{\frac{5}{3}}}\nonumber\\
			\le&\frac{\hat{\sigma}^2}{2L^2\alpha}\int_{-1}^{K-1}\frac{1}{(x+k_0+1)^{\frac{5}{3}}}dx\nonumber\\
			\le&\frac{\hat{\sigma}^2}{2L^2\alpha}\cdot\frac{3}{2k_0^{\frac{2}{3}}}\nonumber\\
			=&\frac{3\hat{\sigma}^2}{4L^2\alpha k_0^{\frac{2}{3}}}\label{corollary:diminishing_convergence:bound1}
		\end{align}
	Again, utilizing~\eqref{corollary:diminishing_convergence:alpha_k_relation} and multiplying both sides of~\eqref{corollary:diminishing_convergence:eqn1} by $(k+k_0+1)^{-\frac{2}{3}}$, we have
		\begin{align}
			2\hat{\sigma}^2\sum_{k=0}^{K-1}\frac{48}{L(1-\tilde{\rho})^2}\left(1-\frac{\alpha_{k+1}}{\alpha_k}\right)^2&\le\frac{24\hat{\sigma}^2}{L(1-\tilde{\rho})^2}\sum_{k=0}^{K-1}\frac{1}{(k+k_0+1)^2}\nonumber\\
			&\le\frac{24\hat{\sigma}^2}{L(1-\tilde{\rho})^2k_0},\label{corollary:diminishing_convergence:bound2}
		\end{align}
	where we have also upper bounded the summation by the corresponding integral, since both $(x+k_0+1)^{-2}$ and $(x+k_0+1)^{-\frac{5}{3}}$ are decreasing for all $x>0$. Next, we bound
		\begin{align}
			2\hat{\sigma}^2\sum_{k=0}^{K-1}\frac{1}{L^2\alpha_{k+1}}\left(48L^2\alpha_{k+1}^2\right)^2=&4608\hat{\sigma}^2L^2\alpha^3\sum_{k=0}^{K-1}\frac{1}{k+k_0+1}\nonumber\\
			\le&4608\hat{\sigma}^2L^2\alpha^3\left(\ln(K+k_0)-\ln(k_0)\right)\nonumber\\
			\le&4608\hat{\sigma}^2L^2\alpha^3\ln(K+k_0),\label{corollary:diminishing_convergence:bound3}
		\end{align}
	and for $b=2\cdot48^3$ we have,
		\begin{align}
			2\hat{\sigma}^2\sum_{k=0}^{K-1}\frac{48}{L(1-\tilde{\rho})^2}\left(48L^2\alpha_{k+1}^2\right)^2&=\frac{b\hat{\sigma}^2L^3\alpha^4}{(1-\tilde{\rho})^2}\sum_{k=0}^{K-1}\frac{1}{(k+k_0+1)^{\frac{4}{3}}}\nonumber\\
			&\le\frac{3b\hat{\sigma}^2L^3\alpha^4}{(1-\tilde{\rho})^2k_0^{\frac{1}{3}}}.\label{corollary:diminishing_convergence:bound4}
		\end{align}
	Plugging~\eqref{corollary:diminishing_convergence:bound1} -~\eqref{corollary:diminishing_convergence:bound4} into~\eqref{corollary:diminishing_convergence:upper_bound} results in an inequality of the form
		\begin{align}
			&\frac{1}{N}\sum_{i=1}^N\Exp\norm{P\left(\vec{z}_i^{(\tau)},\nabla f(\vec{z}_i^{(\tau)}),\alpha_\tau\right)}_2^2+\frac{L^2}{N}\Exp\fronorm{\z_\perp^{(\tau)}}\nonumber\\
			\le&\frac{64\hat{\Delta}}{\alpha K^{\frac{2}{3}}}+\frac{16}{3\alpha K^{\frac{2}{3}}}\left(\frac{3\hat{\sigma}^2}{4L^2\alpha k_0^{\frac{2}{3}}}+\frac{24\hat{\sigma}^2}{L(1-\tilde{\rho})^2k_0}\right)\nonumber\\
			&+\frac{16}{3\alpha K^{\frac{2}{3}}}\left(4608\hat{\sigma}^2L^2\alpha^3\ln(K+k_0)+\frac{3b\hat{\sigma}^2L^3\alpha^4}{(1-\tilde{\rho})^2k_0^{\frac{1}{3}}}\right).\label{corollary:diminishing_convergence:relation}
		\end{align}
	Next, by~\eqref{lemma:base_change_diminishing:lyapunov} and~\eqref{theorem:diminishing_convergence:params}, we have
		\begin{equation}\label{corollary:diminishing_complexity:lyapunov}
			\hat{\Phi}^{(0)}\triangleq\phi(\bar{\blx}^{(0)})+\frac{1}{NL(1-\tilde{\rho})}\fronorm{\y_\perp^{(0)}}+\frac{16k_0^{\frac{1}{3}}}{N(1-\tilde{\rho})\alpha}\fronorm{\x_\perp^{(0)}}+\frac{k_0^{\frac{1}{3}}}{24NL^2\alpha}\fronorm{\vec{R}^{(0)}},
		\end{equation}
	where we have defined $\alpha_{-1}\triangleq\alpha_0$ for the $\gamma_2^{(-1)}$ term in~\eqref{lemma:base_change_diminishing:lyapunov}. Similar to the proof of Corollary~\ref{corollary:constant_convergence}, we bound each of the terms on the right-hand side of~\eqref{corollary:diminishing_complexity:lyapunov} by the initialization from Algorithm~\ref{algo:}. We have
		\begin{align}\label{corollary:diminishing_complexity:initial_batch}
			\frac{k_0^{\frac{1}{3}}}{24NL^2\alpha}\fronorm{\vec{R}^{(0)}}\le\frac{k_0^{\frac{1}{3}}\sigma^2}{24L^2\alpha m_0}
		\end{align}
	by~\eqref{assumption:variance} since $\vec{d}_i^{(0)}=\frac{1}{m_0}\sum_{\xi\in B_i^{(0)}}\nabla f_i(\blx_i^{(0)},\xi)$ for all $i=1,\dots,N$. Notice that $\y^{(0)}=\w_{T_0}(\d^{(0)})$. Hence~\eqref{corollary:constant_convergence:initial_y} still holds, so by $T_0=\left\lceil\frac{-2\ln(1-\tilde{\rho})}{\sqrt{1-\rho}}\right\rceil$, we have~\eqref{corollary:constant_convergence:initial_y_perp}. Thus, by $\blx_i^{(0)}=\blx_j^{(0)}$ for all $i,j=1,\dots,N$, we have that $\hat{\Delta}\le\phi(\bar{\blx}^{(0)})+4\fronorm{\d^{(0)}-\bar{\d}^{(0)}}+\frac{k_0^{\frac{1}{3}}\sigma^2}{24L^2\alpha m_0}-\phi^*$.
	By $k_0=\lceil\frac{2}{(1-\tilde{\rho})^6}\rceil$, the $\alpha$ in~\eqref{corollary:diminishing_convergence:alpha} satisfies $\alpha\le\frac{(1-\tilde{\rho})^2k_0^{\frac{1}{3}}}{64L}$. Hence,
		\begin{equation}\label{corollary:diminishing_convergence:initial_gap}
			\frac{\hat{\Delta}}{\alpha}\le256L\fronorm{\d^{(0)}-\bar{\d}^{(0)}}+\frac{172k_0^{\frac{1}{3}}\sigma^2}{m_0}+64L\left(\phi(\bar{\blx}^{(0)})-\phi^*\right).
		\end{equation}
	 Additionally we further bound the two terms on the right-hand side of~\eqref{corollary:diminishing_convergence:relation} that contain $(1-\tilde{\rho})^{-2}$. By the choice of $k_0$, we have $k_0\ge\frac{1}{(1-\tilde{\rho})^6}$, thus it holds that
	 	\begin{align*}
	 		\frac{1}{(1-\tilde{\rho})^2k_0^{\frac{1}{3}}}\le1 \text{ and } \frac{1}{(1-\tilde{\rho})^2k_0}\le(1-\tilde{\rho})^4\le1.
	 	\end{align*}
	 Hence, by recalling $\alpha=\frac{1}{64L}$ and $b=2\cdot48^3$, we have
		\begin{equation}\label{corollary:diminishing_convergence:rho_eqn1}
			\frac{24\hat{\sigma}^2}{\alpha L(1-\tilde{\rho})^2k_0}\le1536\hat{\sigma}^2,
		\end{equation}
	and
		\begin{equation}\label{corollary:diminishing_convergence:rho_eqn2}
			\frac{3b\hat{\sigma}^2L^3\alpha^3}{(1-\tilde{\rho})^2k_0^{\frac{1}{3}}}\le 3\hat{\sigma}^2.
		\end{equation}
	Further, it holds that $\frac{3\hat{\sigma}^2}{4L^2\alpha^2k_0^{\frac{2}{3}}}\le3072\hat{\sigma}^2$; using this and plugging~\eqref{corollary:diminishing_convergence:initial_gap},~\eqref{corollary:diminishing_convergence:rho_eqn1}, and~\eqref{corollary:diminishing_convergence:rho_eqn2} into~\eqref{corollary:diminishing_convergence:relation} yields
		\begin{align}
			&\frac{1}{N}\sum_{i=1}^N\Exp\norm{P\left(\vec{z}_i^{(\tau)},\nabla f(\vec{z}_i^{(\tau)}),\alpha_\tau\right)}_2^2+\frac{L^2}{N}\Exp\fronorm{\z_\perp^{(\tau)}}\nonumber\\
			\le&\frac{64}{K^{\frac{2}{3}}}\left(256L\fronorm{\d^{(0)}-\bar{\d}^{(0)}}+\frac{172k_0^{\frac{1}{3}}\sigma^2}{m_0}+64L\left(\phi(\bar{\blx}^{(0)})-\phi^*\right)\right)\nonumber\\
			&+\frac{16}{3K^{\frac{2}{3}}}\left(3072\hat{\sigma}^2+3\hat{\sigma}^2+1536\hat{\sigma}^2\right)+\frac{6\hat{\sigma}^2\ln\left(K+k_0\right)}{K^{\frac{2}{3}}}\nonumber\\
			\le&\frac{4096L\left(4\fronorm{\d^{(0)}-\bar{\d}^{(0)}}+\phi(\bar{\blx}^{(0)})-\phi^*\right)}{K^{\frac{2}{3}}}\nonumber\\
			&+\frac{\hat{\sigma}^2}{K^{\frac{2}{3}}}\left(24592+32\ln\left(K+k_0\right)+\frac{11008k_0^{\frac{1}{3}}\sigma^2}{m_0\hat{\sigma}^2}\right).\label{corollary:diminishing_convergence:relation_updated}
		\end{align}
	Finally, for $\tau$ chosen according to~\eqref{corollary:diminishing_convergence:probability}, we have
		\begin{align*}
							\frac{1}{N}\sum_{i=1}^N\Exp\norm{P\left(\vec{z}_i^{(\tau)},\nabla f(\vec{z}_i^{(\tau)}),\alpha_\tau\right)}_2^2+\frac{L^2}{N}\Exp\fronorm{\z_\perp^{(\tau)}}\le\varepsilon,
		\end{align*}
	provided
		\begin{equation}\label{corollary:diminishing_convergence:complexity}
			K=\bigO{\max\left\{\frac{\left(L\delta\right)^{\frac{3}{2}}+\hat{\sigma}^3}{\varepsilon^{\frac{3}{2}}},\frac{k_0^{\frac{1}{2}}\sigma^3}{m_0^{\frac{3}{2}}\varepsilon^{\frac{3}{2}}},\frac{\hat{\sigma}^3\left(\abs{\ln\varepsilon}+\abs{\ln\hat{\sigma}}\right)^{\frac{3}{2}}}{\varepsilon^{\frac{3}{2}}}\right\}}
		\end{equation}
	where $\delta\triangleq\fronorm{\d^{(0)}-\bar{\d}^{(0)}}+\phi(\bar{\blx}^{(0)})-\phi^*$. Choosing the initial batch size $m_0=m=\bigO{1}$ yields the total number of gradient computations $mK$ in~\eqref{corollary:diminishing_convergence:grad_and_comms}, provided $K$ satisfies~\eqref{corollary:diminishing_convergence:transient_iterations}.
\end{proof}

\begin{remark}\label{remark:complexity_discussion_2}
	Similar to the discussion provided in Remark~\ref{remark:complexity_discussion}, we note that Chebyshev acceleration can be utilized to perform the neighbor communications (lines 3 and 6 in Algorithm~\ref{algo:}). Since $k_0^{\frac{1}{2}}=\bigO{(1-\tilde{\rho})^{-3}}$, this number can dominate in~\eqref{corollary:diminishing_convergence:grad_and_comms}, indicating that the sample complexity result is \emph{network-dependent}. In order to have the complexity result as indicated in Table~\ref{table:related_works}, we perform $T=\lceil\frac{2}{\sqrt{1-\rho}}\rceil$ Chebyshev communications rounds by Algorithm~\ref{algo:cheby} so that $\tilde{\rho}\le\frac{1}{2}$ by~\eqref{appendix:cheby:upper_bound} and hence the sample complexity cost is independent of $\rho$ and $\tilde{\rho}.$ In this regime, the number of local neighbor communications is
		\begin{align*}
			T_0+TK=\bigO{\frac{1}{\sqrt{1-\rho}}\max\left\{\frac{\left(L\delta\right)^{\frac{3}{2}}+\hat{\sigma}^3+\sigma^3}{\varepsilon^{\frac{3}{2}}},\frac{\hat{\sigma}^3\left(\abs{\ln\varepsilon}+\abs{\ln\hat{\sigma}}\right)^{\frac{3}{2}}}{\varepsilon^{\frac{3}{2}}}\right\}}
		\end{align*}
	which is optimal in terms of dependence upon $\rho$~\cite{scaman17}. Alternatively, if we let the initial batch size be $m_0=\bigO{\frac{k_0^{\frac{1}{3}}}{\sigma^2}}$ independent of $\varepsilon$, then the middle term in~\eqref{corollary:diminishing_convergence:complexity} can be dominated by the first term, in which case the sample complexity is network-independent, after the initial iteration. For cases where the original communication network is not too sparse, e.g. $\rho$ is not too close to 1, this may be preferred over performing Chebyshev acceleration.
\end{remark}

\end{document}